\documentclass[11pt]{article}
\usepackage{mathtools}
\usepackage{amsthm, amssymb}
\usepackage{setspace, bigints}
\usepackage{mathrsfs}
\usepackage{mleftright}
\usepackage{bm}
\usepackage{bbm}
\usepackage[margin=0.75 in]{geometry}
\usepackage{caption}
\usepackage{subcaption}
\usepackage[toc,page]{appendix}     
\usepackage{pdfpages}
\usepackage{epstopdf}
\usepackage{booktabs}
\usepackage{graphicx}
\usepackage{listings}
\usepackage{xcolor}
\usepackage{tabularx, multirow}
\usepackage[ruled,linesnumbered]{algorithm2e}
\usepackage{tabu}
\usepackage{comment}
\usepackage[sort]{natbib}
\usepackage{authblk}

\usepackage{amsthm}

\newtheoremstyle{plainupright}
  {\topsep}   
  {\topsep}   
  {}          
  {}          
  {\bfseries} 
  {.}         
  {.5em}      
  {}          

\theoremstyle{plainupright}
\newtheorem{assumption}{Assumption}

\usepackage{empheq}
\usepackage{accents}
\usepackage{bm}
\newcommand*{\B}[1]{\ifmmode\bm{#1}\else\textbf{#1}\fi}

\newcommand{\trieq}[1]{\stackrel{\triangle_{#1}}{=}}

\newcommand{\be} {\begin{eqnarray*}}
\newcommand{\ee} {\end{eqnarray*}}

\DeclareMathOperator{\W}{W}

\DeclareMathOperator{\tr}{Tr}
\DeclareMathOperator{\lip}{Lip}
\DeclareMathOperator{\cof}{cof}

\DeclareMathOperator{\eot}{EOT}
\DeclareMathOperator{\id}{Id}
\DeclareMathOperator{\KL}{KL}
\DeclareMathOperator{\Law}{Law}

\DeclareMathOperator{\vol}{Vol}
\DeclareMathOperator{\osc}{Osc}
\DeclareMathOperator{\lhs}{LHS}
\DeclareMathOperator{\rhs}{RHS}
\DeclareMathOperator{\schro}{ScB}
\DeclareMathOperator{\poly}{Poly}
\DeclareMathOperator{\FI}{FI}

\newcommand{\argmin}{\mathop{\rm argmin~}}
\newcommand{\argmax}{\mathop{\rm argmax~}}

\newcommand{\dd}{{\rm d}}

\newcommand{\wht}{\widehat}

\newcommand{\wt}{\widetilde}

\newcommand{\rytodo}[1]{\textcolor{red}{#1}}
\newcommand{\cdn}[1]{\textcolor{brown}{#1}}
\newcommand{\td}{{\rm dist}_{\mb T^d}}

\newcommand{\myqed}{\hfill\ensuremath{\diamondsuit}}

\def\m{\mathcal}
\def\ms{\mathscr}
\def\mb{\mathbb}

\def\mx{\mbox}

\newcommand{\matnorm}[1]{{\left\vert\kern-0.25ex\left\vert\kern-0.25ex\left\vert #1 
    \right\vert\kern-0.25ex\right\vert\kern-0.25ex\right\vert}}

\newcommand{\ri}{\textrm{(i)}}
\newcommand{\rii}{\textrm{(ii)}}
\newcommand{\riii}{\textrm{(iii)}}
\newcommand{\Lip}{{\rm{Lip}}}

\newtheorem{theorem}{Theorem}[section]
\newtheorem{lemma}[theorem]{Lemma}
\newtheorem{defn}[theorem]{Definition}
\newtheorem{proposition}[theorem]{Proposition}
\newtheorem{corollary}[theorem]{Corollary}
\newtheorem{rem}[theorem]{Remark}

\numberwithin{equation}{section}

\usepackage{hyperref}
\usepackage[capitalize]{cleveref}

\usepackage{xspace}

\newcommand\yh{\color{purple}}
\title{Quantitative Stability of Many-Marginal Schr\"odinger Bridge}

\author{Rentian Yao}
\author{Young-Heon Kim}
\author{Geoffrey Schiebinger}
\affil{Department of Mathematics, University of British Columbia \authorcr Email: \{rentian2, yhkim, geoff\}@math.ubc.ca}
\date{\vspace{-2em}}
\date{}

\begin{document}

\maketitle



\begin{abstract}
In this paper, we explore quantitative stability of multi-marginal Schr\"odinger bridges with respect to the marginal constraints.
We focus on the case where the number of marginal constraints is large (i.e. ``many-marginals"). When this number increases, 
we show that the Kullback--Leibler (KL) divergence between two multi-marginal Schr\"odinger bridges, as measures on the path space, can be asymptotically bounded by the terminal marginal KL divergence and a time-integrated squared discrepancy {that combines} Wasserstein-2 geodesic velocity fields with a log-density gradient term.
Our stability upper bound is also asymptotically tight: it converges to zero as the number of marginal constraints increases with unperturbed marginal constraints. To the best of our knowledge, this is the first such stability result that addresses the many-marginal regime, giving error estimates that are asymptotically independent of the number of marginals. 

To achieve our result, the key step is to derive an asymptotic  expansion  (of order $k\ge 2$) of Schr\"odinger potentials with respect to a diminishing regularization coefficient. This result can also be applied to deriving asymptotic expansions of entropic Brenier maps in entropic optimal self-transport problems. As byproducts of our analyses, we also establish the asymptotic expansion of entropic optimal transport cost with respect to the diminishing regularization coefficient when two marginal constraints are sufficiently close. We also prove  a stability property of the Schr\"odinger functional.
\end{abstract}

\section{Introduction and Main Results}\label{sec: intro+main}


The notion of Schr\"odinger bridges (SB)~\citep{Schrodinger1932theorie, leonard2014survey, chen2021stochastic} has been introduced as a probabilistic concept for inferring the trajectories of random particles, see e.g., ~\citep{lavenant2024toward,yao2025learning}. Namely, given the distributions of particles at different time points, which are induced by a certain stochastic dynamics, how can we infer the trajectories of those particles in a systematic way? A Schr\"odinger bridge is a probability measure on the path-space that describes such trajectories and is a solution to minimizing the Kullback--Leibler (KL) divergence with respect to the Wiener measure, the probability distribution of the Brownian motion. This problem has seen fascinating recent developments, especially due to its profound applications~\citep{lavenant2024toward, de2021diffusion, cuturi2013sinkhorn} and its connection to entropic optimal transport (EOT)~\citep{mikami2004monge, mikami2002optimal, leonard2014survey}.
A particular interest is in multi-marginal Schr\"odinger bridges where many time points and corresponding marginal distributions are given. This is relevant, for example, when one considers trajectory inference problems in mathematical biology, where one wants to understand developmental procedures while only knowing the data samples at each given time point, without observing how they move~\citep{schiebinger2019optimal, lavenant2024toward, chizat2022trajectory, yao2025learning}. With the possible applications to such statistical inference problems, it is important to understand how stable the Schr\"odinger bridges are under the perturbations of the given marginal data at those given time points.

In this paper, we explore quantitative stability of multi-marginal Schr\"odinger bridges, where marginals come from curves in the space of probability measures (as illustrated in Figure~\ref{fig: illstration}). 
This setting naturally emerges from trajectory inference problems in mathematical biology, where the density evolution curve represents the evolution of gene expression levels in cell populations~\cite{lavenant2024toward,chizat2022trajectory, schiebinger2021reconstructing}.
We focus on the setting where the number of marginal constraints may grow to infinity within a finite time horizon. This limit is essential for consistently recover the ground-truth in the limit of infinitely many time-points~\cite{lavenant2024toward}. 


Here, we briefly summarize our main result, leaving its rigorous statement in Theorem~\ref{thm: multiSB_stable}.




\paragraph{Main Theorem (informal)}
Given two density evolution curves $(\rho_t^\mu)_{t\in[0, 1]}$ and $(\rho_t^\nu)_{t\in[0, 1]}$, as well as two sequences of marginal constraints $\bm\mu^m = (\rho^\mu_{t_0}, \cdots, \rho_{t_m}^\mu)$ and $\bm\nu^m = (\rho_{t_0}^\nu, \cdots, \rho_{t_m}^\nu)$ collected at time points $0 = t_0 < t_1 < \cdots < t_m = 1$, the associated Schr\"odinger bridges  $R^{\bm\mu^m}$ and $R^{\bm\nu^m}$ satisfy 
\begin{align*}
\KL(R^{\bm\nu^m}\,\|\,R^{\bm\mu^m})
\lesssim
\left(
\begin{array}{c}
\textnormal{quantities depending only on the marginals } \bm\nu^m \\
\textnormal{and the geodesic interpolations connecting } \bm\mu^m
\end{array}
\right)
+ O\Big(\frac{1}{m}\Big)
\end{align*}
under appropriate assumptions. Here, $O(1/m)$ represents a term converging to zero as $m\to\infty$, and the inequality $\lesssim$ is independent of $m$ so that the estimate is valid as $m\to \infty$.
Our estimate is asymptotically tight in the sense that the first term on the right-hand side also converges to zero as $m\to\infty$, when the two sequences of marginal constraints are identical.

\paragraph{Comparison to prior work}
Using the Markov property of multi-marginal Schr\"odinger bridges (see Section~\ref{sec: multi-margin SB} for more details), one can reduce computing multi-marginal Schr\"odinger bridges to solving $m$ entropically-regularized optimal transport (EOT) problems. These EOT problems have the regularization coefficient $\varepsilon = \frac{1}{m}$ when the time points are equally spaced in $[0, 1]$. 
Therefore, the stability of multi-marginal Schr\"odinger bridge problems with an increasing number of marginal constraints, is  naturally related to EOT problems with a diminishing regularization coefficient  $\varepsilon \to 0$.
This is the regime in which existing approaches 
fail to analyze the stability of EOT or SB.
Existing works either allow perturbation of only one marginal constraint~\citep{chiarini2024semiconcavity, delalande2022nearly, kitagawa2025stability} or consider only the regime where $\varepsilon$ is fixed or bounded from below, directly applying $L^\infty$ estimates when analyzing EOT problems~\citep{carlier2024displacement, rigollet2025sample, ghosal2022stability, eckstein2022quantitative}, leading to a stability bound of order $O\big(\poly(\varepsilon)e^{1/\varepsilon}\big)$.
They ignore that consecutive marginal constraints become closer as the regularization coefficient decreases to zero, which we utilize in the present paper.


To clearly explain the intuition behind establishing our stability result, let $\phi_j^\mu$ denote the Schr\"odinger potential for solving the EOT problem with marginals $\rho_{t_{j-1}}^\mu, \rho_{t_j}^\mu$ and regularization coefficient $\varepsilon = \frac{1}{m}$; let $\phi_j^\nu$ be defined similarly. In the small $\varepsilon$ regime, the Schr\"odinger potential $\phi_j^\mu$ is approximately the Kantorovich potential between $\rho_{t_{j-1}}^\mu$ and $\rho_{t_j}^\mu$ with a correction term due to entropic regularization. Both terms vanish as $\varepsilon\to 0$, indicating that $\phi_j^\mu\to 0$. A similar argument applies to $\phi_j^\nu$, so the difference $\phi_j^\mu - \phi_j^\nu$ remains as a small order term $o_m(1)$ even when the marginal constraints change significantly. Based on this intuition, our approach performs a higher-order analysis of the Schr\"odinger potentials, which 
takes care of the effects of both marginal constraints and the regularization coefficient.
Figure~\ref{fig: illstration} below also illustrates the key idea of our analysis.

\begin{figure}[h]
\centering
\includegraphics[width=0.5\linewidth]{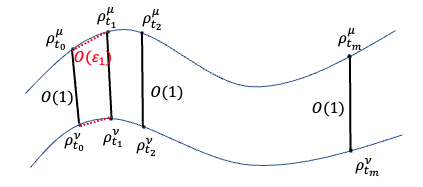}
\caption{\textbf{Why existing stability analyses for multi-marginal Schr\"odinger bridges fail as the EOT regularization coefficient $\varepsilon  = m^{-1}\to 0^+$}: 
Existing results bound the difference of Schr\"odinger potentials along two curves $\rho_t^\mu$ and $\rho_t^\nu$ using the distances between their marginal constraints $\rho_{t_j}^\mu$ and $\rho_{t_j}^\nu$, which are typically of order $O(1)$  with a poor scaling $O(\poly(\varepsilon)e^{1/\varepsilon})$ with respect to $\varepsilon$. As $\varepsilon = 1/m \to 0^+$, the total discrepancy between two multi-marginal Schr\"odinger bridges diverges. 
In contrast, we observe that because adjacent marginals $\rho_{t_{j-1}}^\mu$ and $\rho_{t_j}^\mu$ are close, the corresponding Schr\"odinger potential is of order at most $O(\varepsilon_j)$. Exploiting this observation and summing these over $j$ can lead to a stability bound of order $O(1)$. (See Theorem~\ref{thm: multiSB_stable}).}
\label{fig: illstration}
\end{figure}



\paragraph{Organization of the paper}
The paper is structured as follows: The remainder of the Section~\ref{sec: intro+main} introduces the background on Schr\"odinger bridges and entropic optimal transport, as well as the main results of this paper. Section~\ref{sec: EOT expansion} derives the asymptotic expansion of entropic optimal transport cost with diminishing regularization coefficient using Fourier series. In Section~\ref{sec: stability}, we establish a stability result for the dual functional of entropic optimal transport. We prove the main results of this paper in Section~\ref{sec: main pf}, while deferring the proofs of technical lemmas to the appendices.

\subsection{Preliminaries}\label{sec: prelim}
Let us assume that the ambient space is the flat torus $\mb T^d = [0, 2\pi)^d = \mb R^d/2\pi\mb Z^d$, equipped with the metric
\begin{align*}
\td(x, y) \coloneqq \min_{k\in\mb Z^d} \|x - y - 2\pi k\|,
\end{align*}
where $\|\cdot\|$ represents the usual $\ell^2$-norm on the Euclidean space $\mb R^d$. Let $\ms P(\mb T^d)$ denote the set of all probability distributions over $\mb T^d$, and by $\ms P_{ac}(\mb T^d)$ the subset of $\ms P(\mb T^d)$ consisting of distributions that are absolutely continuous with respect to the Lebesgue measure on $\mb T^d$. Let $\m K_\varepsilon$ be the Gaussian kernel on $\mb T^d$, defined by
\begin{align}\label{eqn: Gaussian kernel}
\m K_\varepsilon(x) \coloneqq \frac{1}{\sqrt{(2\pi\varepsilon)^d}}\sum_{k\in\mb Z^d}e^{-\frac{\|x-2\pi k\|^2}{2\varepsilon}},\quad x\in\mb T^d.
\end{align}
Then, it is known that $\m K_\varepsilon$ solves the heat equation $\partial_t\m K_t(x) = \frac{1}{2}\Delta\m K_t(x)$ for $t>0$, where $\Delta$ is the Laplace--Beltrami operator on the flat torus.
Set the associated cost function $$c_\varepsilon(x, y) \coloneqq -\varepsilon\log\m K_\varepsilon(x-y).$$
Let $\Omega = \m C([0, 1]; \mb T^d)$ denote the path space of all continuous functions mapping from the time interval $[0, 1]$ to the metric space $\mb T^d$.

\subsubsection{Optimal transport}
We briefly review some basic concepts of optimal transport theory here, while for more details, we refer the reader to standard references, including \cite[Section 1.3.2]{santambrogio2015optimal} especially for optimal transport on the flat torus $\mb T^d$ with quadratic cost.

Given two probability distributions $\mu, \nu\in\ms P(\mb T^d)$, the squared Wasserstein-2 ($W_2$)  distance between $\mu$ and $\nu$ is defined by
\begin{align}\label{eqn: w2-def}
\W_2^2(\mu, \nu) \coloneqq \min_{\pi\in\Pi(\mu, \nu)} \int\!\!\!\int_{\mb T^d\times\mb T^d} \frac{1}{2}\td(x, y)^2\,\dd\pi(x, y),
\end{align}
where $\Pi(\mu, \nu)$ denotes the set of all joint distributions over $\mb T^d \times \mb T^d$ with marginals $\mu$ and $\nu$.  The optimal solution $\pi_0\in\Pi(\mu, \nu)$ to the above minimization problem, also known as {\em the optimal transport plan}, exists. Moreover, when $\mu\in\ms P_{ac}(\mb T^d)$ admits a density, it is uniquely determined by the dual problem of~\eqref{eqn: w2-def}. Specifically, the dual formulation is
\begin{align}
\begin{aligned}
\W_2^2(\mu, \nu) = \max_{\phi\in L^1(\mu), \psi\in L^1(\nu)} &\int_{\mb T^d}\phi\,\dd\mu + \int_{\mb T^d}\psi\,\dd\nu,\\
\textrm{s.t.}\quad\quad & \phi(x) + \psi(y) \leq \frac{1}{2}\td(x, y)^2.
\end{aligned}
\end{align}
The optimal dual solution $(\phi_0, \psi_0)$ to the above maximization problem is called {\em the Kantorovich potentials}. They are Lipschitz functions on $\mathbb{T}^d$.
The Kantorovich potentials induce an optimal transport map $T_0$ defined a.e. by $T_0(x) \coloneqq x - \nabla\phi_0(x)$, which pushes $\mu$ to $\nu$, i.e.
\begin{align*}
\nu = (T_0)_\#\mu,
\qquad\mx{where}\quad
[(T_0)_\#\mu](A) \coloneqq \mu(T_0^{-1}(A))
\end{align*}
for any measurable set $A\subset \mb T^d$. The optimal transport plan, which is the primal solution, is then given by
\begin{align*}
\pi_0 = (\id, T_0)_\#\mu,
\end{align*}
i.e. $\pi_0(A\times B) = \mu(A\cap T_0^{-1}(B))$ holds for any measurable sets $A, B\subset \mb T^d$.

Given the optimal transport map $T_0$, we can define the geodesics interpolation between $\mu$ and $\nu$ by
\begin{align*}
    \rho_t = [t T_0 + (1-t)\id]_\#\mu,
    \quad t\in[0, 1],
\end{align*}
which satisfies $\rho_0 = \mu$ and $\rho_1 = \nu$.
Geometrically, the interpolation $(\rho_t)_{0\leq t\leq 1}$ is the shortest path, i.e. the geodesic, connecting $\mu$ and $\nu$ under the $\W_2$ distance. By the Benamous--Brenier's Theorem, there exists a velocity vector field $v_t = \nabla\Phi_t$ for $0\leq t\leq 1$, such that the geodesics interpolation $(\rho_t)_{0\leq t\leq 1}$
satisfies the continuity equation
\begin{align*}
\partial_t\rho_t + \nabla\cdot(\rho_t\nabla\Phi_t) = 0,
\quad t\in[0,1],\, x\in\mb T^d.
\end{align*}

\subsubsection{Entropic optimal transport}
We recall some basic notions of entropic optimal transport (EOT).
Given a regularization coefficient $\varepsilon > 0$, the EOT has the cost between two probability distributions $\mu, \nu\in\ms P(\mb T^d)$ defined as
\begin{align}\label{eqn: general_EOT}
\eot_{\varepsilon}(\mu, \nu) = \min_{\pi\in\Pi(\mu, \nu)} \int\!\!\!\int_{\mb T^d\times\mb T^d} c_\varepsilon(x, y)\,\dd\pi(x, y) + \varepsilon\KL(\pi\,\|\,\mu\otimes\nu),
\end{align}
where $\KL(\pi\,\|\,\mu \otimes \nu)$ is the Kullback--Leibler (KL) divergence between $\pi$ and the product measure $\mu \otimes \nu$.
Analogous to the optimal transport problem, the minimization problem~\eqref{eqn: general_EOT} also admits a dual formulation, written as
\begin{align}\label{eqn: duality}
\begin{aligned}
\eot_{\varepsilon}(\mu, \nu) 
&= \max_{\phi\in L^1(\mu), \psi\in L^1(\nu)} \int_{\mb T^d}\phi\,\dd\mu 
+ \int_{\mb T^d}\psi\,\dd\nu
- \varepsilon\int\!\!\!\int_{\mb T^d\times\mb T^d} e^{\frac{\phi(x) + \psi(y) - c_\varepsilon(x, y)}{\varepsilon}}\,\dd\mu(x)\dd\nu(y) + \varepsilon.
\end{aligned}
\end{align}
Notice that the functional on the righthand side is strictly concave in $\phi$ and $\psi$; thus the dual solution is unique.
The dual solution $(\phi_{\varepsilon}, \psi_{\varepsilon})$ is a pair of functions called {\em Schr\"odinger potentials}. 
The Schr\"odinger potentials are Lipschitz continuous and satisfy the following Schr\"odinger system:
\begin{align}\label{eqn: Schro_sys}
\begin{aligned}
\phi_{\varepsilon}(x) &= \m T_{\nu}^{\varepsilon}[\psi_{\varepsilon}](x) \coloneqq -\varepsilon\log\bigg(\int_{\mb T^d}e^{\frac{\psi_{\varepsilon}(y) - c_\varepsilon(x, y)}{\varepsilon}}\,\dd\nu(y)\bigg),\\
\psi_{\varepsilon}(y) &= \m T_\mu^{\varepsilon}[\phi_{\varepsilon}](y) \coloneqq -\varepsilon\log\bigg(\int_{\mb T^d}e^{\frac{\phi_{\varepsilon}(x) - c_\varepsilon(x, y)}{\varepsilon}}\,\dd\mu(x)\bigg).
\end{aligned}
\end{align}
Notice that the functional $\phi \mapsto \m T^{\varepsilon}_\mu [\phi]$ is concave, as it is nothing but the negative LogSumExp function that is convex due to Jensen's inequality.
Using the notation $\m T^{\varepsilon}_\cdot$, we can define the Schr\"odinger functionals as follows:
\begin{subequations}\label{eqn: dual functional}
\begin{align}
\m I_{\mu,\nu}^\varepsilon[\phi] &\coloneqq \int_{\mb T^d} \phi\,\dd\mu + \int_{\mb T^d} \m T_\mu^\varepsilon[\phi]\,\dd\nu,\\
\overline{\m I}^\varepsilon_{\mu, \nu}[\psi] &\coloneqq \int_{\mb T^d} \m T_\nu^\varepsilon[\psi]\,\dd\mu + \int_{\mb T^d}\psi\,\dd\nu. 
\end{align}
\end{subequations}
These are concave functionals, but not strictly, as one can observe  $\m I_{\mu,\nu}^\varepsilon[\phi] = \m I_{\mu,\nu}^\varepsilon[\phi + C]$ for any constant $C$.
Now, the dual formulation~\eqref{eqn: duality} can also be expressed as
\begin{align}\label{eqn: strong_dual}
\eot_{\varepsilon}(\mu, \nu) 
= \max_{\phi\in L^1(\mu)} \m I_{\mu, \nu}^\varepsilon[\phi]
= \max_{\psi\in L^1(\nu)} \overline{\m I}_{\mu, \nu}^\varepsilon[\psi],
\end{align}
with the Schr\"odinger potentials being the optimal dual solution, which satisfies
\begin{align*}
\phi_\varepsilon = \argmax_{\phi\in L^1(\mu)}\m I_{\mu, \nu}^\varepsilon[\phi]
\quad\mx{and}\quad
\psi_\varepsilon = \argmax_{\psi\in L^1(\nu)}\overline{\m I}_{\mu, \nu}^\varepsilon[\psi]
\end{align*} 
Given the dual solution $(\phi_\varepsilon, \psi_\varepsilon)$, the optimal coupling $\pi_{\varepsilon}$ for solving the primal problem~\eqref{eqn: general_EOT} can be constructed via
\begin{align}\label{eqn: SB_primal_sol}
\frac{\dd \pi_{\varepsilon}}{\dd\mu\otimes\nu}(x, y) = e^{\frac{\phi_{\varepsilon}(x) + \psi_{\varepsilon}(y) - c_\varepsilon(x, y)}{\varepsilon}}.
\end{align}
In practice, the Schr\"odinger potentials can be efficiently computed using the celebrated Sinkhorn algorithm~\citep{sinkhorn1967diagonal, cuturi2013sinkhorn}. As a result, both the optimal coupling $\pi_{\varepsilon}$ and the EOT cost can be computed effectively.


\subsubsection{Schr\"odinger bridges}
The Schr\"odinger bridge (SB) problem seeks a probability distribution over the path space $\Omega$ that is closest to the Weiner measure while satisfying given marginal constraints. Specifically, let $W\in\ms P(\Omega)$ be the Wiener measure, which is the law of the standard Brownian motion on $\mb T^d$ over the time interval $[0, 1]$. 
Then, given a source distribution $\mu\in\ms P(\mb T^d)$ and a target distribution $\nu\in\ms P(\mb T^d)$, the SB problem is formulated as the entropy minimization problem 
\begin{align*}
R^{\mu,\nu} = \argmin_{R\in\ms P(\Omega): R_0 = \mu, R_1 = \nu} \KL(R\,\|\,W),
\end{align*}
where $R_t\in\ms P(\mb T^d)$ denotes the marginal distribution of $R\in\ms P(\Omega)$ at time $t\in[0, 1]$. That is, if a stochastic process  $X_{[0, 1]}$ follows the distribution $R$, then $X_t$ is distributed according to $R_t$.
Here, the optimal solution  $R^{\mu, \nu}$ is uniquely determined under natural assumptions, e.g. $\mu, \nu$ are absolutely continuous probability measures (see, e.g.,~\citep{nutz2021introduction} for a detailed proof).

The SB problem is closely realted to the EOT problem. In fact,~\cite{leonard2014survey} showed that the optimal solution $R^{\mu, \nu}$ satisfies
\begin{align}\label{eqn: SB-2-margin}
R^{\mu, \nu}(\dd\omega) = \int\!\!\!\int_{\mb T^d\times\mb T^d} W(\dd\omega\,|\,\omega_0 = x, \omega_1 = y) \pi_1(\dd x, \dd y)
\end{align}
for every continuous path $\omega\in\m C([0, 1]; \mb T^d)$. Here, the probability measure $\pi_1$ is the solution to the EOT problem~\eqref{eqn: general_EOT} with regularization coefficient $\varepsilon=1$. Equation~\eqref{eqn: SB-2-margin} shows that the Schr\"odinger bridge is nothing but the collection of Brownian bridges where the initial and final points are coupled via the entropic optimal transport plan.

\subsubsection{Multi-marginal Schr\"odinger bridges}\label{sec: multi-margin SB}
This framework naturally extends to the case of multi-marginal constraints, leading to the multi-marginal Schr\"odinger bridge problem. Let $0 = t_0 < t_1 < \cdots < t_m =1$ be $m+1$ time points, and let $\bm\mu^m = (\mu_0, \cdots, \mu_m)\in\ms P(\mb T^d)^{\otimes(m+1)}$ be a sequence of prescribed marginals. Then, the multi-marginal SB problem is given by:
\begin{align}\label{eqn:SB-multi}
R^{\bm \mu^m} \coloneqq \argmin_{R_{t_j} = \mu_j, j=0,1,\cdots, m} \KL(R\,\|\,W).
\end{align}
Here, the optimal solution  $R^{\bm \mu^m}$ is uniquely determined under natural assumptions, e.g. $\mu_j$ are absolutely continuous probability measures (e.g. see~\citep{nutz2021introduction} for a detailed proof).
Analogous to the result~\eqref{eqn: SB-2-margin} in the two-marginal case, the optimal solution $R^{\bm\mu^m}$ admits the representation:
\begin{align}\label{eqn: Markov}
R^{\bm\mu^m}(\dd\omega) = \int_{(\mb T^d)^{m+1}} W(\dd \omega\,|\,\omega_0 = x_0, \cdots, \omega_m = x_m)\,\pi^{\bm\mu^m}(\dd x_0, \cdots. \dd x_m).
\end{align}
Here $\pi^{\bm\mu^m}$ is a probability distribution over $(\mb T^d)^{m+1}$ which defines a Markov transition kernel over $\mb T^d$ by
\begin{align}\label{eqn: MSB_Markov}
\pi^{\bm\mu^m}(\dd x_0, \cdots, \dd x_m)
= \pi_{0,1}(\dd x_0, \dd x_1)\pi_{1,2}(x_1, \dd x_2)\cdots \pi_{m-1, m}(x_{m-1}, \dd x_m),
\end{align}
with $\pi_{j-1, j}$ being the primal solution for solving $\eot_{\varepsilon_j}(\mu_{j-1},\mu_j)$ with the regularization coefficient \begin{align}\label{eqn:e-j}\varepsilon_j = t_j - t_{j-1}.
\end{align}
The Markov structure in \eqref{eqn: Markov}  allows the multi-marginal SB problem to be decomposed into a sequence of two-marginal SB problems.

It is important to note that the regularization coefficient in \eqref{eqn:e-j} dimishes as the number of time points $m \to \infty.$

\subsection{Main results}
The main focus of this paper is the quantitative stability of the optimal solution  $R^{\bm\mu^m}$ of~\eqref{eqn:SB-multi} with respect to 
the marginal constraints $\bm{\mu}^m$, allowing the number of marginal constraints $m+1$ to grow to infinity. 
A main challenge is that in this case the regularization coefficients $\varepsilon_j$ given in~\eqref{eqn:e-j} diminishes.
Our analysis relies on the asymptotic expansion of the Schr\"odinger potentials and the EOT cost in~\eqref{eqn: general_EOT} as the regularization coefficient $\varepsilon \to 0^+$. These related results 
are of independent interest,
as the EOT problem with vanishing regularization has recently attracted growing attention due to its potential applications, e.g., in finite-sample analysis for trajectory inference~\citep{lavenant2024toward, yao2025learning} and score function approximation~\citep{agarwal2024iterated, mordant2024entropic}.

We start with the overarching regularity  assumption on the marginal distributions:
\begin{assumption}\label{assump: density_expansion}
Let $(\rho_t^\mu)_{t\in[0, 1]}$ be a density evolution on $\mb T^d$ and assume that  $\rho_t^\mu(x)\in\m C^\infty([0, 1]\times\mb T^d)$. Assume there exists a function $\Phi_t^\mu(x)\in\m C^\infty([0, 1]\times\mb T^d)$, satisfying the continuity equation 
\begin{align}\label{eqn: CE}
   \partial_t\rho_t^\mu + \nabla\cdot(\rho_t^\mu\nabla\Phi_t^\mu) = 0 .
\end{align} 
Therefore, there exists a sequence of  smooth functions $\rho_t^{(1)}, \rho_t^{(2)}, \cdots \in \m C^\infty([0, 1]\times\mb T^d$), such that $\rho_{t+\varepsilon}^\mu$ admits an expansion
\begin{align*}
\rho_{t+\varepsilon}^\mu(x) = \rho_t^\mu(x) + \varepsilon\rho_t^{(1)}(x) + \varepsilon^2\rho_t^{(2)}(x) + \cdots
\end{align*}
whenever $t, t+\varepsilon\in[0, 1]$.
\myqed
\end{assumption}

We emphasize that the above smoothness assumption is made for simplicity. It can be relaxed to $\rho_t^\mu(x)\in\m C^{k,\alpha}([0, 1]\times\mb T^d)$ for some $\alpha\in(0, 1)$ and $k = O(\max\{K, d\})$, if a $K$-th order Taylor expansion of the Schr\"odinger potentials in Theorem~\ref{thm: asymp_schro_potential} below is desired. 
Our main case corresponds to $K=3$, 
which leads to the following main theorem of this paper.

\begin{theorem}\label{thm: multiSB_stable}
Let $(\rho_t^\mu)_{t\in[0, 1]}$ and $(\rho_t^\nu)_{t\in[0, 1]}$ be two density evolution curves over $\mb T^d$. 
Given time points $0 = t_0 < t_1 < \cdots < t_m = 1$, take $\bm\mu^m = (\rho^\mu_{t_0}, \cdots, \rho_{t_m}^\mu)$ and $\bm\nu^m = (\rho_{t_0}^\nu, \cdots, \rho_{t_m}^\nu)$ as marginal constraints.
Suppose that there exists $L_\rho\in(0, 1)$, such that
\begin{align*}
L_\rho \leq \rho_t^\mu(x), \rho_{t_j}^\nu(x) \leq L_\rho^{-1} \quad \hbox{uniformly for all $t\in[0, 1]$, $j\in\{0, 1, \cdots, m\}$, and $x\in\mb T^d$.}
\end{align*}

Let $R^{\bm\nu^m}$ and $R^{\bm\mu^m}$ be the multi-marginal Schr\"odinger bridges \eqref{eqn:SB-multi} connecting $\bm \nu^m$ and $\bm\mu^m$, respectively.
Under Assumption~\ref{assump: density_expansion} on $(\rho_t^\mu)_{t\in[0, 1]}$, 
we have 
\begin{align}\label{eqn: stability_bound}
\KL(R^{\bm\nu^m}\,\|\,R^{\bm\mu^m})
\leq \KL(\rho_1^\nu\,\|\,\rho_1^\mu) + \frac{1}{2}\sum_{j=1}^m\int_{t_{j-1}}^{t_j}\!\int_{\mb T^d}\Big\|\nabla\bar \Phi_{t_{j-1}}^\mu - \nabla\bar\Phi_t^\nu - \frac{1}{2}\nabla\log\frac{\rho_{t_{j-1}}^\mu}{\bar\rho_t^\nu}\Big\|^2\,\dd \bar\rho_t^\nu\dd t + O\Big(\frac{1}{m}\Big),
\end{align}
where given any $\alpha\in(0, 1)$, the big-O notation omits a constant depending only on $d$, $\alpha$, $L_\rho$,  
\begin{align*}
\begin{cases}
\sum_{j=1}^m \varepsilon_j^{-1} \W_2^2(\rho_{t_{j-1}}^\nu, \rho_{t_j}^\nu),\,\,
\sup_{t\in[0,1]} \big\|\frac{\partial}{\partial t}\Phi_t^\mu\big\|_{\m C^2},\,\,
\sup_{t\in[0,1]}\|\Phi_t^\mu\|_{\m C^{4}},\,\mx{and}\\
\sup_{t\in[0, 1]}\big\{\|\rho_t^{(k)}\|_{\m C^{\max\{2+\lceil\frac{d+1}{2}\rceil, 12\}+k',\alpha}}: 2k+k'=14,\,\,k,k'\in\mb N\big\},\\
\end{cases}
\end{align*}
Also, $(\bar\rho_t^\mu)_{t\in[0, 1]}$ and $(\bar\rho_t^\nu)_{t\in[0, 1]}$ are the piecewise Wasserstein $\W_2$-geodesics connecting $\bm\mu^m$, $\bm\nu^m$, respectively, with velocity vector fields $\nabla\bar\Phi_t^\mu$, $\nabla\bar\Phi_t^\nu$, respectively, satisfying
\[
\partial_t\bar\rho_t^\mu + \nabla\cdot(\bar\rho_t^\mu\nabla\bar\Phi_t^\mu) = 0,  \qquad 
\partial_t\bar\rho_t^\nu + \nabla\cdot(\bar\rho_t^\nu\nabla\bar\Phi_t^\nu) = 0,\quad t\in[t_{j-1}, t_j] \hbox{ for each $j=1, ..., m$}.
\]
\myqed
\end{theorem}

The upper bound in the quantitative stability result \eqref{eqn: stability_bound} consists of three terms. 
The first two terms directly reflect the effect of the difference in the given marginal distributions. The included density ratio term naturally arises from EOT with diminishing regularization coefficient~\citep{conforti2021formula, chizat2020faster}. 
The last term $O(m^{-1})$ represents additional high-order error which vanishes as the number $m$ of marginal constraints increases. It occurs after expanding the Schr\"odinger potentials with respect to the regularization coefficient.


Theorem~\ref{thm: multiSB_stable} is established by first using the connection between the primal solution and the dual solution (i.e. the Schr\"odinger potentials) of EOT as provided in~\eqref{eqn: SB_primal_sol}, and then analyzing the first and second order terms after expanding the Schr\"odinger potentials with respect to the regularization coefficient. This expansion step is one of the key results in this paper, which is summarized in Theorem~\ref{thm: asymp_schro_potential}.

\begin{rem}[Regularity assumptions]
We emphasize that the stability result in Theorem~\ref{thm: multiSB_stable} requires Assumption~\ref{assump: density_expansion} to hold only for $(\rho_t^\mu)_{t\in[0, 1]}$, while imposing very mild assumptions on $(\rho_t^\nu)_{t\in[0, 1]}$. 
We discuss here this regularity assumption from several perspectives.

(1) It is reasonable to only impose the regularity assumption on $(\rho_t^\mu)_{t\in[0, 1]}$. In the potential mathematical biology application of Theorem~\ref{thm: multiSB_stable}, $(\rho_t^\mu)_{t\in[0, 1]}$ represents the ground truth evolution of gene expression levels in cell populations, which is typically assumed to exhibit certain smoothness.  
By contrast, we place fewer assumptions on $(\rho_t^\nu)_{t\in[0, 1]}$, since in practice $\rho_t^\nu$ usually corresponds to an estimator; imposing additional constraints on the estimator will complicate its computation.

(2) We would like to stress again that the smoothness required for $\rho_t^\mu$ in the statement is not sharp. It also indicates that the dependence of the omitted constant in $O(1/m)$ on the regularity of $\rho_t^\mu$ is not optimal.
In fact, both this result and Theorem~\ref{thm: asymp_schro_potential} rely on Theorem~\ref{thm: eot_expansion}, which provides a general Taylor expansion of the EOT cost at arbitrary order. In the proof of Theorem~\ref{thm: eot_expansion}, we repeatedly use $\m C^{k,\alpha}$-norm to bound the $L^2$-norm of the  $k$-th order derivative of the Schr\"odinger potentials for simplicity, rather than using the Sobolev norm $H^k$. A more precise analysis would instead require a lower-order Holder norm together with a Sobolev norm of appropriate order.
\myqed
\end{rem}

\begin{rem}[More general ambient space]
We emphasize that the choice of the flat torus $\mb T^d$ is mainly for simplicity. Our analysis extends to any compact 
smooth Riemannian manifolds without boundary.  
\myqed
\end{rem}

\begin{rem}[Connection with~\citep{yao2025convergence} and comparison with Girsanov's theorem]
If there also exists a vector field $\nabla\Phi_t^\nu$ such that the continuity equation $\partial_t\rho_t^\nu + \nabla\cdot(\rho_t^\nu\nabla\Phi_t^\nu) = 0$ holds, then $\rho_t^\mu$ and $\rho_t^\nu$ are marginal distributions of the solutions to
the following SDEs: 
\begin{align*}
\dd X_t^\mu &= \Big[\nabla\Phi_t^\mu(X_t^\mu) + \frac{1}{2}\nabla\log\rho_t^\mu(X_t^\mu)\Big]\,\dd t + \dd W_t,\\
\dd  X_t^\nu &= \Big[\nabla\Phi_t^\nu(X_t^\nu) + \frac{1}{2}\nabla\log\rho_t^\nu(X_t^\nu)\Big]\,\dd t + \dd W_t.
\end{align*}
The main result in a recent work~\citep{yao2025convergence} implies that 
\begin{align*}
\lim_{m\to\infty} \KL\big(\Law(X^\mu)\,\|\,R^{\bm\mu^m}\big)
= \lim_{m\to\infty} \KL\big(\Law(X^\nu)\,\|\,R^{\bm\nu^m}\big) = 0,
\end{align*}
when the drift coefficients of the above SDEs are smooth enough. 
It means that the many-marginal Schr\"odinger bridges $R^{\bm\mu^m}$ and $R^{\bm\nu^m}$ converge to the law of the above SDEs as $m\to \infty$. Therefore, one may expect that
\begin{align*}
\lim_{m\to\infty} \KL(R^{\bm\nu^m}\,\|\,R^{\bm\mu^m}) 
&= \KL(\rho_0^\nu\,\|\,\rho_0^\mu) + \frac{1}{2}\int_0^1\!\!\int_{\mb T^d}\Big\|\nabla\Phi_t^\mu - \nabla\Phi_t^\nu + \frac{1}{2}\nabla\log\frac{\rho_t^\mu}{\rho_t^\nu}\Big\|^2\,\dd \rho_t^\nu\dd t\\
&= \KL(\rho_1^\nu\,\|\,\rho_1^\mu) + \frac{1}{2} \int_0^1\!\!\int_{\mb T^d}\Big\|\nabla\Phi_t^\mu - \nabla\Phi_t^\nu - \frac{1}{2}\nabla\log\frac{\rho_t^\mu}{\rho_t^\nu}\Big\|^2\,\dd \rho_t^\nu\dd t.
\end{align*}
Here, the first line follows from the classic Girsanov's Theorem, while the second line can be derived by using integration by parts.

From this similarity, one may view Theorem~\ref{thm: multiSB_stable} as a quantitative version of the Girsanov's theorem.
 In our  upper bound~\eqref{eqn: stability_bound} the vector fields $\nabla \Phi^\mu_{t}$ and $\nabla \Phi^\nu_t$, and the marginals $\rho^\mu_t$, $\rho^\nu_t$ are replaced with  the vector fields $\nabla \bar\Phi^\mu_{t_j}$ and $\nabla \bar \Phi^\nu_t$ and the marginals from the piecewise $\W_2$ geodesics.
Importantly, these are  the terms that are determined only by the marginal distributions $\bm\mu^m$, $\bm\nu^m$, and the corresponding estimate holds for any finite number of marginals in the many-marginal Schr\"odinger bridge problem. 
Moreover, we do not impose any regularity assumptions on $\rho_t^\nu$; this feature is important when we apply our stability estimate to those $\bm \nu^m$ obtained as empirical approximations of $\bm \mu^m$.
\myqed
\end{rem}


The following proposition establishes the asymptotic tightness of Theorem~\ref{thm: multiSB_stable}, showing that the upper bound in~\eqref{eqn: stability_bound} is of order of $O(\frac{1}{m})$ when the two sets of marginal constraints in the Schr\"odinger bridge problem coincide.
Its proof is deferred to Appendix~\ref{app: pf vs_conforti}.
\begin{proposition}[Asymptotic tightness]\label{prop: vs_conforti}
Under Assumption~\ref{assump: density_expansion}, in the special case where $\rho_t^\mu = \rho_t^\nu$ for all $t\in[0,1]$, 
the upper bound turns into 
\begin{align*}
\frac{1}{2}\sum_{j=1}^m\int_{t_{j-1}}^{t_j}\!\int_{\mb T^d}\Big\|\nabla\bar \Phi_{t_{j-1}}^\mu - \nabla\bar\Phi_t^\mu - \frac{1}{2}\nabla\log\frac{\rho_{t_{j-1}}^\mu}{\bar\rho_t^\mu}\Big\|^2\,\dd \bar\rho_t^\mu\dd t + O\Big(\frac{1}{m}\Big) = O\Big(\frac{1}{m}\Big).
\end{align*}
\end{proposition}

\paragraph{Asymptotics when $\varepsilon\to 0^+$.}
Based on the Markov property~\eqref{eqn: Markov} of the multi-marginal Schr\"odinger bridge problem and the structure~\eqref{eqn: SB_primal_sol} of the primal solution in the two-marginal case, it can be shown (see Lemma~\ref{lem:KL-decomp-EOT})  that 
\begin{align}\label{eqn:KL-potentials}
\KL(R^{\bm\nu^m}\,\|\,R^{\bm\mu^m})
\stackrel{}{=} \sum_{j=1}^m \frac{1}{\varepsilon_j}\bigg[\eot_{\varepsilon_j}(\rho_{t_{j-1}}^\nu, \rho_{t_j}^\nu) - \int_{\mb T^d}\phi_j^\mu\,\dd\rho_{t_{j-1}}^\nu - \int_{\mb T^d}\psi_j^\mu\,\dd\rho_{t_j}^\nu\bigg] + \sum_{j=0}^m  \KL(\rho_{t_j}^\nu\,\|\,\rho_{t_j}^\mu),
\end{align}
where $(\phi_j^\mu, \psi_j^\mu)$ are the Schr\"odinger potentials for $\eot_{\varepsilon_j}(\rho_{t_{j-1}}^\mu, \rho_{t_j}^\mu)$ with the regularization coefficient $\varepsilon_j = t_j - t_{j-1}$.
Here, note that the last term $\sum_{j=0}^m \KL(\rho_{t_j}^\nu\,\|\,\rho_{t_j}^\mu)$ is of order $O(m)$, so certain cancellations must occur to yield a constant-order estimate. In fact, such cancellations indeed arise when evaluating the integral terms, which capture the dependence structure between the two sets of marginal constraints $\bm\nu^m$ and $\bm\mu^m$ in the KL divergence between two multi-marginal Schr\"odinger bridges.
As illustrated in Figure~\ref{fig: illstration}, the key to proving Theorem~\ref{thm: multiSB_stable} is the fact that $\phi_j^\mu$ and $\psi_j^\mu$ are small when $\varepsilon_j$ is small. Due to the $\varepsilon_j^{-1}$ scaling and the summation in the above formula, it suffices to expand $\phi_j^\mu$ and $\psi_j^\mu$ to second order in $\varepsilon_j$, with the first-order terms together contributing an $O(m)$ quantity after summation, partially canceling 
the sum of KL divergences. The expansion is justified by the following theorem, of which the proof is deferred to Section~\ref{sec: main pf expansion}.



\begin{theorem}[Asymptotics of Schr\"odinger potentials]\label{thm: asymp_schro_potential}
Suppose $(\rho_t^\mu)_{t\in[0, 1]}$ is a strictly positive density evolution satisfying Assumption~\ref{assump: density_expansion}. Moreover, assume that there exists a constant $L_\rho\in(0, 1)$ such that 
\begin{align*}
\hbox{$L_\rho \leq \rho_t^\mu(x) \leq L_\rho^{-1}$ for all $x\in\mb T^d$ and $t\in [0,1]$.}
\end{align*}
Then, for each $K \in \mb N^\ast$, there exists a constant
\begin{align*}
    \hbox{$\varepsilon_{\rm thres} > 0$ depending only on $K, d, L_\rho$, and $\sup_{0\leq t\leq 1}\big\{\|\rho_t^{(i)}\|_{\m C^{s+4K-2i+5}}: 0\leq i\leq 2K+2\big\}$}
\end{align*}
 for which the following holds:
If $\max_{j\in[m]}\varepsilon_j \leq \varepsilon_{\rm thres}$, then for every $j\in[m]\coloneqq\{1, 2,\cdots,m\}$, 
there exist continuous functions $\{f_{j,k}: j\in[m], 1\leq k<K\}$ and $\{g_{j,k}: j\in[m], 1\leq k<K\}$ 
on $[0,1] \times \mathbb{T}^d$
satisfying
\begin{align}\label{eqn: Sch-K-expansion}
\min_{c_j\in\mb R}\bigg\{\Big\|\phi_j^\mu + c_j - \sum_{k=1}^{K-1} \varepsilon_j^k f_{j,k}\Big\|_{L^2(\rho_{t_{j-1}}^\mu)} +
\Big\|\psi_j^\mu - c_j - \sum_{k=1}^{K-1} \varepsilon_j^k g_{j,k}\Big\|_{L^2(\rho_{t_j}^\mu)}\bigg\}
\leq O(\varepsilon_j^K),
\end{align}
where  the big-O notation depends only on 
\begin{align*}
\alpha\in(0, 1),\,\, d,\,\,K,\,\,\mbox{and}\,\, \sup\big\{\|\rho_t^{(k)}\|_{\m C^{\max\{2+\lceil\frac{d+1}{2}\rceil, 4K\}+k',\alpha}}: t\in[0, 1],\,\, 2k+k'=4K+2,\,\,k,k'\in\mb N\big\}.
\end{align*}
\myqed
\end{theorem}
\begin{rem}
    At a first glance, the estiamte \eqref{eqn: Sch-K-expansion} looks nothing but a simple consequence of the Taylor's theorem, as one can get the expansion of $\phi^\mu_j, \psi^\nu_j$'s in $\varepsilon$. However, then the error term depends on  $\phi^\mu_j, \psi^\mu_j$.  What is nontrivial in our case is that the error term in the right-hand side $O(\varepsilon_j^K)$ depends only on the given data $\rho^\mu_t$, while those $\phi^\mu_j, \psi^\mu_j$'s are (dual) solutions of  $\eot_{\varepsilon_j}(\rho_{t_{j-1}}^\mu, \rho_{t_j}^\mu)$. 
\end{rem}

\begin{rem}[Coefficient functions of the expansion]
The coefficient functions $\{f_{j, k}: j\in[m], 1\leq k < K\}$ and $\{g_{j,k}: j\in[m], 1\leq k<K\}$ can be explicitly constructed after solving a system of functional  equations \eqref{eqn: asymp-Schro} for $\schro(\rho_{t_{j-1}}^\mu, \rho_{t_{j-1}}^{(1)}$, $\rho_{t_{j-1}}^{(2)}, \cdots, \rho_{t_{j-1}}^{(K)})$ defined in Section~\ref{sec: EOT expansion}. This system arises naturally from a Fourier transform argument inspired by~\citep{mordant2024entropic}. We postpone the derivation of this system of equations and the existence of its solution in Section~\ref{sec: EOT expansion}, and the explicit construction of $f_{j,k}, g_{j,k}$ is referred to~\eqref{eqn: explicit_construct} in its proof.
\end{rem}

Theorem~\ref{thm: asymp_schro_potential} can also be applied to entropic optimal self-transport problems, where two marginal constraints in the EOT problem are same. It has been shown that self-transport problem is related to the self-attention mechanism in artificial neural networks~\citep{sander2022sinkformers} and score function estimation~\citep{mordant2024entropic}. Recently,~\cite{agarwal2024iterated, agarwal2025langevin} explored the quantitative difference between the entropic Brenier maps (gradient of Schr\"odinger potentials in self-transport problems) and its first-order expansion with respect to $\varepsilon$, known as the score function (gradient of log-density function), on the Euclidean space. They showed that the $L^2$-norm of the difference depends on the Fisher information of the density function. This aspect is aligned with the appearance of the log density ratio in~\eqref{eqn: stability_bound}.

The following corollary of Theorem~\ref{thm: asymp_schro_potential} generalizes the result in~\cite{agarwal2024iterated, agarwal2025langevin} to expansions at any arbitrary order on the flat torus $\mb T^d$. Its proof is deferred to Appendix~\ref{app: pf Brenier_map}.
\begin{corollary}[Estimating entropic Brenier maps in self-transport problems]\label{coro: Brenier_map}
Suppose $\rho(x)\in\m C^\infty(\mb T^d)$ is a probability density function over $\mb T^d$ bounded away from zero. Let $\phi_\varepsilon^{\rm self}$ be the Schr\"odinger potential for the entropic optimal self-transport problem $\eot_\varepsilon(\rho, \rho)$. Then, for any given $K\in\mb N^\ast$ and $\alpha\in(0, 1)$, there exists functions $f_1, \cdots, f_{K-1}\in\m C^\infty(\mb T^d)$ and a constant $C> 0$ depending only on
\begin{align*}
\alpha,\,\,d,\,\,K,\,\,L_\rho,\,\,\mbox{and}\,\, \sup_{t\in[0, 1]}\big\{\|\rho_t^{(k)}\|_{\m C^{\max\{2+\lceil\frac{d+1}{2}\rceil, 12K\}+k',\alpha}}: , 2k+k'=12K+2,\,\,k,k'\in\mb N\big\}
\end{align*}
such that
\begin{align*}
\Big\|\nabla \phi_\varepsilon^{\rm self} - \sum_{k=1}^{K-1} \varepsilon^k \nabla f_k\Big\|_{L^2(\rho)} \leq C \varepsilon^K.
\end{align*}
\myqed
\end{corollary}
Here again, the nontrivial part is that the constant $C$ depends only on the given data $\rho^\mu_t$, not on the Schr\"odinger potential $\phi_\varepsilon^{\rm self}$ derived from the data.

The proof of Theorem~\ref{thm: asymp_schro_potential} relies on two-step estimates. First, we show in Theorem~\ref{thm: eot_expansion} that the cost value $\eot_\varepsilon(\rho_t^\mu, \rho_{t+\varepsilon}^\mu)$ admits a Taylor expansion with respect to the regularization coefficient $\varepsilon$.  Next, we prove a quantitative stability result, 
Theorem~\ref{thm: stability}, of the EOT cost as a functional of the Schr\"odinger potential by applying the duality formulation~\eqref{eqn: duality}. 
This stability result implies that, for any function $\phi$ at which the dual functional value defined in~\eqref{eqn: strong_dual} is close to $\eot_\varepsilon(\rho_t^\mu, \rho_{t+\varepsilon}^\mu)$ up to order $\varepsilon^{2K+1}$, 
the function $\phi$ itself is close to the true Schr\"odinger potential for solving $\eot_\varepsilon(\rho_t^\mu, \rho_{t+\varepsilon}^\mu)$ up to order $\varepsilon^K$ in $L^2$ sense.
Therefore, we can first propose an explicit form for the Taylor expansion of $\phi_j^\mu$ with respect to $\varepsilon_j = t_j - t_{j-1}$, then evaluate the dual functional in~\eqref{eqn: strong_dual} at this candidate function, and finally show that the resulting dual functional value and $\eot_{\varepsilon_j}(\rho_{t_{j-1}}^\mu, \rho_{t_j}^\mu)$ share the same low-order Taylor expansion in $\varepsilon_j$ using Theorem~\ref{thm: eot_expansion}.
The detailed proof of Theorem~\ref{thm: asymp_schro_potential} is deferred to Section~\ref{sec: main pf expansion}.



\subsection{Existing literature}
In this section, we review prior results most relevant to our analysis. To prove Theorem~\ref{thm: multiSB_stable}, it suffices to establish the asymptotic behavior of the Schr\"odinger potentials as $\varepsilon\to 0^+$. For this purpose, we first summarize two essential components: the quantitative stability results for the Schr\"odinger functionals $\m I_{\mu, \nu}^\varepsilon$ and $\overline{\m I}_{\mu, \nu}^\varepsilon$ defined in~\eqref{eqn: dual functional}, and the asymptotic expansions of the EOT cost with respect to the regularization coefficient $\varepsilon$. After reviewing these two topics, we then discuss existing studies on the asymptotics of the Schr\"odinger potentials themselves.

\paragraph{Stability of Schr\"odinger bridge problem.}
\cite{carlier2024displacement} studied the displacement smoothness of the multi-marginal Schr\"odinger bridge problem, showing that the change in Schr\"odinger potentials under the $\m C^k$-norm can be bounded by the $\W_2$-perturbation of the marginal constraints, with a constant of order $O(\poly(\varepsilon)e^{\frac{c}{\varepsilon}})$; notice that this constant blows up as $\varepsilon \to 0^+$.
A similar exponential dependence on the regularization coefficient also appears in \cite[Theorem 17]{deligiannidis2024quantitative}, which established the stability of the product of Schr\"odinger potentials in the $L^\infty$-norm.
Recently, a growing number of works have investigated the stability of (entropic) optimal transport with improved dependence on the regularization coefficient $\varepsilon$  under various settings. Most of these focus on the case where only one of the two marginals is perturbed, which is closely related to the linearized OT problem~\citep{wang2013linear}.
For example,~\cite{divol2025tight} established a Lipschitz bound between the $L^2$-norm of the difference in the gradients of the Schr\"odinger potentials---also known as \emph{entropic Brenier maps}---and the $\W_2$-distance between the marginals, with a constant of order $O(\varepsilon^{-1})$. 
Similarly, \cite{chiarini2024semiconcavity} analyzed stability of the optimal coupling in the EOT problem when the sum of cost function and the Schr\"odinger potential  $\phi_\varepsilon + c_\varepsilon(x,\cdot)$ is semiconcave. They showed that the KL divergence between optimal couplings can be controlled by the KL divergence between the perturbed marginals and their squared $\W_2$-distance with a coefficient of order $O(\varepsilon^{-1})$.
Another line of work on the stability of (entropic) optimal transport relies on the Brascamp--Lieb inequality~\citep{delalande2023quantitative} or the Prekopa--Leindler's inequality~\citep{delalande2022nearly}. 
Their approaches have been extended to the manifold setting using a covering argument~\citep{kitagawa2025stability}.
The developed techniques have recently been applied to study the statistical rates of estimating Wasserstein barycenters~\citep{carlier2024quantitative} and to analyze the convergence of the Sinkhorn algorithm~\citep{chizat2025sharper}.

\paragraph{Asymptotics of the EOT cost.} 

In the continuous setting on Euclidean space,~\cite{duong2013wasserstein, erbar2015large} established a first-order asymptotic expansion of the functional $\eot_\varepsilon(\cdot, \nu)$ with quadratic cost in $\varepsilon$ for a given probability distribution $\nu$, in the sense of $\Gamma$-convergence. \cite{pal2024difference} analyzed the expansion of $\eot_\varepsilon(\mu, \nu)$ in $\varepsilon$ for a general cost function given two probability distributions $\mu$ and $\nu$, using an explicit approximation of the primal solution to the EOT problem. Furthermore, by leveraging the dynamical formulation of entropic optimal transport~\cite[][Section 4.5]{chen2021stochastic}, the second-order asymptotic expansion of $\eot_\varepsilon(\mu, \nu)$ in $\varepsilon$ is derived on the Euclidean space by~\cite{chizat2020faster} and on a smooth, connected, and complete Riemannian manifold without boundary by~\cite{conforti2021formula}, when the cost function is induced by the heat kernel.
In our case, we establish an expansion in arbitrary order (Theorem~\ref{thm: eot_expansion}) when the marginal constraints also get closer as the regularization coefficient $\varepsilon$ decreases.
In the semi-descrete setting,~\cite{altschuler2022asymptotics} derived a similar asymptotic expansion with a different second-order term.

\paragraph{Asymptotics of Schr\"odinger potentials.}
Most existing results on this line focus on entropic Brenier maps, that is, the gradients of Schr\"odinger potentials. In the continuous setting,~\cite{pooladian2021entropic} showed that the entropic Brenier map converges to the optimal transport map in $L^2$-norm at a rate of $O(\varepsilon)$. The papers~\citep{agarwal2024iterated, agarwal2025langevin} studied the entropic optimal self-transport problem, where both marginal constraints in the EOT problem are the same. They established the $L^2$-convergence of the entropic Brenier map to its first-order asymptotic expansion, which includes a score function term. \cite{mordant2024entropic} used Fourier transform method to provide an informal explanation of the asymptotic expansion (in higher order) of Schr\"odinger potentials in the entropic self-transport setting. We use this idea to get our higher order expansion in  Theorem~\ref{thm: asymp_schro_potential} which is beyond the self-transport case.

We note that in the semi-discrete setting, the convergence rate of the entropic Brenier map to the optimal transport map deteriorates to $O(\sqrt{\varepsilon})$ in $L^2$ sense~\citep{pooladian2023minimax}. A recent work~\citep{sadhu2025approximation} established a convergence rate of $O(\varepsilon^{1+\alpha} \vee \varepsilon^2 \log^3(1/\varepsilon))$ in the dual Hölder norm $(\m C^\alpha)^*$ for $\alpha \in (0, 1]$.

\subsection{Other notation}
Let $B_{\mb T^d}(x, R) \coloneqq \{y\in\mb T^d: \td(x, y) < R\}$ denote the open ball centered at $x$ with radius $R > 0$ in $\mb T^d$. 
Let $1_Q(x)$ be the indicator function of a set $Q$, taking value $1$ for $x\in Q$ and $0$ otherwise. 
Let $H^p(\mb T^d) = W^{p,2}(\mb T^d)$ be the Sobolev space (of $L^2$-type) over the flat torus $\mb T^d$.
For a distribution $\rho\in\ms P_{ac}(\mb T^d)$, we define its negative self-entropy as $\m H(\rho) = \int_{\mb T^d}\log (\dd\rho/\dd x) \dd\rho$. Let $\m C(\mb T^d)$ denote the set of all continuous functions on $\mb T^d$.



\section{Expansion of Entropic Optimal Transport Cost}\label{sec: EOT expansion}
In this section, we develop the expansion of $\eot_{\varepsilon}(\rho_t^\mu, \rho_{t+\varepsilon}^\mu)$ with respect to the regularization coefficient $\varepsilon>0$ under Assumption~\ref{assump: density_expansion}. Let $(\phi_{t, \varepsilon}, \psi_{t,\varepsilon})$ denote the Schr\"odinger potentials for the entropic optimal transport problem $\eot_\varepsilon(\rho_t, \rho_{t+\varepsilon})$. We also use the notation 
\begin{align}\label{eqn:notation-u-v}
u_{t,\varepsilon} = e^{\frac{\phi_{t,\varepsilon}}{\varepsilon}}
\quad\mx{and}\quad
v_{t,\varepsilon} = e^{\frac{\psi_{t,\varepsilon}}{\varepsilon}}.
\end{align}
The main theorem of this section, that is, Theorem~\ref{thm: eot_expansion} below, is obtained from the following natural idea.
Suppose $u_{t,\varepsilon}$ and $v_{t,\varepsilon}$ and their reciprocals admit Taylor expansion with respect to $\varepsilon$:
\begin{align}\label{eqn:def-u-k}
u_{t,\varepsilon} = \sum_{k=0}^\infty \varepsilon^k u_{t,k},\qquad
\frac{1}{u_{t,\varepsilon}} = \sum_{k=0}^\infty\varepsilon^k u_{t,k}^\dagger,\qquad
v_{t,\varepsilon} = \sum_{k=0}^\infty \varepsilon^k v_{t,k},
\quad\mx{and}\quad
\frac{1}{v_{t,\varepsilon}} = \sum_{k=0}^\infty \varepsilon^k v_{t,k}^\dagger.
\end{align}
Then, by defining the following high-order approximations of $u_{t,\varepsilon}$ and $v_{t,\varepsilon}$,
\begin{align}\label{eqn: proxy_Schro_sys}
U_{t,K}(x):= \sum_{k=0}^K \varepsilon^k u_{t,k}(x)
\quad\,\,\mx{and}\quad\,\,
V_{t,K}(y) := \sum_{k=0}^K \varepsilon^k v_{t,k}(y),
\end{align}
we expect the approximations $\phi_{t,\varepsilon}(x)\approx \varepsilon\log U_{t,K}(x)$ and $\psi_{t,\varepsilon}(y) \approx \varepsilon\log V_{t,K}(y)$ up to high-order terms. Applying the dual formulation~\eqref{eqn: duality} of the EOT problem leads to the heuristic estimate
\[
\eot_\varepsilon(\rho_t, \rho_{t+\varepsilon}) \approx \varepsilon\bigg[\int_{\mb T^d}\log U_{t,K}\,\dd\rho_\varepsilon + \int_{\mb T^d}\log V_{t,K}\,\dd\rho_{t+\varepsilon}\bigg],
\]
which gives the desired expansion in Theorem~\ref{thm: eot_expansion}. Although this idea seems natural, making it rigorous is surprisingly nontrivial and involved.
The rest of this section is devoted 
to the proof of this estimate.

We first present the  relation between the functions $u_{t,k}$ and $v_{t,k}$.  Recall that $\Delta$ is the Laplace--Beltrami operator on the flat torus $\mb T^d$. Now, given functions $\rho_0, \cdots, \rho_K$ on $[0,1] \times \mb T^d$, we introduce the following system of equations involving $\{u_{k}, v_{k}, u_{k}^\dagger, v_{k}^\dagger\}_{k=0}^K$:
\begin{subequations}\label{eqn: asymp-Schro}
\begin{align}
&u_{0}\, v_{0}\rho_0 = 1,\qquad u_{0} u_{0}^\dagger = v_{0} v_{0}^\dagger = 1;\label{eqn: init}\\  
&u_{k} = -u_{0}\sum_{i=1}^k u_{k-i}\, u_{i}^\dagger,\qquad
v_{k} = -v_{0}\sum_{i=1}^k v_{k-i}\, v_{i}^\dagger;\label{eqn: iter_dagger}\\
&u_{k}^\dagger = \sum_{l=0}^k\sum_{i=0}^l \frac{\Delta^{k-l}(v_{i}\, \rho_{l-i})}{2^{k-l}(k-l)!},
\qquad
v_{k}^\dagger = \sum_{l=0}^k \frac{\Delta^{k-l}(u_{l}\,\rho_0)}{2^{k-l}(k-l)!}\label{eqn: iter_fourier}.
\end{align}
\end{subequations}
We also use the notation
\begin{align}\label{eqn:notation-ScB}
\{(u_{k}, v_{k}): k=0, 1, \cdots, K\} = \schro(\rho_0, \rho_1, \cdots, \rho_K)
\end{align}
for the solution to this system~\eqref{eqn: asymp-Schro}, if it exists. A heuristic derivation of~\eqref{eqn: asymp-Schro} and the existence of its solution will be provided later in Section~\ref{subsec: derive_eqn}. 

With the above notation, we are now ready to present the first  main result of  this section. 

\begin{theorem}[EOT expansion]\label{thm: eot_expansion}
Suppose $(\rho_t)_{t\in[0, 1]}$ satisfies Assumption~\ref{assump: density_expansion}.
Further assume that  there exists a constant $L_\rho\in(0, 1)$ such that $L_\rho \leq \rho_t(x) \leq L_\rho^{-1}$ uniformly holds for all $x\in\mb T^d$ and $t\in[0, 1]$. Following~\eqref{eqn:notation-ScB}, let
\begin{align}\label{eqn: uvtk}
\{(u_{t,k}, v_{t,k}): k=0,1,\cdots, K+1\} = \schro(\rho_t, \rho_t^{(1)}, \cdots, \rho_t^{(K+1)}).
\end{align}
Let $\alpha\in(0, 1)$ be any fixed number and $s = \lceil\frac{d+1}{2}\rceil$.
Then, for any fixed $K\in\mb N^\ast$, there exists $\varepsilon_0 > 0$ depending only  on 
\begin{align*}
\begin{cases}
  & \hbox{$\alpha$, $K$, $d$, $L_\rho$,  and}\\
  & \sup\big\{\|\rho_r^{(i)}\|_{H^{2j+s}}: r\in[t,t+\varepsilon], \,\,i,j\in\mb N,\,\, i+j\leq 2\big\},  
 \hbox{ and } \big\{\|\rho_t^{(K+2-i)}\|_{\m C^{s+2i,\alpha}}: i\in[K+2]\big\}, 
\end{cases}
\end{align*}
such that whenever $0\leq \varepsilon \leq \min\{\varepsilon_0, 1 - t\}$, we have
\begin{align}\label{eqn: eot_cost_expand}
\eot_{\varepsilon}(\rho_{t}, \rho_{t+\varepsilon}) = \varepsilon\int_{\mb T^d} \log\Big[\sum_{k=0}^K\varepsilon^k u_{t, k}\Big]\,\dd\rho_{t}  + \varepsilon\int_{\mb T^d}\log\Big[\sum_{k=0}^K \varepsilon^k v_{t,k}\Big]\,\dd\rho_{t+\varepsilon} +  O(\varepsilon^{K+1}).
\end{align}
Here, the big-O notation omits a constant depending  only on 
\begin{align*}
\begin{cases}
    & \hbox{ $d$, $\alpha$, $K$, $L_\rho$, and}\\
& 
\big\{\|\rho_t^{(K+2-i)}\|_{\m C^{2K+2i-2,\alpha}}: i\in[K+2]\big\}, \hbox{ and }
\sup\big\{\|\rho_r^{(i)}\|_{\m C^{2j}}: r\in[t, t+\varepsilon], \,\, i,j\in\mb N,\,\,i+j\leq K+1\big\}
\end{cases}.
\end{align*}
\myqed
\end{theorem}

\begin{rem}[Dependence on $\{\rho_r: t\leq r \leq t+\varepsilon\}$]
In the proof, we first track how $\varepsilon_0$ and the constant hidden in the big-O notation depend on $u_{t,0}, \cdots, u_{t,K}$ and $v_{t,0}, \cdots, v_{t,K}$, then apply Theorem~\ref{thm: existence} below to convert this dependence to one on $\{\rho_r: t\leq r\leq t+\varepsilon\}$. In fact, the Holder norm $\|\rho_t^{(K+2-i)}\|_{\m C^{2K+2i-2,\alpha}}$ used in the statement can be relaxed to Holder norms with smaller indices combined with Sobolev norms. For simplicity, we directly use the Holder norms to bound the Sobolev norms.
\end{rem}

\begin{rem}[Comparison with existing results]
To prove the main stability result, Theorem~\ref{thm: multiSB_stable}, it suffices to use the above theorem with $K=6$.  To the best of our knowledge, existing works~\cite{chizat2020faster, conforti2021formula} addressed the asymptotic expansion of the entropic optimal transport cost $\eot_\varepsilon(\mu, \nu)$ for two arbitrary probability distributions in the small-$\varepsilon$ regime, up to  $K = 2$. In contrast, Theorem~\ref{thm: eot_expansion} focuses on higher-order expansions with two sufficiently close marginal distributions. 
\end{rem}

Theorem~\ref{thm: eot_expansion} is proved by establishing a matching upper bound using the primal formulation~\eqref{eqn: general_EOT} and an lower bound using the dual formulation~\eqref{eqn: strong_dual} of the EOT problem. These two parts will be presented separately in Section~\ref{sec: eot_expansion_lb} and Section~\ref{sec: eot_expansion_ub}. Before diving into the proof, we state  a corollary of the theorem which provides a polynomial expansion (in $\varepsilon$) of $\eot_\varepsilon(\rho_t, \rho_{t+\varepsilon})$. The proof directly follows from the statement of Theorem~\ref{thm: eot_expansion}.

The rest of this section is structured as follows. In Section~\ref{subsec: derive_eqn}, we provide the heuristic derivation of the system of equations~\eqref{eqn: asymp-Schro} and establish its existence of solution. Sections~\ref{sec: eot_expansion_lb} and~\ref{sec: eot_expansion_ub} focus on the proof of Theorem~\ref{thm: eot_expansion}, with helps from Sections~\ref{sec: exist_Fredholm} and~\ref{sec: existence}.
In Section~\ref{sec: existence} we prove existence of the solution to~\eqref{eqn: asymp-Schro};
see Theorem~\ref{thm: existence}. 
For simplicity, we omit the dependence on time $t$ and only consider the case where $t = 0$ in the rest of this section.

\subsection{Derivation of~\eqref{eqn: asymp-Schro} and the existence of solution}\label{subsec: derive_eqn}
Recall that $(\phi_\varepsilon, \psi_\varepsilon)$ denote the Schr\"odinger potentials for the entropic optimal transport problem $\eot_\varepsilon(\rho_0, \rho_\varepsilon)$, and we use the notation $(u_\varepsilon, v_\varepsilon) = (e^{\frac{\phi_\varepsilon}{\varepsilon}}, e^{\frac{\psi_\varepsilon}{\varepsilon}})$.
Under the regularity Assumption~\ref{assump: density_expansion}, we can assume that $\rho_\varepsilon$ admits a Taylor expansion with respect to $\varepsilon$ of the form $\rho_\varepsilon = \sum_{k=0}^\infty\varepsilon^k\rho_k$. 
Recall that the cost function is given by $c_\varepsilon(x, y) = -\varepsilon\log\m K_\varepsilon(x-y)$. Then, the Schr\"odinger system~\eqref{eqn: Schro_sys} can be reformulated as
\begin{subequations}\label{eqn: Schro_sys_Td}
\begin{align}
\frac{1}{u_\varepsilon(x)} &= \int_{\mb T^d}\m K_\varepsilon(x-y)\big[v_\varepsilon(y)\rho_\varepsilon(y)\big]\dd y\label{eqn: Schro_sys_Td-a},\\
\frac{1}{v_\varepsilon(y)} &= \int_{\mb T^d}\m K_\varepsilon(y-x) \big[u_\varepsilon(x)\rho_0 (x)\big]\dd x.\label{eqn: Schro_sys_Td-b}
\end{align}
\end{subequations}

Now, let us provide an \emph{informal} derivation of the dependence of $u_\varepsilon$ and $v_\varepsilon$ on $\varepsilon$ using Fourier series; this idea is inspired by~\cite{mordant2024entropic}, which studied the entropic optimal self-transport problem.  Suppose $u_\varepsilon$ and its reciprocal function $\frac{1}{u_\varepsilon}$  
admit Taylor expansions with respect to $\varepsilon$ in the form 
\begin{align}\label{eqn: inverse expansion}
u_\varepsilon = \sum_{k=0}^\infty \varepsilon^k u_k
\quad\mx{and}\quad
\frac{1}{u_\varepsilon} = \sum_{k=0}^\infty \varepsilon^k u_k^\dagger,
\end{align}
and similar for $v_\varepsilon$ and $\frac{1}{v_\varepsilon}$. Then, from 
\begin{align*}
1 = \Big(\sum_{k=0}^\infty \varepsilon^k u_k\Big)\Big(\sum_{k=0}^\infty \varepsilon^k u_k^\dagger\Big) = \Big(\sum_{k=0}^\infty \varepsilon^k v_k\Big)\Big(\sum_{k=0}^\infty \varepsilon^k v_k^\dagger\Big),
\end{align*}
we obtain $u_0u_0^\dagger = v_0v_0^\dagger = 1$, as well as the following identities for higher-order terms:
\begin{align*}
u_k = -u_0\sum_{i=1}^k u_{k-i}u_i^\dagger
\,\,\quad\mx{and}\,\,\quad
v_k = -v_0\sum_{i=1}^k v_{k-i}v_i^\dagger,
\qquad\forall\,k\in\mb Z_+.
\end{align*}
This leads to the derivation of~\eqref{eqn: iter_dagger} and the second part of~\eqref{eqn: init}. The first part of~\eqref{eqn: init} follows directly by taking the limit $\varepsilon\to 0^+$ in~\eqref{eqn: Schro_sys_Td}.

Let us now consider derivation of~\eqref{eqn: iter_fourier}.
For a function $f$ on $\mb T^d$ and $z\in\mb Z^d$, let $\wht f(z)$ be its Fourier transform defined by
\begin{align*}
\wht f(z) = \frac{1}{(2\pi)^d}\int_{\mb T^d} f(x)e^{-iz\cdot x}\,\dd x.
\end{align*}
Then, as a direct consequence of Poisson summation formula~\citep{stein2011fourier}, the Gaussian convolution kernel $\m K_\varepsilon$ in \eqref{eqn: Gaussian kernel} has the Fourier transform
$$
\wht{\m K}_\varepsilon(z) = e^{-\frac{\varepsilon\|z\|^2}{2}},\quad\forall\,z\in\mb Z^d.
$$ 
Note that we also expect
\begin{align*}
v_\varepsilon\rho_\varepsilon = \Big(\sum_{k=0}^\infty \varepsilon^k v_k\Big)\Big(\sum_{k=0}^\infty\varepsilon^k \rho_k\Big) = \sum_{k=0}^\infty\varepsilon^k\Big(\sum_{l=0}^k v_l\rho_{k-l}\Big).
\end{align*}
Recall the expression \eqref{eqn: inverse expansion} and take the Fourier transform on both sides of~\eqref{eqn: Schro_sys_Td-a}, and get 
\begin{align*}
   \sum_{k=0}^\infty \varepsilon^k \wht{u_k^\dagger}(z) 
& = \wht{\m K_\varepsilon\ast[v_\varepsilon\rho_\varepsilon]}(z)
= \wht{\m K_\varepsilon}(z)\wht{v_\varepsilon\rho_\varepsilon}(z)
= e^{-\frac{\varepsilon\|z\|^2}{2}}\sum_{k=0}^\infty\varepsilon^k\sum_{l=0}^k\widehat{v_l\rho_{k-l}}(z)\\
&= \Big[\sum_{k=0}^\infty \frac{1}{k!}\Big(-\frac{\varepsilon\|z\|^2}{2}\Big)^k\Big]\Big[\sum_{k=0}^\infty\varepsilon^k\Big(\sum_{l=0}^k \wht{v_l\rho_{k-l}}(z)\Big)\Big]
= \sum_{k=0}^\infty \varepsilon^k\sum_{l=0}^k \frac{(-\|z\|^2)^{k-l}}{2^{k-l}(k-l)!}\sum_{i=0}^l\widehat{v_i\rho_{l-i}}(z)\\
&\stackrel{\ri}{=} \sum_{k=0}^\infty\varepsilon^k \sum_{l=0}^k\sum_{i=0}^l \frac{\widehat{\Delta^{k-l}(v_i\rho_{l-i})}(z)}{2^{k-l}(k-l)!}
\end{align*}
for every $z\in\mb Z^d$, where (i) follows from the identity $\widehat{\Delta^{k-l}f}(z) = (-\|z\|^2)^{k-l}\wht f(z)$ for any $f\in H^{2k-2l}(\mb T^d)$.
Similarly, applying the same argument to~\eqref{eqn: Schro_sys_Td-b} yields
\begin{align*}
\sum_{k=0}^\infty \varepsilon^k \widehat{v}_k^\dagger(z) 
= \sum_{k=0}^\infty\varepsilon^k \sum_{l=0}^k \frac{\widehat{\Delta^{k-l}(u_l\rho_0)}(z)}{2^{k-l}(k-l)!}.
\end{align*}
Comparing the coefficients of $\varepsilon^k$ on both sides of the equations indicates that
\begin{align*}
u_k^\dagger = \sum_{l=0}^k\sum_{i=0}^l \frac{\Delta^{k-l}(v_i\rho_{l-i})}{2^{k-l}(k-l)!}
\quad\,\,\mx{and}\quad\,\,
v_k^\dagger = \sum_{l=0}^k \frac{\Delta^{k-l}(u_l\rho_0)}{2^{k-l}(k-l)!},
\quad\forall\, k\in\mb N.
\end{align*}
This is exactly~\eqref{eqn: iter_fourier}, completing the heuristic derivation of~\eqref{eqn: asymp-Schro}.


The following result establishes  existence of
the functions $\{u_k\}_{k\in\mb N}$ and $\{v_k\}_{k\in\mb N}$ satisfying the above conditions, namely the solution to the system~\eqref{eqn: asymp-Schro}
given in the notation~\eqref{eqn:notation-ScB}; note that we are focusing without loss of generality, on the case $t=0$.
To motivate the assumption in the following theorem, notice that since $\rho_\varepsilon$ is a probability density function, its Taylor expansion in $\varepsilon$ satisfies
\begin{align*}
k!\int_{\mb T^d}\rho_k\,\dd x = \int_{\mb T^d}\frac{\dd^k\rho_\varepsilon}{\dd \varepsilon^k}\bigg|_{\varepsilon=0}\,\dd x = \bigg(\frac{\dd^k}{\dd\varepsilon^k}\int_{\mb T^d} \rho_\varepsilon\,\dd x\bigg)\bigg|_{\varepsilon=0} = 0.
\end{align*}

\begin{theorem}[Solutions to the PDE system~\eqref{eqn: asymp-Schro}]\label{thm: existence}
Suppose $\rho_0,\rho_1, \cdots, \rho_K\in\m C^\infty(\mb T^d)$ and satisfy $\rho_0 > 0$,
\begin{align*}
\int_{\mb T^d}\rho_0(x)\,\dd x = 1,
\quad\mx{and}\quad
\int_{\mb T^d}\rho_k(x)\,\dd x = 0 
\quad\forall\,k\in[K].
\end{align*}
Then there exist $\{u_k\}_{k=0}^{K}$ and $\{v_k\}_{k=0}^{K}$ satisfying Equations~\eqref{eqn: asymp-Schro}, that is,
\begin{align*}
\{(u_k, v_k): k=0, 1, \cdots, K\} = \schro(\rho_0, \rho_1, \cdots, \rho_K).
\end{align*}

Moreover, we have:
\begin{itemize}
\item $u_0$ and $v_0$ satisfy
\begin{align*}
\nabla\cdot(\rho_0\nabla\log u_0) &= \rho_1 - \frac{\Delta\rho_0}{2},\\
\nabla\cdot(\rho_0\nabla\log v_0) &= -\rho_1 - \frac{\Delta\rho_0}{2}.
\end{align*}
As a result, there exists constants $L_u, L_v \in(0, 1)$ depending only on $\rho_0$ and $\rho_1$, such that 
$L_u \leq u_0(x) \leq L_u^{-1}$ and $L_v \leq v_0(x) \leq L_v^{-1}$ hold for all $x\in\mb T^d$.

\item For every $1\leq k\leq K-1$, $u_k$ is the solution of the elliptic PDE
\begin{align*}
\nabla\cdot\Big(\rho_0\nabla\frac{u_k}{u_0}\Big) = \frac{\rho_0 F_{k-1}}{u_0},
\end{align*}
where $F_{k-1}$ is a function depending only on $u_0, \cdots, u_{k-1}$, $v_0, \cdots, v_{k-1}$, and $\rho_0, \rho_1, \cdots, \rho_{k+1}$.

\item 

For every non-negative integer $k<K$, we have $u_k, v_k\in\m C^\infty(\mb T^d)$, and for each  $\alpha\in(0, 1)$ and $s\in\mb N^\ast$,
\begin{align*}
    \|u_k \|_ {\m C^{s+2,\alpha}},
     \|v_k \|_ {\m C^{s+2,\alpha}} \le C
\end{align*}
for the constant $C$  depending only on $d, s, k,\alpha$ and $\|\rho_0\|_{\m C^{s+2k+2,\alpha}}, \|\rho_1\|_{\m C^{s+2k,\alpha}}, \cdots, \|\rho_{k+1}\|_{\m C^{s,\alpha}}$.

\end{itemize}
\myqed
\end{theorem}
It is important to have this theorem, which controls $u_j$'s by only using the data $\rho_j$'s. The function $u_j$ and its derivative norms appear in many of the estimates in this paper.
The proof of the above result is deferred to Section~\ref{sec: existence}.

\subsection{Proof of Theorem~\ref{thm: eot_expansion}: lower bound}\label{sec: eot_expansion_lb}
Recall the definitions of $U_K(x)$ and $V_K(x)$ in~\eqref{eqn: proxy_Schro_sys} by omitting the dependence on $t=0$. In this section, we will prove the following lower bound of the EOT cost
\begin{align}\label{eqn: eot_lb}
\eot_\varepsilon(\rho_0, \rho_\varepsilon) \geq \varepsilon\int_{\mb T^d}\log U_K\,\dd\rho_0 + \varepsilon\int_{\mb T^d}\log V_K\,\dd\rho_\varepsilon + C\varepsilon^{K+1},
\end{align}
where $C$ is a constant depending on a regularity set $\mathcal{RS}_{\rm lb}^{(K)}(0, \varepsilon)$ defined later in Lemma~\ref{lem: estim_dual_lb}.

We will get the lower bound~\eqref{eqn: eot_lb} by using the identity~\eqref{eqn: strong_dual}.
The following result shows that the pair $(U_K, V_K)$ 
approximately satisfies the Schr\"odinger system~\eqref{eqn: Schro_sys_Td} associated with $(u_\varepsilon, v_\varepsilon)$ in $\eot_\varepsilon(\rho_0, \rho_\varepsilon)$; its proof is deferred to Appendix~\ref{sec: pf_int_remainder}.

\begin{lemma}\label{lem: int_remainder}
Define the functions
\begin{align}
\begin{aligned}\label{eqn: remainder}
R_\varepsilon(y) &\coloneqq \rho_\varepsilon(y) - \rho_\varepsilon(y)V_K(y)\int_{\mb T^d}\m K_\varepsilon(y-x)U_K(x)\rho_0(x)\,\dd x,\\
Q_\varepsilon(x) &\coloneqq \rho_0(x) - \rho_0(x) U_K(x)\int_{\mb T^d} \m K_\varepsilon(x-y) V_K(y)\rho_\varepsilon(y)\,\dd y.
\end{aligned}
\end{align}
Then, we have
\begin{align*}
\Big\|\frac{R_\varepsilon}{\rho_\varepsilon}\Big\|_{L^2(\mb T^d)} \leq C\varepsilon^{K+1}
\quad\mx{and}\quad
\Big\|\frac{Q_\varepsilon}{\rho_0}\Big\|_{L^2(\mb T^d)} \leq C\varepsilon^{K+1},
\end{align*}
for some constant $C > 0$ depending on the regularity set $\mathcal{RS}_{\rm SBE}^{(K)}(0, \varepsilon)$. Here, the set $\mathcal{RS}_{\rm SBE}^{(K)}(t, t+\varepsilon)$ with a given time point $0\leq t\leq 1-\varepsilon$ contains the elements
\begin{align}\label{set: RS_SBE^K}
\begin{cases}
\|u_{t,0}\|_{\m C^0}, \cdots \|u_{t,K}\|_{\m C^0}, \|u_{t,0}\|_{H^{2K+2}}, \cdots, \|u_{t,K}\|_{H^{2K+2}},\\
\|v_{t,0}\|_{\m C^0}, \cdots \|v_{t,K}\|_{\m C^0}, \|v_{t,0}\|_{H^{2K+2}}, \cdots, \|v_{t,K}\|_{H^{2K+2}},\\
\sup\big\{\|\rho_s^{(k)}\|_{\m C^{2k'}}: t\leq s\leq t+\varepsilon, \,\, k,k'\in\mb N,\,\,k+k'\leq K+1\big\},
\end{cases}
\end{align}
where $u_{t, k}$ and $v_{t,k}$ are defined by Equation~\eqref{eqn: uvtk}.
\end{lemma}

With the above result showing that the $L^2(\mb T^d)$-norm of the remainder terms is of high order, the following lemma provides a lower bound estimate when choosing $\phi_{\rm lb}(x) \coloneqq \varepsilon\log U_K(x)$ and $\psi_{\rm lb}(x) \coloneqq\varepsilon\log V_K(x)$ as the dual potential function in~\eqref{eqn: strong_dual}.
\begin{lemma}\label{lem: estim_dual_lb}
By choosing $\phi_{\rm lb}(x) \coloneqq \varepsilon\log U_K(x)$ and $\psi_{\rm lb}(x) \coloneqq\varepsilon\log V_K(x)$ defined in~\eqref{eqn: proxy_Schro_sys}, we have
\begin{subequations}\label{eqn: strong_dual_lb}
\begin{align}
\int_{\mb T^d}\phi_{\rm lb}\,\dd\rho_0 + \int_{\mb T^d}\m T_{\rho_0}^\varepsilon[\phi_{\rm lb}]\,\dd\rho_\varepsilon &\geq \varepsilon\int_{\mb T^d}\log U_K\,\dd\rho_0 + \varepsilon\int_{\mb T^d}\log V_K\,\dd\rho_\varepsilon + C\varepsilon^{K+2},\label{eqn: strong_dual_lb1}\\
\int_{\mb T^d}\m T_{\rho_\varepsilon}^\varepsilon[\psi_{\rm lb}]\,\dd\rho_0 + \int_{\mb T^d}\psi_{\rm lb}\,\dd\rho_\varepsilon &\geq \varepsilon\int_{\mb T^d}\log U_K\,\dd\rho_0 + \varepsilon\int_{\mb T^d}\log V_K\,\dd\rho_\varepsilon + C\varepsilon^{K+2}\label{eqn: strong_dual_lb2}
\end{align}
\end{subequations}
for some constant $C>0$ depending only on the elements in the set $\m RS_{\rm lb}^{(K)}(0, \varepsilon)$. Here, given $t\in[0, 1-\varepsilon]$, we define $\m RS_{\rm lb}^{(K)}(t, t+\varepsilon)$ as the set containing the elements
\begin{align}\label{set: RS_lb^K}
\begin{cases}
\|u_{t,0}\|_{\m C^0}, \cdots \|u_{t,K}\|_{\m C^0}, \|u_{t,0}\|_{H^{2K+2}}, \cdots, \|u_{t,K}\|_{H^{2K+2}},\\
\|v_{t,0}\|_{\m C^0}, \cdots \|v_{t,K}\|_{\m C^0}, \|v_{t,0}\|_{H^{2K+2}}, \cdots, \|v_{t,K}\|_{H^{2K+2}},\\
L_\rho,\,\, \sup\big\{\|\rho_s^{(k)}\|_{\m C^{2k'}}: t\leq s\leq t+\varepsilon, \,\, k,k'\in\mb N,\,\,k+k'\leq K+1\big\},
\end{cases}
\end{align}
where $u_{t, k}$ and $v_{t,k}$ are defined by Equation~\eqref{eqn: uvtk}.
\end{lemma}
\begin{proof}
First, let us evaluate $\m T_{\rho_0}^\varepsilon[\phi_{\rm lb}]$. By the definition of $\m T_{\rho_0}^\varepsilon$ as in \eqref{eqn: Schro_sys}), we have
\begin{align*}
\m T_{\rho_0}^\varepsilon[\phi_{\rm lb}](y)
&= -\varepsilon\log\int_{\mb T^d} e^{\frac{\phi_{\rm lb}(x) - c_\varepsilon(x, y)}{\varepsilon}}\,\dd\rho_0(x)
=  - \varepsilon\log \int_{\mb T^d}\m K_\varepsilon(x-y)U_K(x)\rho_0(x)\,\dd x\\
&\stackrel{\ri}{=} - \varepsilon\log\frac{\rho_\varepsilon(y) - R_\varepsilon(y)}{\rho_\varepsilon(y)V_K(y)}
=  \varepsilon\log V_K(y) - \varepsilon\log\Big(1 - \frac{R_\varepsilon(y)}{\rho_\varepsilon(y)}\Big)\\
&\geq  \varepsilon\log V_K(y) + \frac{\varepsilon R_\varepsilon(y)}{\rho_\varepsilon(y)},
\end{align*}
where (i) follows from the definition of $R_\varepsilon$ in~\eqref{eqn: remainder}, and the last inequality is due to $\log(1+z) \leq z$ for every $z>-1$. 
By Lemma~\ref{lem: int_remainder}, we have
\begin{align*}
\int_{\mb T^d} \frac{R_\varepsilon}{\rho_\varepsilon}\,\dd \rho_\varepsilon \geq -\bigg(\int_{\mb T^d}\rho_\varepsilon^2\,\dd y\bigg)^{\frac{1}{2}} \bigg(\int_{\mb T^d} \frac{R_\varepsilon^2}{\rho_\varepsilon^2}\,\dd y\bigg)^{\frac{1}{2}} = O(\varepsilon^{K+1}),
\end{align*} 
which implies~\eqref{eqn: strong_dual_lb1} by combining with the previous lower bound of $\m T_{\rho_0}^\varepsilon[\phi_{\rm lb}]$. Here the big-O notation omits a constant depending on $L_\rho$ and the set $\mathcal{RS}_{\rm SBE}^{(K)}(0, \varepsilon)$ defined through~\eqref{set: RS_SBE^K}.

To prove~\eqref{eqn: strong_dual_lb2}, similarly we have
\begin{align*}
\m T_{\rho_\varepsilon}^\varepsilon[\psi_{\rm lb}](x)
&= -\varepsilon\log\int_{\mb T^d}e^{\frac{\psi_{\rm lb}(y) - c_\varepsilon(x, y)}{\varepsilon}}\,\dd\rho_\varepsilon(y)
= -\varepsilon\log\int_{\mb T^d} \m K_\varepsilon(x-y)V_K(y)\rho_\varepsilon(y)\,\dd y\\
&=-\varepsilon\log\frac{\rho_0(x) - Q_\varepsilon(x)}{\rho_0(x)U_K(x)} 
= \varepsilon\log U_K(x) - \varepsilon\log\Big(1 - \frac{Q_\varepsilon(x)}{\rho_0(x)}\Big)\\
&\geq \varepsilon\log U_K(x) + \frac{\varepsilon Q_\varepsilon(x)}{\rho_0(x)}.
\end{align*}
Again, by Lemma~\ref{lem: int_remainder}, we have $\int_{\mb T^d}\frac{Q_\varepsilon}{\rho_0}\,\dd \rho_0\geq O(\varepsilon^{K+1})$. This completes the proof of~\eqref{eqn: strong_dual_lb2}.
\end{proof}

\paragraph{Proof of the lower bound~\eqref{eqn: eot_lb}.}
The result directly follows from the strong duality~\eqref{eqn: strong_dual} and Lemma~\ref{lem: estim_dual_lb}:
\begin{align*}
\eot_\varepsilon(\rho_0, \rho_\varepsilon) 
\geq \int_{\mb T^d}\phi_{\rm lb}\,\dd\rho_0 + \int_{\mb T^d}\m T_{\rho_0}^\varepsilon[\phi_{\rm lb}]\,\dd\rho_\varepsilon
\geq \varepsilon\int_{\mb T^d}\log U_K\,\dd\rho_0 + \varepsilon\int_{\mb T^d}\log V_K\,\dd\rho_\varepsilon + C\varepsilon^{K+1}.
\end{align*}
\qed

\subsection{Proof of Theorem~\ref{thm: eot_expansion}: upper bound}\label{sec: eot_expansion_ub}
In this section, we will prove that there exists a constant $\varepsilon_\pi > 0$, which depends on
\begin{align}\label{eqn: dependence-ub}
\begin{cases}
L_\rho,\,\, \min_{x\in\mb T^d} u_0(x) > 0,\,\, \min_{x\in\mb T^d}v_0(x) > 0\\
\sup\big\{\|\rho_t^{(k)}\|_{H^{2k'+s}}: t\in[0,\varepsilon], \,\,k,k'\in\mb N,\,\, k+k'\leq 2\big\}\\
\|u_0\|_{H^{s+4}}, \cdots, \|u_K\|_{H^{s+4}},\,\,\mbox{and}\,\, \|v_0\|_{H^{s+4}}, \cdots, \|v_K\|_{H^{s+4}}
\end{cases}
\end{align}
with $s = \lceil\frac{d+1}{2}\rceil$, such that when $\varepsilon \leq \varepsilon_\pi$, it holds that
\begin{align}\label{eqn: eot_ub}
\eot_\varepsilon(\rho_0, \rho_\varepsilon) \leq \varepsilon\int_{\mb T^d}\log U_K\,\dd\rho_0 + \varepsilon\int_{\mb T^d}\log V_K\,\dd\rho_\varepsilon + C\varepsilon^{K+1}
\end{align}
for a constant $C > 0$ depending on~\eqref{eqn: dependence-ub}. 

The main idea is to construct a coupling $\pi_{\rm ub}\in\Pi(\rho_0, \rho_\varepsilon)$ using $U_K$ and $V_K$ that approximately solve the Schr\"odinger system. Mimicking the relation \eqref{eqn: SB_primal_sol} between the Schr\"odinger potentials and the primal optimal solution, we use the following ansatz:
\begin{align}\label{eqn: pi ub}
\pi_{\rm ub}(x, y) = \m K_\varepsilon(x-y)\big[\rho_0(x)\rho_\varepsilon(y) U_K(x)V_K(y) + r_U^\varepsilon(x) + r_V^\varepsilon(y)\big]
\end{align}
for which the augmenting terms $r_U^\varepsilon(x)$, $r_V^\varepsilon(y)$ are to satisfy the marginal conditions
\begin{align}\label{eqn: pi ub marginal}
\int_{\mb T^d} \pi_{\rm ub}(x, y)\,\dd x = \rho_\varepsilon(y)\quad\mx{and}\quad \int\pi_{\rm ub}(x, y)\,\dd y = \rho_0(x).
\end{align}
From the definitions of $R_\varepsilon, Q_\varepsilon$ in~\eqref{eqn: remainder}, 
the marginal conditions above are satisfied if $(r_U^\varepsilon, r_V^\varepsilon)$ solves the following system of equations:
\begin{subequations}\label{eqn: Fredholm}
\begin{align}
\m K_\varepsilon\ast r_U + r_V &= 
R_\varepsilon\label{eqn: Fredholm_R},\\
\m K_\varepsilon\ast r_V + r_U &= 
Q_\varepsilon.\label{eqn: Fredholm_Q}
\end{align}
\end{subequations}
Notice that when $\varepsilon =0$, $R_0=Q_0=0$ so \eqref{eqn: Fredholm} has a trivial solution.
This system is solved in the following lemma,  with its proof deferred to Section~\ref{sec: exist_Fredholm}.

\begin{lemma}\label{lem: exist_Fredholm}
If $U_K, V_K\in\m C(\mb T^d)$, then for every $\varepsilon\geq 0$ there exists a unique (up to additive constant) solution $(r_U^\varepsilon, r_V^\varepsilon)\in\m C(\mb T^d)\times \m C(\mb T^d)$ to~\eqref{eqn: Fredholm}.
Moreover, we have
\begin{itemize}

\item When $\varepsilon\leq 2$, it holds that
\begin{align*}
\|r_U^\varepsilon\|^2_{L^2(\mb T^d)} + \|r_V^\varepsilon\|^2_{L^2(\mb T^d)} &\leq \frac{16}{\varepsilon^2}\Big[\|R_\varepsilon\|_{L^2(\mb T^d)}^2 + \|Q_\varepsilon\|_{L^2(\mb T^d)}^2\Big].
\end{align*}

\item 
Suppose $K\geq 1$. For every $\delta > 0$, there exists $\varepsilon_r > 0$ such that $\|r_U^\varepsilon\|_{\m C^0} + \|r_V^\varepsilon\|_{\m C^0}\leq \delta$ whenever $\varepsilon \in[0, \varepsilon_r]$. Moreover, $\varepsilon_r$ depends only on $\delta$, $\sup\big\{\|\rho_t^{(k)}\|_{H^{2k'+s}}: t\in[0,\varepsilon], \,\,k,k'\in\mb N,\,\, k+k'\leq 2\big\}$, $\|u_0\|_{H^{s+4}}, \cdots, \|u_K\|_{H^{s+4}}$, and $\|v_0\|_{H^{s+4}}, \cdots, \|v_K\|_{H^{s+4}}$, where $s = \lceil\frac{d+1}{2}\rceil$.
\end{itemize}
\myqed
\end{lemma}

\paragraph{Proof of the upper bound~\eqref{eqn: eot_ub}.}
Note that $U_K(x) \to u_0(x) \geq \min_{x\in\mb T^d} u_0(x) \eqqcolon L_u > 0$ as $\varepsilon\to 0^+$, and $u_0, \cdots, u_K$ are continuous on $\mb T^d$. Therefore, there exists $\varepsilon_U > 0$ depending only on $L_u$ and the $\m C^0(\mb T^d)$-norm of $u_1, \cdots, u_K$, such that $U_K(x) > \frac{1}{2}L_u > 0$ uniformly holds for all $x\in\mb T^d$ and $\varepsilon\in[0, \varepsilon_U]$. Similarly, let $L_v\coloneqq \min_{x\in\mb T^d} v_0(x) > 0$. Then, there exists $\varepsilon_V > 0$ only depending on $L_v$ and the $\m C^0(\mb T^d)$-norm of $v_1, \cdots, v_K$, such that $V_K(x) > \frac{1}{2}L_v > 0$ uniformly holds for all $x\in\mb T^d$ and $\varepsilon\in[0, \varepsilon_V]$.

By Lemma~\ref{lem: exist_Fredholm}, there exists functions $r_U^\varepsilon$ and $r_V^\varepsilon$ satisfying the equation set~\eqref{eqn: Fredholm}. 
Then $\pi_{\rm ub}$ given as in \eqref{eqn: pi ub} satisfies the marginal conditions \eqref{eqn: pi ub marginal}, i.e. $\pi_{\rm ub}\in\Pi(\rho_0, \rho_\varepsilon)$, thus it is an admissible solution to the primal problem of $\eot_\varepsilon$.
Moreover, from the same lemma, there exists $\varepsilon_0 > 0$ such that $$\sup_{x,y\in\mb T^d}|r_U^\varepsilon(x) + r_V^\varepsilon(y)| \leq \frac{1}{8}L_\rho^2 L_uL_\nu$$ uniformly holds for all $\varepsilon\in[0, \varepsilon_0]$. Moreover, $\varepsilon_0$ only depends on $L_\rho$, $L_u$, $L_v$, $\sup\big\{\|\rho_t^{(k)}\|_{H^{2k'+s}}: t\in[0,\varepsilon], \,\,k,k'\in\mb N,\,\, k+k'\leq 2\big\}$, $\|u_0\|_{H^{s+4}}, \cdots, \|u_K\|_{H^{s+4}}$, and $\|v_0\|_{H^{s+4}}, \cdots, \|v_K\|_{H^{s+4}}$, where $s = \lceil\frac{d+1}{2}\rceil$. 

Therefore, for every $\varepsilon \leq \varepsilon_\pi\coloneqq\min\{\varepsilon_0, \varepsilon_U, \varepsilon_V\}$, it holds that
\begin{align*}
\rho_0(x)\rho_\varepsilon(y)U_K(x)V_K(y) + r_U^\varepsilon(x) + r_V^\varepsilon(y) 
\geq \frac{1}{4}L_\rho^2L_uL_v - \frac{1}{8}L_\rho^2L_uL_v = \frac{1}{8}L_\rho^2L_uL_v > 0.
\end{align*}
Note that $\m K_\varepsilon(x-y) > 0$, so we have 
\begin{align*}
    \hbox{$\pi_{\rm ub}(x, y) > 0$ uniformly holds for all $(x, y, \varepsilon)\in\mb T^d\times\mb T^d\times [0,\varepsilon_\pi]$.}
\end{align*}
Now, let us evaluate its corresponding EOT cost. It is easy to check that
\begin{align*}
\int_{\mb T^d\times\mb T^d} c_\varepsilon(x, y)\,\dd\pi + \varepsilon\KL(\pi\,\|\,\rho_0\otimes\rho_\varepsilon) = \varepsilon\KL(\pi\,\|\,\pi_{\rho_0, \varepsilon}) - \varepsilon\int_{\mb T^d}\rho_\varepsilon\log\rho_\varepsilon
\end{align*}
holds for every $\pi\in\Pi(\rho_0, \rho_\varepsilon)$, where $\pi_{\rho_0, \varepsilon}(x, y) \coloneqq \rho_0(x)\m K_\varepsilon(y-x)\in\ms P_{ac}(\mb T^d\times\mb T^d)$. Therefore, we have
\begin{align}\label{eqn: temp_ub1}
\eot_\varepsilon(\rho_0, \rho_\varepsilon) \leq \varepsilon\KL(\pi_{\rm ub}\,\|\,\pi_{\rho_0,\varepsilon}) - \varepsilon\int_{\mb T^d}\rho_\varepsilon\log\rho_\varepsilon. 
\end{align}
To evaluate the KL divergence term, note that
\begin{align*}
&\quad\,\KL(\pi_{\rm ub}\,\|\,\pi_{\rho,\varepsilon})
= \mb E_{(X, Y)\sim \pi_{\rm ub}}\Big[\log\frac{\rho_0(X)\rho_\varepsilon(Y)\m K_\varepsilon(X-Y)\big[U_K(X)V_K(Y) + \frac{r_U^\varepsilon(X) + r_V^\varepsilon(Y)}{\rho_0(X)\rho_\varepsilon(Y)}\big]}{\rho_0(X)\m K_\varepsilon(X-Y)}\Big]\\
&=\mb E_{\pi_{\rm ub}}[\log\rho_\varepsilon(Y)] + \mb E_{\pi_{\rm ub}}[\log U_K(X)] + \mb E_{\pi_{\rm ub}}[\log V_K(Y)] + \mb E_{\pi_{\rm ub}}\Big[\log\Big(1 + \frac{r_U^\varepsilon(x) + r_V^\varepsilon(y)}{\rho_0(X)\rho_\varepsilon(Y)U_K(X)V_K(Y)}\Big)\Big]\\
&\leq \int_{\mb T^d}\rho_\varepsilon\log\rho_\varepsilon + \int_{\mb T^d}\log U_K\,\dd\rho_0 + \int_{\mb T^d}\log V_K\,\dd\rho_\varepsilon + \int_{\mb T^d} \frac{r_U^\varepsilon(x)+r_V^\varepsilon(y)}{\rho_0(x)\rho_\varepsilon(y) U_K(x)V_K(y)}\,\dd\pi_{\rm ub}.
\end{align*}
In the last line, we use the fact that $\pi_{\rm ub}\in\Pi(\rho_0, \rho_\varepsilon)$ and $\log(1+z) \leq z$ for every $z > -1$. Combining with~\eqref{eqn: temp_ub1}, we have
\begin{align}\label{eqn: eotub}
\eot_\varepsilon(\rho_0, \rho_\varepsilon)
\leq \varepsilon\int_{\mb T^d}\log U_K\,\dd\rho_0 + \varepsilon\int_{\mb T^d}\log V_K\,\dd\rho_\varepsilon + \varepsilon\int_{\mb T^d} \frac{r_U^\varepsilon(x)+r_V^\varepsilon(y)}{\rho_0(x)\rho_\varepsilon(y) U_K(x)V_K(y)}\,\dd\pi_{\rm ub}.
\end{align}
Since $L_\rho\leq \rho_0(x), \rho_\varepsilon(y)\leq L_\rho^{-1}$,  $U_K(x) \geq \frac{1}{2}L_u$, and $V_K(y) \geq \frac{1}{2}L_v$, we have
\begin{align*}
\int_{\mb T^d}\frac{r_U^\varepsilon(x)+r_V^\varepsilon(y)}{\rho_0(x)\rho_\varepsilon(y) U_K(x)V_K(y)}\,\dd\pi_{\rm ub}
&\leq \int\Big|\frac{r_U^\varepsilon(x) + r_V^\varepsilon(y)}{L_\rho^2\cdot \frac{1}{2}L_u\cdot\frac{1}{2}L_v}\Big|\,\dd\pi_m \leq \frac{4L_uL_v}{L_\rho^2}\Big[\|r_U^\varepsilon\|_{L^2(\rho_0)} + \|r_V^\varepsilon\|_{L^2(\rho_\varepsilon)}\Big]\\
&\leq \frac{4L_\rho^{-1} L_uL_v}{L_\rho^2}\Big[\|r_U^\varepsilon\|_{L^2(\mb T^d)} + \|r_V^\varepsilon\|_{L^2(\mb T^d)}\Big]\\
&\stackrel{\ri}{\leq} \frac{16\sqrt{2} L_uL_v}{\varepsilon L_\rho^3} \sqrt{\|R_\varepsilon\|_{L^2(\mb T^d)}^2 + \|Q_\varepsilon\|_{L^2(\mb T^d)}^2}\\
&\stackrel{\rii}{=} O(\varepsilon^K).
\end{align*}
Here, (i) follows from Lemma~\ref{lem: exist_Fredholm}, and (ii) follows from Lemma~\ref{lem: int_remainder}; the big-O notation omits a constant only depending on $\|u_0\|_{\m C^0}, \cdots \|u_K\|_{\m C^0}$, $\|u_0\|_{H^{2K+2}}, \cdots, \|u_K\|_{H^{2K+2}}$, $\|v_0\|_{\m C^0}, \cdots, \|v_K\|_{\m C^0}$, $\|v_0\|_{H^{2K+2}}, \cdots, \|v_K\|_{H^{2K+2}}$, $L_\rho, L_u, L_v$, and $\sup\big\{\|\rho_t^{(k)}\|_{\m C^{2k'}}: 0\leq t\leq\varepsilon, \,\, k,k'\in\mb N,\,\,k+k'\leq K+1\big\}$. Combining with~\eqref{eqn: eotub}, the proof is completed.
\qed

\subsection{Proof of Theorem~\ref{thm: existence}}\label{sec: existence}
To establish the result, we first need the following version of the Schauder estimates of divergence form elliptic PDEs; see \cref{sec:Schauder-Lemma} for a proof. 
\begin{lemma}\label{lem: Schauder}
Let $\alpha\in(0, 1)$ and $s\in\mb N$ be two given constants.
Assume $\rho_0\in\m C^{s+1,\alpha}(\mb T^d)$ and satisfies $0< L_\rho \leq \rho_0(x) \leq L_\rho^{-1}$ with some constant $L_\rho\in(0, 1)$. Suppose $f\in \m C^{s,\alpha}(\mb T^d)$ and satisfies $\int_{\mb T^d}f\,\dd x = 0$. Then, the following elliptic PDE in divergence form 
\begin{align*}
\nabla\cdot(\rho_0\nabla u) = f
\end{align*}
has a (unique) solution $u\in\m C^{s+2,\alpha}(\mb T^d)$ such that $\int_{\mb T^d}u\,\dd x = 0$. Moreover, there exists a constant $C = C(\alpha, s, d, L_\rho, \|\rho_0\|_{\m C^{s+1,\alpha}}) > 0$, such that
\begin{align*}
\|u\|_{\m C^{s+2,\alpha}} \leq C\|f\|_{\m C^{s,\alpha}}.
\end{align*}
\end{lemma}

We are now ready to prove the existence of the series of equations~\eqref{eqn: asymp-Schro}.

\begin{proof}[\bf Proof of Theorem~\ref{thm: existence}]
The proof is based on induction. First we need to collect some identities. A point of these identities is to express objects with index $k$ in terms of objects with index less than $k$.
We especially use the notation 
\begin{align}\label{eqn:notation-Delta-s}
f \trieq{s} g
\end{align}
to indicate that $f-g$ is a function that does not depend on $\{u_i\}_{i\geq s}$, $\{v_i\}_{i\geq s}$, $\{u_i^\dagger\}_{i\geq s}$, and $\{v_i^\dagger\}_{i\geq s}$; this will be useful in our induction argument.

For any integer $k\geq 1$, Equation~\eqref{eqn: iter_fourier} implies
\begin{align}
\begin{aligned}\label{eqn: defAB}
v_k^\dagger - u_k\rho_0 &= \sum_{l=0}^{k-1} \frac{\Delta^{k-l}(u_l\rho_0)}{2^{k-l}(k-l)!} \eqqcolon A_{k-1},\\
u_k^\dagger - v_k\rho_0 &= \sum_{l=0}^{k-1}\sum_{i=0}^l \frac{\Delta^{k-l}(v_i\rho_{l-i})}{2^{k-l}(k-l)!} + \sum_{i=0}^{k-1}v_i\rho_{k-i} \eqqcolon B_{k-1},
\end{aligned}
\end{align}
and Equation~\eqref{eqn: iter_dagger} implies
\begin{align*}
u_k + u_0^2u_k^\dagger &= -u_0\sum_{i=1}^{k-1}u_{k-i}u_i^\dagger \eqqcolon C_{k-1},\\
v_k + v_0^2v_k^\dagger &= -v_0\sum_{i=1}^{k-1}v_{k-i}v_i^\dagger \eqqcolon D_{k-1}.
\end{align*}
It is also easy to check that
\begin{align}\label{eqn: bridge}
v_0^\dagger D_{k-1} - v_0A_{k-1}
= v_0^\dagger v_k + u_0^\dagger u_k = u_0^\dagger C_{k-1} - u_0B_{k-1}
\end{align}
holds for all $k\geq 1$, and
\begin{align*}
A_0 = \frac{\Delta(u_0\rho_0)}{2},
\quad
B_0 = \frac{\Delta(v_0\rho_0)}{2} + v_0\rho_1,
\quad
C_0 = D_0 = 0.
\end{align*}

In the notation in \eqref{eqn:notation-Delta-s}, the equation \eqref{eqn: bridge} implies $$v_0^\dagger v_k \trieq{k} - u_0^\dagger u_k,$$ since $A_{k-1}, B_{k-1}, C_{k-1}, D_{k-1}$ do not depend on $u_k, v_k, u_k^\dagger, v_k^\dagger$.

Also, we shall note that
\begin{align*}
\sum_{i=0}^k v_{k-i} A_i
&= \sum_{i=0}^k v_{k-i}(v_{i+1}^\dagger - u_{i+1}\rho_0)
= \sum_{i=0}^k v_{k-i}v_{i+1}^\dagger - \rho_0\sum_{i=0}^k v_{k-i}u_{i+1}\\
&\stackrel{\ri}{=} -v_0^\dagger v_{k+1} - \rho_0 \sum_{i=0}^{k+1}v_iu_{k-i+1} + \rho_0 u_0 v_{k+1}\\
&= -\rho_0\sum_{i=0}^{k+1}v_iu_{k-i+1}.
\end{align*}
Similarly, we can get
\begin{align}\label{eqn: uB=vA}
\sum_{i=0}^k u_{k-i}B_i = -\rho_0\sum_{i=0}^{k+1} v_iu_{k-i+1} = \sum_{i=0}^k v_{k-i}A_i.
\end{align}

Now, we are ready to prove the statement.

\vspace{0.5em}
\noindent\underline{Step 1: The equalities hold for $k=1$.} In this step, we will show
\begin{align}\label{eqn: PDEu0}
\nabla\cdot(\rho_0\nabla \log u_0) = \rho_1 - \frac{1}{2}\Delta\rho_0,
\end{align}
and use this PDE to derive the regularity of $u_0$ and $v_0$.

Recall that we have
\begin{align}
&u_1 = -u_0^2 u_1^\dagger,
\qquad 
v_1 = -v_0^2 v_1^\dagger,\notag\\
&u_1^\dagger = \frac{\Delta(v_0\rho_0)}{2} + v_0\rho_1 + v_1\rho_0,
\qquad
v_1^\dagger = \frac{\Delta(u_0\rho_0)}{2} + u_1\rho_0\label{eqn: uv1dagger}.
\end{align}
First, we will show
\begin{align}\label{eqn: init_dagger}
u_0^\dagger \Delta v_0^\dagger - v_0^\dagger\Delta u_0^\dagger &= 2\rho_1
\quad\,\,\mx{and}\quad\,\,
u_0^\dagger v_0^\dagger = \rho_0
\end{align}
with $\rho_0 > 0$ and $\int_{\mb T^d}\rho_1 = 0$.

Note that $u_0u_0^\dagger = v_0v_0^\dagger = 1$. Therefore, $\rho_0u_0v_0 = 1$ directly implies $u_0^\dagger v_0^\dagger = \rho_0$. Meanwhile, we get $u_0\rho_0 = v_0^\dagger$ and $v_0\rho_0 = u_0^\dagger$, and 
\begin{align*}
\rho_0u_1 = -\rho_0u_0^2u_1^\dagger = -u_0v_0^\dagger u_1^\dagger
\quad\mx{and}\quad
\rho_0v_1 =- \rho_0v_0^2 v_1^\dagger = -v_0u_0^\dagger v_1^\dagger.
\end{align*}
Now, by the definition of $u_1^\dagger$ and $v_1^\dagger$ in~\eqref{eqn: uv1dagger}, we get
\begin{align*}
v_0^\dagger u_1^\dagger
&= v_0^\dagger\Big[\frac{\Delta u_0^\dagger}{2} + v_0\rho_1 - v_0u_0^\dagger v_1^\dagger\Big]
= \frac{v_0^\dagger\Delta u_0^\dagger}{2} + \rho_1 - u_0^\dagger v_1^\dagger,\\
u_0^\dagger v_1^\dagger 
&= u_0^\dagger \Big[\frac{\Delta v_0^\dagger}{2} - u_0 v_0^\dagger u_1^\dagger\Big]
= \frac{u_0^\dagger \Delta v_0^\dagger}{2} - v_0^\dagger u_1^\dagger.
\end{align*}
Subtracting the first equation with the second equation yields~\eqref{eqn: init_dagger}.

Next, we will use Equation~\eqref{eqn: init_dagger} to get the elliptic PDE~\eqref{eqn: PDEu0}. Note that
\begin{align}\label{eqn:logu0d}
\nabla\cdot(\rho_0\nabla\log u_0^\dagger)
&= \nabla\cdot\Big(\rho_0 \frac{\nabla u_0^\dagger}{u_0^\dagger}\Big)
= \nabla\cdot(v_0^\dagger \nabla u_0^\dagger).
\end{align}
In the last equality, we use the fact that $\rho_0 / u_0^\dagger = \rho_0 u_0 = v_0^\dagger$. Similarly, we get
\begin{align*}
\nabla\cdot(\rho_0\nabla\log v_0^\dagger) = \nabla\cdot(u_0^\dagger \nabla v_0^\dagger).
\end{align*}
On the other hand, we have
\begin{align*}
\nabla\cdot(\rho_0\nabla\log v_0^\dagger)
= \nabla\cdot(\rho_0\nabla\log(u_0\rho_0)) = \nabla\cdot(\rho_0\nabla\log u_0) + \Delta\rho_0.
\end{align*}
Using $\log u_0 = -\log u_0^\dagger$, we have
\begin{align*}
\nabla\cdot(\rho_0\nabla\log u_0^\dagger) 
&= \Delta\rho_0 - \nabla\cdot(\rho_0\nabla\log v_0^\dagger)
= \Delta\rho_0 - \nabla\cdot(u_0^\dagger\nabla v_0^\dagger).
\end{align*}
Combining with~\eqref{eqn:logu0d}, we get
\begin{align*}
\nabla\cdot(\rho_0\nabla\log u_0^\dagger) 
&= \frac{1}{2}\Big[\Delta\rho_0 - \nabla\cdot(u_0^\dagger\nabla v_0^\dagger)\Big] + \frac{1}{2}\nabla\cdot(v_0^\dagger\nabla u_0^\dagger)\\
&= \frac{1}{2}\Delta\rho_0 + \frac{1}{2}\nabla\cdot(v_0^\dagger \nabla u_0^\dagger - u_0^\dagger \nabla v_0^\dagger)\\
&=\frac{1}{2}\Delta\rho_0 + \frac{1}{2}\Big[v_0^\dagger\Delta u_0^\dagger - u_0^\dagger \Delta v_0^\dagger\Big]\\
&= \frac{1}{2}\Delta\rho_0  - \rho_1.
\end{align*}
Then applying again $\log u_0 = -\log u_0^\dagger$ implies the elliptic PDE~\eqref{eqn: PDEu0}.

Since $\int_{\mb T^d}\frac{1}{2}\Delta\rho_0 - \rho_1 = 0$, the solution $u_0$ of the above equation exists. Moreover, for any $s\in\mb N^\ast$ and $\alpha\in(0, 1)$, Lemma~\ref{lem: Schauder} implies that
\begin{align*}
\|\log u_0\|_{\m C^{s+2,\alpha}} \leq C(\alpha, s, d, L_\rho, \|\rho_0\|_{\m C^{s+1,\alpha}})\Big[\|\rho_1\|_{\m C^{s,\alpha}} + \frac{1}{2}\|\rho_0\|_{\m C^{s+2,\alpha}}\Big].
\end{align*}
holds for some constant $C(\alpha, s, d, L_\rho, \|\rho_0\|_{\m C^{s,\alpha}}) > 0$. The same upper bound also holds for $\|\log v_0\|$. Then, we know $u_0, v_0\in\m C^{s+2,\alpha}$ also holds for similar estimates with a difference constant.

\vspace{0.5em}
\noindent\underline{Step 2: The equalities hold for $1, 2, \cdots, k+1$.}
The induction hypothesis indicates the existence of $\{u_n\}_{n=0}^{k-1}$, $\{v_n\}_{n=0}^{k-1}$, $\{u_n^\dagger\}_{n=0}^{k-1}$, and $\{v_n^\dagger\}_{n=0}^{k-1}$. Now, we only need to prove the existence of $u_k$ and $v_k$, after which we know $v_k^\dagger = u_k\rho_0 + A_{k-1}$ and $u_k^\dagger = v_k\rho_0 + B_{k-1}$ automatically. The key is to establish the following elliptic PDE for $u_k$:
\begin{align}\label{eqn: uk_PDE}
\nabla\cdot\big(\rho_0\nabla(u_ku_0^\dagger)\big) =  \rho_0 u_0^\dagger F_{k-1},
\end{align}
where $F_{k-1}$ is a function depending only on $\{u_n\}_{n=0}^{k-1}$, $\{v_n\}_{n=0}^{k-1}$, $\{u_n^\dagger\}_{n=0}^{k-1}$, and $\{v_n^\dagger\}_{n=0}^{k-1}$.

From Equation~\eqref{eqn: uB=vA}, we have $\sum_{n=0}^k u_n B_{k-n} = \sum_{n=0}^k v_n A_{k-n}$. Note that the left-hand side is
\begin{align*}
\lhs &\trieq{k} u_0 B_k + u_kB_0
\trieq{k} u_0\Big[\frac{\Delta(\rho_0 v_k)}{2} + \rho_1 v_k\Big]  + u_kB_0\\
&\trieq{k} u_0\Big[\frac{\Delta(u_0^\dagger v_0^\dagger v_k)}{2} + \rho_1 v_0 v_0^\dagger v_k\Big] + u_kB_0\\
&\trieq{k} -u_0\Big[\frac{\Delta(u_0^\dagger u_0^\dagger u_k)}{2} + \rho_1v_0 u_0^\dagger u_k\Big] + u_k B_0\\
&\trieq{k} -\frac{u_0^\dagger}{2}\Delta u_k - 2\langle\nabla u_0^\dagger, \nabla u_k\rangle + \Big[B_0 - \rho_1 v_0 - \frac{u_0}{2}\Delta(u_0^\dagger)^2\Big]u_k.
\end{align*}
In these equalities, we use the facts that $v_0^\dagger v_k \trieq{k} - u_0^\dagger u_k$ derived from~\eqref{eqn: bridge}, $\rho_0 = u_0^\dagger v_0^\dagger$ and $u_0u_0^\dagger = v_0v_0^\dagger = 1$. Similarly, we have
\begin{align*}
\rhs &\trieq{k} v_0A_k + v_k A_0
\trieq{k} v_0 \cdot\frac{\Delta(u_k\rho_0)}{2} + A_0v_0v_0^\dagger v_k
\trieq{k} \frac{v_0\Delta(u_k\rho_0)}{2} - A_0v_0 u_0^\dagger u_k.
\end{align*}
Therefore, the above arguments imply
\begin{align*}
0 &\trieq{k} \frac{v_0\Delta(u_k\rho_0)}{2} - A_0v_0 u_0^\dagger u_k - \bigg[
-\frac{u_0^\dagger}{2}\Delta u_k - 2\langle\nabla u_0^\dagger, \nabla u_k\rangle + \Big[B_0 - \rho_1 v_0 - \frac{u_0}{2}\Delta(u_0^\dagger)^2\Big]u_k\bigg]\\
&\trieq{k} u_0^\dagger\Delta u_k + \langle 2\nabla u_0^\dagger + v_0\nabla\rho_0, \nabla u_k\rangle + \Big[\frac{v_0\Delta\rho_0}{2} - A_0v_0 u_0^\dagger - B_0 + \rho_1v_0 + \frac{u_0}{2}\Delta(u_0^\dagger)^2\Big]u_k\\
&\trieq{k} u_0^\dagger\Big[\Delta u_k + \langle 2\nabla\log u_0^\dagger + \nabla\log\rho_0, \nabla u_k\rangle + \frac{1}{\rho_0}\Big(\frac{\Delta\rho_0}{2} - A_0 u_0^\dagger - B_0v_0^\dagger + \rho_1 + \frac{u_0^2\rho_0}{2}\Delta(u_0^\dagger)^2\Big)u_k\Big]\\
&\trieq{k} u_0^\dagger\Big[\Delta u_k + \langle 2\nabla\log u_0^\dagger + \nabla\log\rho_0, \nabla u_k\rangle + \Big(\Delta\log u_0^\dagger + \|\nabla\log u_0^\dagger\|^2 + \langle\nabla\log\rho_0, \nabla\log u_0^\dagger\rangle\Big)u_k\Big]
\end{align*}
Define the elliptic operator
\begin{align*}
\m L u_k &\coloneqq \Delta u_k + \langle 2\nabla\log u_0^\dagger + \nabla\log\rho_0, \nabla u_k\rangle + \Big(\Delta\log u_0^\dagger + \|\nabla\log u_0^\dagger\|^2 + \langle\nabla\log\rho_0, \nabla\log u_0^\dagger\rangle\Big)u_k\\
&= u_0\big[\Delta(u_ku_0^\dagger) + \langle\nabla\log\rho_0, \nabla(u_ku_0^\dagger)\rangle\big]
= \frac{u_0}{\rho_0}\nabla\cdot(\rho_0\nabla(u_ku_0^\dagger)).
\end{align*}

Then, the above argument indicates that Equation~\eqref{eqn: uB=vA} implies $u_0^\dagger \m L u_k \trieq{k} 0$. Thus, we get 
$$\m L u_k \trieq{k} 0.$$ 
By the notation $\trieq{k}$ defined in~\eqref{eqn:notation-Delta-s}, it means the existence of a function $F_{k-1}$, depending only $\{u_i\}_{i<k}$, $\{v_i\}_{i<k}$, $\{u_i^\dagger\}_{i<k}$, and $\{v_i^\dagger\}_{i<k}$, such that
\begin{align*}
\m L u_k = F_{k-1}.
\end{align*}
Equation~\eqref{eqn: uB=vA} implies that 
\begin{align}\label{eqn: expression-Fk1}
F_{k-1} \coloneqq \m L u_k + u_0\Big[\sum_{n=0}^k u_{k-n}B_n - \sum_{n=0}^k v_{k-n}A_n\Big].
\end{align}
Then, using the definition of the operator $\m L$, $u_k$ is the solution to the following strictly elliptic PDE:
\begin{align}
\nabla\cdot\big(\rho_0\nabla(u_ku_0^\dagger)\big) = \frac{\rho_0 F_{k-1}}{u_0} = \rho_0 u_0^\dagger F_{k-1}.
\end{align}
This is exactly Equation~\eqref{eqn: uk_PDE}.

The rest of the proof can be decomposed into two parts: existence of the solution $u_k$ and the estimation of its Holder norm.

\underline{Step 2.1: existence of solution.} To prove~\eqref{eqn: uk_PDE} has a solution, we only need to show that 
$$\int_{\mb T^d} \rho_0 u_0^\dagger F_{k-1}\,\dd x = 0.$$ 
By the expression of $F_{k-1}$ in Equation~\eqref{eqn: expression-Fk1}, we have
\begin{align*}
\int_{\mb T^d} \rho_0 u_0^\dagger F_{k-1}\,\dd x
= \big\langle \rho_0 u_0^\dagger, \m L u_k\rangle_{L^2(\mb T^d)} + \int_{\mb T^d}\rho_0\Big[\sum_{n=0}^k u_{k-n}B_n - \sum_{n=0}^k v_{k-n}A_n\Big]\,\dd x.
\end{align*}
Since $\langle\rho_0 u_0^\dagger, \m L u_k\rangle_{L^2(\mb T^d)} = \langle\m L^\ast(\rho_0 u_0^\dagger), u_k\rangle_{L^2(\mb T^d)} = 0$, we only need to prove that
\begin{align}\label{eqn: uB=vA-int}
\int_{\mb T^d}\rho_0\sum_{n=0}^k u_{k-n}B_n\,\dd x = \int_{\mb T^d}\rho_0\sum_{n=0}^k v_{k-n}A_n\,\dd x.
\end{align}
To prove Equation~\eqref{eqn: uB=vA-int}, first note that
\begin{align*}
B_n &= \sum_{l=0}^n \sum_{i=0}^l \frac{\Delta^{n-l+1}(v_i\rho_{l-i})}{2^{n-l+1}(n-l+1)!} + \sum_{i=0}^n v_i\rho_{n-i+1}\\
&= \sum_{i=0}^n \sum_{l=i}^n \frac{\Delta^{n-l+1}(v_{l-i}\rho_{i})}{2^{n-l+1}(n-l+1)!} + \sum_{i=1}^{n+1}\rho_i v_{n-i+1}.
\end{align*}
Therefore, we have
\begin{align*}
&\quad\,\int_{\mb T^d}\rho_0\sum_{n=0}^k u_{k-n}B_k \,\dd x
= \int_{\mb T^d} \rho_0\sum_{n=0}^k u_{k-n}\Big[\sum_{i=0}^n \sum_{l=i}^n \frac{\Delta^{n-l+1}(v_{l-i}\rho_{i})}{2^{n-l+1}(n-l+1)!} + \sum_{i=1}^{n+1}\rho_i v_{n-i+1}\Big]\,\dd x\\
&=\sum_{i=0}^k \sum_{n=i}^k\sum_{l=i}^n \int_{\mb T^d}\frac{\rho_0 u_{k-n}\Delta^{n-l+1}(v_{l-i}\rho_{i})}{2^{n-l+1}(n-l+1)!}\,\dd x + \sum_{i=1}^{k+1}\sum_{n=i-1}^k \int_{\mb T^d}\rho_0 u_{k-n} \rho_i v_{n-i+1}\,\dd x\\
&\stackrel{\ri}{=} \sum_{i=0}^k \sum_{n=i}^k\sum_{l=i}^n \int_{\mb T^d}\frac{v_{l-i}\rho_{i}\Delta^{n-l+1}(\rho_0 u_{k-n})}{2^{n-l+1}(n-l+1)!}\,\dd x + \sum_{i=1}^{k+1}\sum_{n=i-1}^k \int_{\mb T^d}\rho_0 u_{k-n} \rho_i v_{n-i+1}\,\dd x\\
&= \sum_{i=0}^k \sum_{l=i}^k \sum_{n=0}^{k-l} \int_{\mb T^d}\frac{v_{l-i}\rho_{i}\Delta^{k-l-n+1}(\rho_0 u_{n})}{2^{k-l-n+1}(k-l-n+1)!}\,\dd x + \sum_{i=1}^{k+1}\sum_{n=i-1}^k \int_{\mb T^d}\rho_0 u_{k-n} \rho_i v_{n-i+1}\,\dd x\\
&\stackrel{}{=} \sum_{i=0}^k \sum_{l=i}^k \sum_{n=0}^{k-l} \int_{\mb T^d}\frac{v_{l-i}\rho_{i}\Delta^{k-l-n+1}(\rho_0 u_{n})}{2^{k-l-n+1}(k-l-n+1)!}\,\dd x + \sum_{i=1}^{k+1}\sum_{n=i-1}^k \int_{\mb T^d}\rho_0 u_{k-n} \rho_i v_{n-i+1}\,\dd x\\
&\stackrel{\rii}{=} \sum_{i=0}^k \int_{\mb T^d}\rho_i \sum_{l=i}^k v_{l-i}A_{k-l}\,\dd x + \sum_{i=1}^{k+1}\sum_{n=i-1}^k \int_{\mb T^d}\rho_0 u_{k-n} \rho_i v_{n-i+1}\,\dd x.
\end{align*}
Here, (i) is due to integration by parts; (ii) is due to the definition of $A_{k-l}$.
Now, let us check the terms related to $\rho_i$ for $i = 0, 1, \cdots, k+1$ on the right-hand side.
The terms with $\rho_0$ is
\begin{align*}
\int_{\mb T^d} \rho_0 \sum_{l=0}^k v_l A_{k-l}\,\dd x.
\end{align*}
For $1\leq i\leq k$, the terms related to $\rho_i$ is
\begin{align*}
&\quad\,\int_{\mb T^d} \rho_i \sum_{l=i}^k v_{l-i} A_{k-l}\,\dd x + \sum_{n=i-1}^k\int_{\mb T^d} \rho_0 u_{k-n}\rho_i v_{n-i+1}\,\dd x
= \int_{\mb T^d}\rho_i\Big[\sum_{l=0}^{k-i}v_l A_{k-i-l} + \rho_0\sum_{n=0}^{k-i+1} v_n u_{k-i+1 - n}\Big]\,\dd x = 0,
\end{align*}
where the last equality is due to
\begin{align*}
\sum_{l=0}^{k-i} v_l A_{k-i-l}
&\stackrel{\ri}{=} \sum_{l=0}^{k-i} v_{k-i-l}[v_{l+1}^\dagger - u_{l+1}\rho_0] = \sum_{l=0}^{k-i} v_{l+1}^\dagger v_{k-i-l} - \rho_0 \sum_{l=0}^{k-i} u_{l+1} v_{k-i-l}\\
&\stackrel{\rii}{=} - v_0^\dagger v_{k-i+1} - \rho_0\sum_{l=0}^{k-i+1} v_l u_{k-i+1-l} + \rho_0 u_0 v_{k-i+1}\\
&\stackrel{\rii}{=} - \rho_0\sum_{l=0}^{k-i+1} v_l u_{k-i+1-l}.
\end{align*}
Here in the above equations, (i) follows from the expression of $A_{k-i-l}$ defined in~\eqref{eqn: defAB}; (ii) is due to $\sum_{l=0}^{k-i+1} v_l^\dagger v_{k-i+1 - l} = 0$ for $i\leq k$; (iii) is due to $\rho_0u_0 = 1/v_0 = v_0^\dagger$. 
The term with $\rho_{k+1}$ is
\begin{align*}
\int_{\mb T^d} \rho_0u_0\rho_{k+1}v_0\,\dd x
= \int_{\mb T^d} \rho_{k+1}\,\dd x
= 0.
\end{align*}
Therefore, we have shown the equality~\eqref{eqn: uB=vA-int}.

So far, we have shown the existence of $u_k$. Using~\eqref{eqn: bridge} and $u_0^\dagger > 0$, as well as the induction hypothesis, $v_k$ exists. 

\underline{Step 2.2: regularity of the solution.} Here, we only focus on the regularity of $u_k$, as we expect it has the same regularity as $v_k$, $u_k^\dagger$, and $v_k^\dagger$.
Using Lemma~\ref{lem: Schauder}, the solution to the PDE~\eqref{eqn: uk_PDE} satisfies
\begin{align*}
\|u_ku_0^\dagger\|_{\m C^{s+2,\alpha}} \leq C(\alpha, s,d,L_\rho,\|\rho_0\|_{\m C^{s+1,\alpha}})\Big\|\frac{\rho_0 F_{k-1}}{u_0}\Big\|_{\m C^{s,\alpha}}
\end{align*}
Some tedious calculations (see Appendix~\ref{app: calculation}) implies that~\eqref{eqn: uk_PDE} can also be formulated as
\begin{align}\label{eqn: xxx}
\begin{aligned}
\nabla\cdot\big(\rho_0\nabla(u_0^\dagger u_k)\big)
&= \rho_0\bigg[\sum_{i=1}^{k-1} u_{k-i}B_i - \sum_{i=1}^{k-1} v_{k-i}A_i\bigg] - \rho_0\bigg[v_0\sum_{l=0}^{k-1}\frac{\Delta^{k-l+1}(u_l\rho_0)}{2^{k-l+1}(k-l+1)!} - S_0S_{k-1}\bigg]\\
&\quad + v_0^\dagger\bigg[\frac{\Delta(u_0^\dagger S_{k-1})}{2} + v_0\rho_1S_{k-1} + \sum_{l=0}^{k-1}\sum_{i=0}^l \frac{\Delta^{k+1-l}(v_0\rho_{l-i})}{2^{k+1-l}(k+1-l)!} + \sum_{i=0}^{k-1}\frac{\Delta(v_i\rho_{k-i})}{2} + \sum_{i=0}^{k-1}v_i\rho_{k-i+1}\bigg].
\end{aligned}
\end{align}
Moreover, it can be shown (see Appendix~\ref{app: Holder norm}) that the $\m C^{s,\alpha}$-norm of the right-hand side can be bounded above by $d,s,k,\alpha$, $\|u_0\|_{\m C^{s+2k+2,\alpha}}, \|u_1\|_{\m C^{s+2k,\alpha}} \cdots, \|u_{k-1}\|_{\m C^{s+4,\alpha}}$, $\|v_0\|_{\m C^{s+2k,\alpha}}, \|v_1\|_{\m C^{s+2k-2,\alpha}}, \cdots, \|v_{k-1}\|_{\m C^{s+2,\alpha}}$, $\|v_1^\dagger\|_{\m C^{s+2,\alpha}}, \cdots, \|v_{k-1}^\dagger\|_{\m C^{s+2,\alpha}}$, $\|\rho_0\|_{\m C^{s+2k+2,\alpha}}, \|\rho_1\|_{\m C^{s+2k,\alpha}}, \cdots, \|\rho_{k+1}\|_{\m C^{s,\alpha}}$, $\|u_0^\dagger\|_{\m C^{s+2,\alpha}}, \|v_0^\dagger\|_{\m C^{s+2,\alpha}}$, and $\|v_0\|_{\m C^{s+2k+2,\alpha}}$.
Using induction hypothesis, all these terms can be bounded by a constant only depending on $d, s, k,\alpha$ and $\|\rho_0\|_{\m C^{s+2k+2,\alpha}}, \|\rho_1\|_{\m C^{s+2k,\alpha}}, \cdots, \|\rho_{k+1}\|_{\m C^{s,\alpha}}$.

Now, we have a control of $\|u_ku_0^\dagger\|_{\m C^{s+2,\alpha}}$. Then, the estimates of $\|u_k\|_{\m C^{s+2,\alpha}}$ follows from
\begin{align*}
\|u_k\|_{\m C^{s+2,\alpha}}
\leq C(d, s, k, \alpha)\|u_ku_0^\dagger\|_{\m C^{s+2,\alpha}}\|u_0\|_{\m C^{s+2,\alpha}}.
\end{align*}

Using induction, the statement is completed.
\end{proof}

\section{Quantitative stability of the Schr\"odinger Functional $\m I_{\mu, \nu}^\varepsilon$}\label{sec: stability}
In this section, we focus on deriving another key component for proving our main result, quantitative stability  of the Schr\"odinger functionals $\m I_{\mu, \nu}^\varepsilon$ and $\overline{\m I}_{\mu, \nu}^\varepsilon$ \eqref{eqn: dual functional}, namely, they are quadratically nondegenerate near the maximum.
Given a function $f$, its variance under a probability distribution $\mu\in\ms P(\mb T^d)$ is denoted as
\begin{align*}
\mb V_\mu(f) 
\coloneqq \int_{\mb T^d}\bigg(f - \int_{\mb T^d}f\,\dd\mu\bigg)^2\,\dd\mu
= \min_{c\in\mb R}\int_{\mb T^d} (f-c)^2\,\dd\mu.
\end{align*}
The goal of this section is to prove the following result.


\begin{theorem}\label{thm: stability}
Suppose $\mu(x)$ and $\nu(x)$ are probability density functions on $\mb T^d$. Assume there exist positive constants 
$m_\mu, m_\nu, M_\mu, M_\nu\in\mb R_+$ such that $\mu(x)\in[m_\mu, M_\mu]$ and $\nu(x)\in[m_\nu, M_\nu]$ hold for all $x\in\mb T^d$. 
Let $(\phi_\varepsilon, \psi_\varepsilon)$ be the Schr\"odinger potentials~\eqref{eqn: Schro_sys}.
Then, there exist positive constants $C_\nu$ and $C_\mu$ depending only on $m_\nu, M_\nu, d$ and $m_\mu, M_\mu, d$ respectively, such that for any functions $\chi,\eta\in \m C(\mb T^d)$, 
we have
\begin{subequations}\label{eqn: stability}
\begin{align}
\mb V_\mu(\chi) &\leq \Big(4\varepsilon + C_\nu e^{\frac{2R^2}{\alpha} + 8\pi^2 d} + \frac{8\osc(\chi)^2}{3\varepsilon}\Big)\big[\m I_{\mu, \nu}^\varepsilon[\phi_\varepsilon] - \m I_{\mu, \nu}^\varepsilon[\phi_\varepsilon + \chi]\big],\label{eqn: stability1}\\
\mb V_\nu(\eta) &\leq \Big(4\varepsilon + C_\mu e^{\frac{2R^2}{\alpha} + 8\pi^2 d} + \frac{8\osc(\eta)^2}{3\varepsilon}\Big)\big[\overline{\m I}_{\mu, \nu}^\varepsilon[\psi_\varepsilon] - \overline{\m I}_{\mu, \nu}^\varepsilon[\psi_\varepsilon + \eta]\big].\label{eqn: stability2}
\end{align}    
\end{subequations}
Here, $\osc(\chi) = \sup_x \chi(x) - \inf_x\chi(x)$ is the oscillation of $\chi$ on $\mb T^d$; $R$ 
and $\alpha$ are the  two universal positive constants introduced in Proposition~\ref{prop: 2norm_local_convex}.
\myqed
\end{theorem}


The estimate~\eqref{eqn: stability} can be interpreted as a quantitative concavity of the functional $\phi\mapsto \m I_{\mu, \nu}^\varepsilon[\phi]$, which then is governed by the behavior of the functional  $\phi\mapsto \mb E_\nu\big[\m T_\mu^\varepsilon [\phi]\big]$. \cite{delalande2023quantitative} first considered this kind of result in optimal transport, and it was later applied to the quantitative stability of Wasserstein barycenter by~\cite{carlier2024quantitative}. A crucial idea to prove the concavity is to decompose the variance into two components by using the law of total variance (c.f. see Theorem 3.27 in~\citep{wasserman2013all}), and to subsequently prove that the variance of conditional expectation can be bounded above by the expectation of conditional variance; this can be treated as the quantitative concavity of $\m T_\mu^\varepsilon$. 

In entropic optimal transport, \cite{delalande2022nearly} and~\cite{kitagawa2025stability} also studied the concavity of the functional $\phi\mapsto \m I_{\mu,\nu}^\varepsilon[\phi]$ for the semi-discrete case and the continuous case, respectively. However, these results on EOT only established the lower bound of the Hessian of the loss functionial at the Schr\"odinger potential (e.g. Theorem 3.2 in~\cite{delalande2022nearly}), which corresponds to $\nabla^2\m I_{\mu, \nu}^\varepsilon[\phi_\varepsilon]$ in our setting. In contrast, deriving the control~\eqref{thm: stability} in Theorem~\ref{thm: stability} requires evaluating the Hessian at all functions connecting $\phi_\varepsilon$ and $\phi_\varepsilon + \chi$.

The proof of~\eqref{eqn: stability} consists of two steps.
We first prove in Section~\ref{sec: local strong concavity}, a quantitative concavity estimate of 
$\m T_\mu^\varepsilon$, locally  
on balls with sufficiently small radius by using the Prekopa--Leindler's inequality, as demonstrated in~\citep{delalande2022nearly, kitagawa2025stability}.
Then,  in Section~\ref{sec: BomanChain}, we use the Boman chain condition, a recently developed technique~\citep{kitagawa2025stability, letrouit2024gluing}, to glue the local inequalities into a global inequality over the entire Riemannian manifold ($\mb T^d$ in our setting). 
At last, in Section~\ref{sec: pf_stability}, the resulting global quantitative concavity estimate is translated into \eqref{eqn: stability}. The proof follows the argument in~\citep{chizat2025sharper} to derive an upper bound on the expectation of the conditional variance in terms of the difference of the functional values $\m I_{\mu, \nu}^\varepsilon$ (and $\overline{\m I}_{\mu,\nu}^\varepsilon$).

\subsection{Local quantitative stability
}\label{sec: local strong concavity}
For a given geodesically convex subset $S\subset\mb T^d$ and $\alpha \geq 0$, a function $f$ is $\alpha$-semiconcave on $S$ if
\begin{align*}
f(z) \geq (1-\lambda) f(x_0) + \lambda f(x_1) - \frac{\alpha\lambda(1-\lambda)}{2}\td(x_0, x_1)^2
\end{align*}
holds for all $\lambda \in[0, 1]$, $x_0, x_1\in S$, and the geodesic interpolant 
\begin{align*}
z\in Z^S_\lambda(x_0, x_1) \coloneqq \big\{z\in S: \td(z, x_0) = \lambda\td(x_0, x_1),\,\, \td(z, x_1) = (1-\lambda)\td(x_0, x_1)\big\}.
\end{align*}
Similarly,  $f$ is $\alpha$-strongly convex on $S$ if
\begin{align*}
f(z) \leq (1-\lambda) f(x_0) + \lambda f(x_1) - \frac{\alpha\lambda(1-\lambda)}{2}\td(x_0, x_1)^2
\end{align*}
holds for all $\lambda \in [0, 1]$ and $z\in Z^S_\lambda(x_0, x_1)$.

It is well known that globally $\alpha$-strongly convex functions do not exist on $\mb T^d$ for $\alpha > 0$. However, the squared distance function on a Riemannian manifold is locally strongly convex on  balls with sufficiently small radius. 

\begin{proposition}[local strong convexity of squared distance]\label{prop: 2norm_local_convex}
There exists constants $\alpha, R > 0$ such that the map $x\mapsto\td(0_{d}, x)^2$ is $\alpha$-strongly convex for all $x\in B_{\mb T^d}(0_d; R)$.
\myqed
\end{proposition}

We also have semi-concavity of the cost function $c_\varepsilon$. The uniform semi-concavity can be derived by using the explicit expression of $c_\varepsilon$; in~\citep{hamilton1993matrix}, the estimate of the semi-concavity of the logarithmic Gaussian kernel is provided for all compact manifolds with non-negative curvature.
\begin{proposition}\label{prop: semi-concave}
The map $y\mapsto c_\varepsilon(x, y)$ is $1$-semiconcave for all $x\in\mb T^d$ and $\varepsilon > 0$.
\end{proposition}

With this, we can prove the following useful semicocavity type result for the operation $\phi \mapsto \m T_\mu^\varepsilon[\phi]$ defined in~\eqref{eqn: Schro_sys}.

\begin{proposition}\label{prop: T semiconcave}
For all $\lambda\in[0, 1]$ and $\phi,\wt\phi\in \m C(\mb T^d)$, 
$y_0, y_1 \in\mb T^d$ and  $z\in Z_\lambda^{\mb T^d}(y_0, y_1)$,
    \begin{align}\label{eqn: T semiconcave}
    \m T_\mu^\varepsilon[\lambda \phi + (1-\lambda)\wt\phi](z)
    \geq \lambda\m T_\mu^\varepsilon[\phi](y_1) + (1-\lambda)\m T_\mu^\varepsilon[\wt\phi](y_2)- \frac{\lambda(1-\lambda)}{2}\td(y_1, y_2)^2.
    \end{align}
\end{proposition}
\begin{proof}
The $1$-semiconcavity of the cost function $c_\varepsilon$ implies 
\begin{align*}
c_\varepsilon(x, z) \geq \lambda c_\varepsilon(x, y_1) + (1-\lambda)c_\varepsilon(x, y_2) - \frac{\lambda(1-\lambda)}{2}\td(y_1, y_2)^2.
\end{align*}
Therefore, we have
\begin{align*}
&\quad\,\m T_\mu^\varepsilon[\lambda \phi + (1-\lambda)\wt\phi](z)
= -\varepsilon\log\bigg(\int_{\mb T^d} e^{\frac{\lambda\phi(x) + (1-\lambda)\wt\phi(x) - c_\varepsilon(x, z)}{\varepsilon}}\,\dd\mu(x)\bigg)\\
&\stackrel{\ri}{\geq} -\varepsilon\log\bigg(\int_{\mb T^d} \exp\Big\{\frac{\lambda\phi(x) + (1-\lambda)\wt\phi(x) - \lambda c_\varepsilon(x, y_1) - (1-\lambda) c_\varepsilon(x, y_2) + \frac{\lambda(1-\lambda)}{2}\td(y_1, y_2)^2}{\varepsilon}\Big\}\,\dd\mu(x)\bigg)\\
&=-\varepsilon\log\bigg(\int_{\mb T^d}\exp\Big\{\frac{\lambda}{\varepsilon}\Big[\phi(x) - c_\varepsilon(x, y_1)\Big] + \frac{1-\lambda}{\varepsilon}\Big[\wt\phi(x) - c_\varepsilon(x, y_2)\Big]\Big\}\,\dd\mu(x)\bigg) - \frac{\lambda(1-\lambda)}{2}\td(y_1, y_2)^2\\
&\stackrel{\rii}{\geq} \lambda\m T_\mu^\varepsilon[\phi](y_1) + (1-\lambda)\m T_\mu^\varepsilon[\wt\phi](y_2)- \frac{\lambda(1-\lambda)}{2}\td(y_1, y_2)^2.
\end{align*}
Here, (i) follows from the $1$-semiconcavity of $c_\varepsilon(x, \cdot)$, and (ii) follows from the convexity of the LogSumExp function.
\end{proof}

With this we prove the following local quantitative stability result towards Theorem~\ref{thm: stability}.
\begin{proposition}[Local quantitative stability]\label{prop: variance lower bound}
For any $x_0\in\mb T^d$, suppose $Q = B_{\mb T^d}(x_0, r)$ is a ball on $\mb T^d$ centered at $x_0$ with radius $r < R$, where $R$ is the constant in Proposition~\ref{prop: 2norm_local_convex}. Let $\nu_Q \coloneqq \frac{\nu 1_Q}{\nu(Q)}$ be the truncated probability distribution of $\nu$ on $Q$. Assume that there are constants $m_\nu, M_\nu\in\mb R_+$ such that $\nu(x)\in[m_\nu, M_\nu]$ holds for all $x\in Q$. 
Then, for every function $\chi,\phi \in \mathcal{C}(\mb T^d)$  and $\lambda\in [0, 1]$,
we have
\begin{align}\label{eqn: local stability}
\mb E_{Y\sim\nu_Q}\Big[ \frac{\dd^2}{\dd\lambda^2} \m T_\mu^\varepsilon[\phi + \lambda\chi](Y) \Big]
 \le - C_1
 \cdot \mb V_{Y\sim\nu_Q}
 \Big[  \frac{\dd}{\dd\lambda} \m T_\mu^\varepsilon [\phi + \lambda\chi](Y) \Big],
\end{align}
where 
\begin{align*}
    C_1= C_1(m_\nu, M_\nu, r, \alpha, d)  := \frac{m_\nu e^{-\frac{2r^2}{\alpha} - 8\pi^2 d}}{ M_\nu}
\end{align*}
with $\alpha$ being  the constant in Proposition~\ref{prop: 2norm_local_convex}.
\end{proposition}

\begin{proof}
First, for every $x\in\mb T^d$, Proposition~\ref{prop: semi-concave} implies that $c_\varepsilon(x, \cdot)$ is $1$-semiconcave. With $\alpha$ from Proposition~\ref{prop: 2norm_local_convex}, define a probability density function 
$$
\gamma_Q(x) = Z_Q^{-1} e^{-\frac{1}{\alpha}\td(x_0, x)^2}1_Q(x)  
\quad \hbox{over the set $Q$}
$$ 
with $Z_Q = \int_Q e^{-\frac{1}{\alpha}\td(x_0, x)^2}\,\dd x$ being the normalizing constant.
We use $\m T_\mu^\varepsilon[\phi]$ in~\eqref{eqn: Schro_sys} to define a functional $\m G$ on $\mathcal{C}(\mb T^d)$,
\begin{align*}
\m G[\phi] \coloneqq \log \bigg(\int_{Q} e^{\m T_\mu^\varepsilon[\phi]}\,\dd\gamma_Q\bigg).
\end{align*}


\noindent\underline{Step 1: concavity of $\m G$.}
To prove $\m G$ is a concave functional, we only need to prove 
\begin{align*}
\m G[\lambda\phi + (1-\lambda)\wt\phi] \geq \lambda \m G[\phi] + (1-\lambda)\m G[\wt\phi]
\end{align*}
for all $\lambda\in[0, 1]$ and $\phi,\wt\phi\in \m C(\mb T^d)$. 
Define functions
\begin{align*}
    \hbox{
    $h(y) = e^{\m T_\mu^\varepsilon[\lambda\phi + (1-\lambda)\wt\phi](y)}$, $f(y) = e^{\m T_\mu^\varepsilon[\phi](y)}$, and $g(y) = e^{\m T_\mu^\varepsilon[\wt\phi](y)}$.
    }
\end{align*}
 Then, we only need to prove
\begin{align}\label{eqn: cond_prekopa}
\int_{Q} h\,\dd \gamma_Q \geq \bigg(\int_{Q} f\,\dd \gamma_Q\bigg)^\lambda\bigg(\int_{Q} g\,\dd\gamma_Q\bigg)^{1-\lambda}.
\end{align} 
Now, for any $z \in Z_\lambda(y_1, y_2)$, $\lambda\in[0, 1]$, and $y_1, y_2\in\mb T^d$,  taking exponential on both sides of \eqref{eqn: T semiconcave} yields
\begin{align}\label{eqn: ineq_hfg}
h(z) \geq e^{-\frac{\lambda(1-\lambda)}{2}\td(y_1, y_2)^2} f(y_1)^\lambda g(y_2)^{1-\lambda}.
\end{align}
Moreover, since $r < R$,
$Q$ is a geodesically convex subset. Therefore, applying the weighted Prekopa--Leindler's inequality (see e.g.~\citep[Theorem 2.9,][]{kitagawa2025stability} or~\citep[Theorem 1.4,][]{cordero2006prekopa}) leads to the desired inequality~\eqref{eqn: cond_prekopa}; notice that we used $\frac{1}{\alpha}\td(x_0, x)^2$ in the definition of $\gamma_Q$, which is locally $1$-strongly convex.

\vspace{0.5em}
\noindent\underline{Step 2: implication of concavity of $\m G$.}$\,\,$
Define a function $g_\chi(\lambda) \coloneqq \m G[\phi + \lambda\chi]$ for $\lambda\in[0, 1]$. Then, $g_\chi(\lambda)$ is concave with respect to $\lambda$ due to the concavity of the functional $\m G$. Now, define a probability distribution with the density function
\begin{align*}
\gamma_Q[\phi + \lambda\chi](y) \coloneqq \frac{\gamma_Q(y)e^{\m T_\mu^\varepsilon[\phi + \lambda\chi](y)}}{\int_{Q}e^{\m T_\mu^\varepsilon[\phi + \lambda\chi](y)}\,\dd\gamma_Q(y)} \in\ms P_{ac}(Q).
\end{align*}
With this notation, it is easy to check that
\begin{align*}
g_\chi''(\lambda) &= \mb V_{\gamma_Q[\phi + \lambda\chi]}\Big(\frac{\dd}{\dd\lambda}\m T_\mu^\varepsilon[\phi + \lambda\chi]\Big) + \mb E_{\gamma_Q[\phi + \lambda\chi]}\Big(\frac{\dd^2}{\dd\lambda^2}\m T_\mu^\varepsilon[\phi + \lambda\chi]\Big).
\end{align*}
Therefore, the concavity of $g_\chi$, that is, 
$g''_\chi  \le 0$,
implies 
\begin{align}\label{eqn: concave_g} 
  \mb E_{\gamma_Q[\phi + \lambda\chi]}\Big(\frac{\dd^2}{\dd\lambda^2}\m T_\mu^\varepsilon[\phi + \lambda\chi]\Big)
\le  - \mb V_{\gamma_Q[\phi + \lambda\chi]}\Big(\frac{\dd}{\dd\lambda}\m T_\mu^\varepsilon[\phi + \lambda\chi]\Big).
\end{align}
For any $\gamma_Q\in\ms P_{ac}(Q)$, we have
\begin{align*}
\mb E_{\gamma_Q}[f] = \mb E_{\nu_Q}\Big[\frac{\gamma_Q}{\nu_Q}f\Big]
\ge \inf_{y \in Q}\frac{\gamma_Q(y)}{\nu_Q (y)}  \mb E_{\nu_Q}(f)  
\end{align*}
for any nonnegative function $f$, and
\begin{align*}
\mb V_{\nu_Q}(f) = \min_{c\in\mb R}\int (f-c)^2\,\dd\nu_Q \leq 
 \sup_{y \in Q}\frac{\nu_Q(y)}{\gamma_Q (y)} 
\min_{c\in \mb R}\int(f-c)^2\,\dd\gamma_Q = 
\sup_{y \in Q}\frac{\nu_Q(y)}{\gamma_Q (y)}
\mb V_{\gamma_Q}(f)
\end{align*}
for any function $f$.
We also have the following result proved in Appendix~\ref{app: DensityRatioLB}.
\begin{lemma}\label{lem: DensityRatioLB}
For every $\phi \in \m C (\mb T^d)$, $y\in Q$,
it holds that
\begin{align*}
m_\nu e^{-\frac{r^2}{\alpha} - 4\pi^2 d}\cdot \frac{\vol(Q)}{\nu(Q)}
\leq
\frac{\nu_Q(y)}{\gamma_Q[\phi](y)} 
\leq M_\nu e^{\frac{r^2}{\alpha} + 4\pi^2 d}\cdot\frac{\vol(Q)}{\nu(Q)}.
\end{align*}
\end{lemma}
We know $\frac{\dd^2}{\dd\lambda^2}\m T_\mu^\varepsilon[\phi+\lambda\chi] \leq 0$ since $\m T_\mu^\varepsilon[\phi]$ is concave with respect to $\phi$ on $\m C(\mb T^d)$.
Combining these with \eqref{eqn: concave_g}, we get the desired inequality
\begin{align*}
  \mb E_{\nu_Q}\Big(\frac{\dd^2}{\dd\lambda^2}\m T_\mu^\varepsilon[\phi + \lambda\chi]\Big)
\le  - \frac{m_\nu e^{-\frac{2r^2}{\alpha} - 8\pi^2 d}}{ M_\nu} \mb V_{\nu_Q}\Big(\frac{\dd}{\dd\lambda}\m T_\mu^\varepsilon[\phi + \lambda\chi]\Big).
\end{align*}
This completes the proof.

\end{proof}

\subsection{Gluing local inequalities using Boman chain condition}\label{sec: BomanChain}
In this section, we introduce the Boman chain, recently developed in~\citep{kitagawa2025stability, letrouit2024gluing}, which plays a crucial role in the proof of the stability result (Theorem~\ref{thm: stability}). This technique allows us to glue the local inequality established in Proposition~\ref{prop: variance lower bound}, which only holds on balls with sufficiently small radius, into a global inequality valid over the entire space $\mb T^d$. To begin, we need the following definition of a doubling probability distribution.
\begin{defn}[doubling]
A probability distribution $\rho\in\ms P(\mb T^d)$ is \emph{doubling} with coefficient $c > 0$ if
\begin{align*}
\rho\big(B_{\mb T^d}(x, 2r)\big) \leq c \rho\big(B_{\mb T^d}(x, r)\big)
\end{align*}
holds for any $x\in\mb T^d$ and $r > 0$. 
\myqed
\end{defn}

The following result illustrates that probability distributions over $\mb T^d$ with positive upper and lower bound are doubling. We omit the proof as it is straightforward.

\begin{proposition}\label{prop: bd2doubling}
If a distribution $\rho\in\ms P_{ac}(\mb T^d)$ satisfying 
$0 < m_\rho \leq \rho(x) \leq M_\rho < \infty$ for some constants $m_\rho, M_\rho$, then $\rho$ is doubling with coefficient $\frac{2^dM_\rho}{m_\rho}$.
\end{proposition}

Next, we introduce the Boman chain condition on $\mb T^d$, of which a general version can be found in~\citep{kitagawa2025stability}.
\begin{defn}[Boman chain condition]\label{defn: Bowen}
A probability measure $\rho\in\ms P(\mb T^d)$ satisfies the \emph{Boman chain condition} with parameters $A, B, C > 1$, if there exists a covering $\m F$ of $\mb T^d$ by open balls such that:
\begin{itemize}
    \item for any $x\in\mb T^d$, $\sum_{Q\in\m F}1_{2Q}(x) \leq A$;
    
    \item given any fixed ball $Q_0\in\m F$, which is called the central ball, for every $Q\in\m F$, there exists a chain $Q_0, Q_1, \cdots, Q_N = Q$ of distinct balls from $\m F$ such that $Q\subset BQ_j$ holds for all $j\in\{0, 1, \cdots, N-1\}$;
    
    \item consecutive balls of the above chain satisfy
    \begin{align*}
        \rho(Q_j\cap Q_{j+1}) \geq C^{-1}\max\{\rho(Q_j), \rho(Q_{j+1})\},\quad\forall\, j\in\{0, 1, \cdots, N-1\}.
    \end{align*}
\end{itemize}
\myqed
\end{defn}

The following result plays a crucial role in gluing the local estimates of Proposition~\ref{prop: variance lower bound} to a global one, and it is a corollary of~\cite[Proposition 3.7, Lemma 3.8,][]{kitagawa2025stability}, together with the doubling coefficient from Proposition~\ref{prop: bd2doubling}. 
\begin{proposition}\label{prop: var_decompose}
Suppose a probability density function $\rho(x)$ satisfies $0<m_\rho\leq \rho(x) \leq M_\rho < \infty$ for all $x\in\mb T^d$. Then, $\rho\in\ms P_{ac}(\mb T^d)$ satisfies the Boman chain condition with constants $A, B, C > 1$ only depending on $d, m_\rho, M_\rho$. Moreover, there exists a constant $\kappa = \kappa(m_\rho, M_\rho, d) > 0$, such that
\begin{align*}
\mb V_\rho(f) \leq \kappa\sum_{Q\in\m F} \rho(Q)\mb V_{\rho_Q}(f)
\end{align*}
holds for every continuous function $f$ on $\mb T^d$. Recall that $\rho_Q = \frac{\rho 1_Q}{\rho(Q)}$ is the probability measure given as the restriction of $\rho$ on $Q\subset\mb T^d$.
\myqed
\end{proposition}

With the Boman chain condition and  Proposition~\ref{prop: var_decompose}, we can globalize the estimate in Proposition~\ref{prop: variance lower bound}:
\begin{proposition}\label{prop: global second derivative}
Use the same notation and assumptions in Proposition~\ref{prop: variance lower bound}. Then, 
\begin{align*}
  \mb E_{Y\sim\nu}\Big[ \frac{\dd^2}{\dd\lambda^2} \m T_\mu^\varepsilon[\phi_\varepsilon + \lambda\chi](Y) \Big]  
  \le - \frac{C_1}{\kappa_\nu A_\nu}
  \mb V_{Y\sim\nu}\left(\frac{\dd}{\dd\lambda} \m T_\mu^\varepsilon [\phi + \lambda\chi](Y)\right).
\end{align*}
where $\kappa_\nu = \kappa_\nu(m_\nu, M_\nu, d) > 0$ and $A_\nu=A(d, m_\nu, M_\nu, d)>1$.
\end{proposition}
\begin{proof}
Let $\m F$ be a covering of $\mb T^d$ satisfying the Boman chain condition with constants $A_\nu, B_\nu, C_\nu > 1$ that only depend on $m_\nu, M_\nu$, and $d$. Since $\nu(x)\in[m_\nu, M_\nu]$ for all $x\in\mb T^d$, by Proposition~\ref{prop: var_decompose}, there exists a constant $\kappa_\nu = \kappa_\nu(m_\nu, M_\nu, d) > 0$ such that
\begin{align*}
\mb V_{Y\sim\nu}\left(\frac{\dd}{\dd\lambda} \m T_\mu^\varepsilon [\phi + \lambda\chi](Y)\right)
\leq \kappa_\nu \sum_{Q\in\m F} \nu(Q) \mb V_{Y\sim \nu_Q}
\left(\frac{\dd}{\dd\lambda} \m T_\mu^\varepsilon [\phi + \lambda\chi](Y))\right)
\end{align*}
Then, Propositions~\ref{prop: variance lower bound} implies that 
\begin{align*}
 \mb V_{Y\sim\nu}\left(\frac{\dd}{\dd\lambda} \m T_\mu^\varepsilon [\phi + \lambda\chi](Y)\right) \le - \,\kappa_\nu \sum_{Q\in\m F}\nu(Q) C_1^{-1}
 \mb E_{Y\sim\nu_Q}\Big[
 \frac{\dd^2}{\dd\lambda^2} \m T_\mu^\varepsilon[\phi_\varepsilon + \lambda\chi](Y)
 \Big].
\end{align*}
Recall that by concavity of $\phi \mapsto T_\mu^\varepsilon[\phi]$, we have $ \frac{\dd^2}{\dd\lambda^2} \m T_\mu^\varepsilon[\phi_\varepsilon + \lambda\chi](y) \le 0$ for all $y$. Thus, from the first part of Boman chain condition (Definition~\ref{defn: Bowen}), we can localize the expectation 
\begin{align*}
\mb E_{Y\sim\nu}\Big[ \frac{\dd^2}{\dd\lambda^2} \m T_\mu^\varepsilon[\phi_\varepsilon + \lambda\chi](Y) \Big]
\le A_\nu^{-1}  \sum_{Q\in\m F}  \nu(Q)
\mb E_{Y\sim\nu_Q}\Big[
 \frac{\dd^2}{\dd\lambda^2} \m T_\mu^\varepsilon[\phi_\varepsilon + \lambda\chi](Y)
 \Big].
\end{align*}
Combining this with the above inequality we get the desired result. 
\end{proof}

\subsection{Proof of Theorem~\ref{thm: stability}}\label{sec: pf_stability}

We first analyze the operation  $\phi \mapsto \m T_\mu^\varepsilon[\phi]$ by defining the following  probability measure for each continuous  function $\phi$,
\begin{align}\label{eqn: EOT_transition_kernel}
\pi_y[\phi](x) \coloneqq \frac{\mu(x)e^{\frac{\phi(x) - c_\varepsilon(x, y)}{\varepsilon}}}{\int_{\mb T^d}e^{\frac{\phi(x) - c_\varepsilon(x,y)}{\varepsilon}}\,\dd\mu(x)} \in\ms P_{ac}(\mb T^d).
\end{align}
This then gives the following useful (but straightforward) computation, which we state without proof:
\begin{lemma}\label{lem:pi-y}
For $\lambda \in \mathbb{R}$,  $\phi, \chi \in \mathcal{C}(\mathbb{T}^d)$, we have
\begin{align*}
\frac{\dd}{\dd\lambda} \m T_\mu^\varepsilon [\phi + \lambda\chi](y)
&= -\frac{\int_{\mb T^d}e^{\frac{\phi(x) + \lambda\chi(x) - c_\varepsilon(x, y)}{\varepsilon}}\chi(x)\,\dd\mu(x)}{\int_{\mb T^d}e^{\frac{\phi(x) + \lambda\chi(x) - c_\varepsilon(x,y)}{\varepsilon}}\,\dd\mu(x)} = -\mb E_{X\sim\pi_y[\phi + \lambda\chi]}[\chi(X)],\\
\frac{\dd^2}{\dd\lambda^2} \m T_\mu^\varepsilon[\phi + \lambda\chi](y)
&= -\frac{1}{\varepsilon}\mb V_{X\sim\pi_y[\phi + \lambda\chi]}[\chi(X)].
\end{align*}
\end{lemma}
We now prove Theorem~\ref{thm: stability}.

\begin{proof}[\bf Proof of Theorem~\ref{thm: stability}]
It suffices to prove~\eqref{eqn: stability1}, as~\eqref{eqn: stability2} follows by applying the same argument with $\mu$ and $\nu$ interchanged. 
We will use optimality of $\phi_\varepsilon$ and the second-order derivative estimate in Proposition~\ref{prop: global second derivative}; this is a natural idea as also done in \citep{delalande2023quantitative}.

\vspace{0.5em}
\noindent\underline{Step 1.} Let $f_\chi(\lambda) = \m I_{\mu, \nu}^\varepsilon[\phi_\varepsilon + \lambda\chi]$. 
Then, from optimality of $\phi_\varepsilon = \argmax_{\phi\in L^1(\mu)}\m I_{\mu, \nu}^\varepsilon[\phi]$, we have $f'(0) = 0$. Therefore, we can write  
\begin{align}\label{eqn: diff_func_2nd}
 \m I_{\mu, \nu}^\varepsilon[\phi_\varepsilon] -\m I_{\mu, \nu}^\varepsilon[\phi_\varepsilon+\chi] 
&= f_\chi(0) - f_\chi(1)
=-\int_0^1\!\!\int_0^t f_\chi''(\lambda)\,\dd \lambda\dd t
= -\int_0^1(1-\lambda) f_\chi''(\lambda)\,\dd\lambda.
\end{align}

\vspace{0.5em}\noindent
\underline{Step 2.} Notice that  
$$
  f_\chi''(\lambda)=  \mb E_{Y\sim\nu}\Big[ \frac{\dd^2}{\dd\lambda^2} \m T_\mu^\varepsilon[\phi_\varepsilon + \lambda\chi](Y) \Big].
$$
To estimate this in terms of $\mb V_\mu(\chi)$ as desired, we will follow several steps. 
First,
 from  Lemma~\ref{lem:pi-y}, we get
\begin{align}\label{eqn:f''EV}
   f_\chi''(\lambda)=  
   - \frac{1}{\varepsilon}\mb E_{Y\sim\nu} \Big[\mb V_{X\sim\pi_y[\phi + \lambda\chi]}[\chi(X)]\Big].
\end{align}
Also, together with Proposition~\ref{prop: global second derivative} and Lemma~\ref{lem:pi-y}, 
\begin{align}\label{eqn:f''VE}   f_\chi''(\lambda)\le 
  &  - \frac{C_1}{\kappa_\nu A_\nu}
  \mb V_{Y\sim\nu}\left(\frac{\dd}{\dd\lambda} \m T_\mu^\varepsilon [\phi + \lambda\chi](Y)\right)
  =  - \frac{C_1}{\kappa_\nu A_\nu}
  \mb V_{Y\sim\nu}\left(
   \mb E_{X\sim\pi_y[\phi + \lambda\chi]}[\chi(X)]\right)
\end{align}

\vspace{0.5em}\noindent
\underline{Step 3.} We now apply the law of total variance (see e.g. Theorem 3.27 in~\cite{wasserman2013all}), namely, for a family of probability density functions $\pi_y$ and the measure $d\pi^\nu (x,y):= \pi_y (x) d\nu(y)$ with a probability measure $\nu$,
\begin{align*}
    \mb V_{\pi^{\nu}}[\chi] 
&= \mb V_{Y\sim\nu}\Big(\mb E_{X\sim\pi_Y}[\chi(X)]\Big) + \mb E_{Y\sim\nu}\Big[\mb V_{X\sim\pi_Y}[\chi(X)]\Big].
\end{align*}
Apply this principle to the  probability distribution
$\pi^\nu[\phi_\varepsilon + \lambda\chi] \in \ms P(\mb T^d) $ defined by
\begin{align*}
\pi^\nu[\phi_\varepsilon + \lambda\chi](x) \coloneqq \int_{\mb T^d}\pi_y[\phi_\varepsilon + \lambda\chi](x)\,\dd\nu(y).
\end{align*}
Then, we have 
\begin{align*}
\mb V_{\pi^{\nu}[\phi_\varepsilon + \lambda\chi]}[\chi] 
&= \mb V_{Y\sim\nu}\Big[\mb E_{X\sim\pi_Y[\phi_\varepsilon + \lambda\chi]}[\chi(X)]\Big] + \mb E_{Y\sim\nu}\Big[\mb V_{X\sim\pi_Y[\phi_\varepsilon + \lambda\chi]}[\chi(X)]\Big]\\
&\le - \left( \varepsilon + \frac{\kappa_\nu A_\nu}{C_1} \right) f_\chi'' (\lambda).
\end{align*}
where the last line follows from  \eqref{eqn:f''EV} and \eqref{eqn:f''VE}.

\vspace{0.5em}\noindent
\underline{Step 4.}
The following lemma relates $\mb V_{\pi^{\nu}[\phi_\varepsilon + \lambda\chi]}[\chi]$ with $\mb V_{\mu}[\chi]$. Its proof follows the arguments in~\citep{chizat2025sharper}, and is deferred to Appendix~\ref{app: integ_lb}.
\begin{lemma}\label{lem: integ_lb}
For every $\lambda\in[0, 1]$, we have
\begin{align*}
\mb V_{X\sim\pi^\nu[\phi_\varepsilon + \lambda\chi]}[\chi(X)]
\geq \frac{1}{2}\mb V_{X\sim\mu}[\chi(X)] - \frac{4\lambda}{\varepsilon}\|\chi\|_{L^\infty(\mb T^d)}^2\Big[\m I_{\mu, \nu}^\varepsilon[\phi_\varepsilon] - \m I_{\mu, \nu}^\varepsilon[\phi_\varepsilon + \chi]\Big].
\end{align*}
\end{lemma}
We note that optimality of $\phi_\varepsilon$ is used to prove this lemma.

\vspace{0.5em}\noindent
\underline{Step 5.}  We now combine the above estimates to 
\eqref{eqn: diff_func_2nd}.
We get from Step 3, 
\begin{align*}
\m I_{\mu, \nu}^\varepsilon[\phi_\varepsilon] - \m I_{\mu, \nu}^\varepsilon[\phi_\varepsilon + \chi]
&\geq \left( \varepsilon + \frac{\kappa_\nu A_\nu}{C_1} \right)^{-1}
\int_0^1 (1-\lambda)\mb V_{\pi^{\nu}[\phi_\varepsilon + \lambda\chi]}[\chi]\,\dd\lambda.
\end{align*}
Applying Lemma~\ref{lem: integ_lb} yields 
\begin{align*}
\m I_{\mu, \nu}^\varepsilon[\phi_\varepsilon] - \m I_{\mu, \nu}^\varepsilon[\phi_\varepsilon + \chi]
&\geq \left( \varepsilon + \frac{\kappa_\nu A_\nu}{C_1} \right)^{-1}
\bigg[\frac{1}{4}\mb V_\mu[\chi] - \frac{2\|\chi\|^2_{L^\infty(\mb T^d)}}{3\varepsilon}\big[\m I_{\mu, \nu}^\varepsilon[\phi_\varepsilon] - \m I_{\mu, \nu}^\varepsilon[\phi_\varepsilon + \chi]\big]\bigg].
\end{align*}
That is,
\begin{align*}
\mb V_\mu(\chi) \leq \Big(
4 \varepsilon + 4\frac{\kappa_\nu A_\nu}{C_1}
 + \frac{8\|\chi\|^2_{L^\infty(\mb T^d)}}{3\varepsilon}\Big)\big[\m I_{\mu, \nu}^\varepsilon[\phi_\varepsilon] - \m I_{\mu, \nu}^\varepsilon[\phi_\varepsilon + \chi]\big].
\end{align*}
Observe that replacing $\chi$ with $\chi - \frac{\max_x \chi(x) + \min_x\chi(x)}{2}$  does not change $\mb V_\mu(\chi)$ and $\m I_{\mu, \nu}^\varepsilon[\phi_\varepsilon + \chi]$, but will replace $\|\chi\|_{L^\infty(\mb T^d)}$ with $\osc(\chi)$.  Now recall that as defined in Proposition~\ref{prop: variance lower bound}, 
\begin{align*}
    C_1= C_1(m_\nu, M_\nu, r, \alpha, d)  := \frac{m_\nu e^{-\frac{2r^2}{\alpha} - 8\pi^2 d}}{ M_\nu}.
\end{align*}
So, we get
\begin{align*}
\mb V_\mu(\chi) &\leq \Big(4\varepsilon + \frac{4\kappa_\nu A_\nu M_\nu e^{\frac{2r^2}{\alpha} + 8\pi^2 d}}{ m_\nu} + \frac{8\osc(\chi)^2}{3\varepsilon}\Big)\big[\m I_{\mu, \nu}^\varepsilon[\phi_\varepsilon] - \m I_{\mu, \nu}^\varepsilon[\phi_\varepsilon + \chi]\big]
\end{align*}
Taking $C_\nu = 4\kappa_\nu A_\nu M_\nu / m_\nu$ and using $r < R$ complete the proof of~\eqref{eqn: stability1}.
\end{proof}
\section{Proof of Main Results}\label{sec: main pf}

We are now ready to prove our main results, namely, Theorem~\ref{thm: multiSB_stable} and Theorem~\ref{thm: asymp_schro_potential}. Due to their logical connections we prove them in reverse order.

\subsection{Proof of Theorem~\ref{thm: asymp_schro_potential}}\label{sec: main pf expansion}
We will prove the statement for Schr\"odinger potentials at an arbitrary time point $t$. 
Fix $K \in \mb N^*$ and $t\in[0, 1)$. Let $(\phi_\varepsilon^\mu, \psi_\varepsilon^\mu)$ be the Schr\"odinger potentials defined through~\eqref{eqn: Schro_sys} for $\eot_\varepsilon(\rho_t^\mu, \rho_{t+\varepsilon}^\mu)$ with $0 < \varepsilon \leq 1 - t$. 
Here, $0< \varepsilon < \varepsilon_{\rm thres}$ for $\varepsilon_{\rm thres}$ to be determined below. We only need to find two sets of functions $\{f_k: 1\leq k < K\}$ and $\{g_k: 1\leq k < K\}$ depending on the time point $t$, such that
\begin{align}\label{eqn: expand_arbitrary_time}
\min_{c\in\mb R}\bigg\{\Big\|\phi_\varepsilon^\mu + c - \sum_{k=1}^{K-1} \varepsilon^k f_{k}\Big\|_{L^2(\rho_{t}^\mu)} +
\Big\|\psi_\varepsilon^\mu - c - \sum_{k=1}^{K-1} \varepsilon^k g_{k}\Big\|_{L^2(\rho_{t}^\mu)}\bigg\}
\leq O(\varepsilon^K).
\end{align}
To prove this result, use the notation \eqref{eqn:notation-ScB} to define
\begin{align*}
\big\{(u_k, v_k): k=0,1,\cdots,2K\big\} = \schro\big(\rho_t^\mu, \rho_t^{(1)}, \cdots, \rho_t^{(2K)}\big),
\end{align*}
and let
\begin{align*}
U_{2K}^\mu \coloneqq \sum_{k=0}^{2K}\varepsilon^k u_k,
\qquad
V_{2K}^\mu\coloneqq \sum_{k=0}^{2K}\varepsilon^k v_k.
\end{align*}
The main idea of proving~\eqref{eqn: expand_arbitrary_time} is as follows: First, from Theorem~\ref{thm: eot_expansion} it holds that the Schr\"odinger functionals $\m I_{\rho_t^\mu, \rho_{t+\varepsilon}^\mu}^\varepsilon$ and $\overline{\m I}_{\rho_t^\mu,\rho_{t+\varepsilon}^\mu}^\varepsilon$ evaluated at
\begin{align}\label{eqn: proxy_potential}
\wt\phi_\varepsilon^\mu \coloneqq \varepsilon \log U_{2K}^\mu
\quad\mx{and}\quad
\wt\psi_\varepsilon^\mu\coloneqq \varepsilon\log V_{2K}^\mu
\end{align}
approximate the value $\eot_\varepsilon(\rho_t^\mu, \rho_{t+\varepsilon}^\mu)$, which is their maximum value by the strong duality~\eqref{eqn: strong_dual}.  Then, we can apply 
the stability of the Schr\"odinger functionals established in Theorem~\ref{thm: stability} to estimate the difference between $(\phi_\varepsilon^\mu, \psi_\varepsilon^\mu)$ and $(\wt\phi_\varepsilon^\mu, \wt\psi_\varepsilon^\mu)$. 
Then the expansions of $\wt\phi_\varepsilon^\mu, \wt\psi_\varepsilon^\mu$ will give those $f_k, g_k$'s  in \eqref{eqn: expand_arbitrary_time}.
A detailed proof is presented below.

\vspace{0.5em}\noindent
\underline{Step 1: Control $\mb V_{\rho_t^\mu}\big(\phi_\varepsilon^\mu - \wt\phi_\varepsilon^\mu\big)$  using the stability of Schr\"odinger functionals.}
By Theorem~\ref{thm: stability}, there exists $C_1>0$ depending only on the constants $L_\rho$, $d$, and $\alpha, R$ introduced in Proposition~\ref{prop: 2norm_local_convex}, such that
\begin{align}\label{eqn: direct_apply_stability}
\mb V_{\rho_{t}^\mu}(\wt\phi_\varepsilon^\mu - \phi_\varepsilon^\mu) 
\leq \Big(4\varepsilon + C_1 + \frac{8\osc(\wt\phi_\varepsilon^\mu - \phi_\varepsilon^\mu)^2}{3\varepsilon}\Big)
\Big[\m I_{\rho_{t}^\mu, \rho_{t+\varepsilon}^\mu}^{\varepsilon}[\phi_\varepsilon^\mu] - \m I_{\rho_{t}^\mu, \rho_{t+\varepsilon}^\mu}^{\varepsilon}[\wt \phi_\varepsilon^\mu]\Big].
\end{align}

Now, let us estimate the oscillation. Note that
\begin{align*}
\osc\big(\wt\phi_\varepsilon^\mu - \phi_\varepsilon^\mu\big) 
 \leq \osc ( \wt\phi_\varepsilon^\mu) + \osc(\phi_\varepsilon^\mu)
\leq 2\big\|\wt\phi_\varepsilon^\mu\big\|_{L^\infty} + \osc(\phi_\varepsilon^\mu)
\leq 2 \varepsilon \big\|\log U_{2K}^\mu\big\|_{L^\infty} + 4\pi^2d.
\end{align*}
Here, the last inequality follows from the Schr\"odinger system $\phi_\varepsilon^\mu = \m T^\varepsilon_{\rho_{t+\varepsilon}^\mu}[\psi_\varepsilon^\mu]$ in~\eqref{eqn: Schro_sys} and the inequality~\eqref{eqn: osc_cost}. 

Next, we control the $L^\infty$-norm of $\log U_{2K}^\mu$. We know $\min_{x\in\mb T^d} u_0(x) > 0$ by Theorem~\ref{thm: existence}. Then, there exists 
\begin{align*}
    \hbox{$\varepsilon_{2K,1} > 0$ depending only
$\|u_1\|_{L^\infty}, \cdots, \|u_{2K}\|_{L^\infty}$ and the positive lower bound of $u_0$,}
\end{align*}
such that for all $\varepsilon \leq \varepsilon_{2K, 1}$ we have
\begin{align*}
\frac{1}{2}\min_{x\in\mb T^d} u_0(x) \geq \sum_{k=1}^{2K}\varepsilon^k \|u_k\|_{L^\infty}.
\end{align*}
This implies
\begin{align*}
0 < \frac{1}{2}\min_{x\in\mb T^d} u_0(x)
\leq \big\|U_{2K}^\mu\big\|_{L^\infty}
\leq \|u_0\|_{L^\infty} +  \frac{1}{2}\min_{x\in\mb T^d}u_0(x) 
\leq \frac{3}{2}\|u_0\|_{L^\infty}.
\end{align*}
Using Theorem~\ref{thm: existence}, there exists a constant $C_2 = C_2\big(d, L_\rho, \big\|\rho_t^\mu\big\|_{\m C^3}, \big\|\rho_t^{(1)}\big\|_{\m C^1}\big) \in (0, 1]$, such that
\begin{align*}
C_2 \leq \min_{x\in\mb T^d} u_0(x) \leq \|u_0\|_{L^\infty} \leq (C_2)^{-1}.
\end{align*}
Applying this to~\eqref{eqn: direct_apply_stability} yields for each $0< \varepsilon \le  \varepsilon_{2K, 1}$,
\begin{align}\label{eqn: estimate_through_stability}
\mb V_{\rho_{t}^\mu}(\wt\phi_\varepsilon^\mu - \phi_\varepsilon^\mu) 
\leq \Big[4\varepsilon + C_1 + \frac{8}{3\varepsilon}\Big(2 \varepsilon \log\frac{3}{2C_2} + 4\pi^2d\Big) ^2\Big]
\Big[\m I_{\rho_{t}^\mu, \rho_{t+\varepsilon}^\mu}^{\varepsilon}[\phi_\varepsilon^\mu] - \m I_{\rho_{t}^\mu, \rho_{t+\varepsilon}^\mu}^{\varepsilon}[\wt \phi_\varepsilon^\mu]\Big].
\end{align}
Again, due to Theorem~\ref{thm: existence}, we know that $\varepsilon_{2K, 1} > 0$ depends only on 
$$
\big\|\rho_t^\mu\big\|_{\m C^{4K+3}}, \big\|\rho_t^{(1)}\big\|_{\m C^{4K+1}}, \cdots, \big\|\rho_t^{(2K+1)}\big\|_{\m C^1}.
$$


\vspace{0.5em}\noindent
\underline{Step 2: Estimate the difference $\m I_{\rho_t^\mu, \rho_{t+\varepsilon}^\mu}^\varepsilon[\phi_\varepsilon^\mu] - \m I_{\rho_t^\mu, \rho_{t+\varepsilon}^\mu}^\varepsilon[\wt\phi_\varepsilon^\mu]$.}
Recall that we have strong duality
\begin{align*}
\m I_{\rho_t^\mu, \rho_{t+\varepsilon}^\mu}^\varepsilon\big[\phi_\varepsilon^\mu\big] = \eot_\varepsilon\big(\rho_t^\mu, \rho_{t+\varepsilon}^\mu\big).
\end{align*}
Applying the estimate of the upper bound of the EOT in\eqref{eqn: eot_ub}, there exist  constants $C_3>0$ and $\varepsilon_{2K, 2}>0$, 
such that for each $0 < \varepsilon \leq \varepsilon_{2K, 2}$,
\begin{align*}
\m I_{\rho_t^\mu, \rho_{t+\varepsilon}^\mu}^\varepsilon\big[\phi_\varepsilon^\mu\big]
\leq \varepsilon\int_{\mb T^d}\log U_{2K}^\mu\,\dd\rho_t^\mu + \varepsilon\int_{\mb T^d}\log V_{2K}^\mu\,\dd\rho_{t+\varepsilon}^\mu + C_3\varepsilon^{2K+1}.
\end{align*}
Here, both $C_3$ and $\varepsilon_{2K, 2}$ depend only on $L_\rho$, and the upper bound  of
\begin{align*}
&\sup\big\{\|\rho_t^{(k)}\|_{H^{2k'+s}}: t\in[0, 1], \,\,k,k'\in\mb N,\,\, k+k'\leq 2\big\},\\
& \|u_0\|_{H^{s+4}}, \cdots, \|u_{2K}\|_{H^{s+4}},
\quad\mx{and}\quad
\|v_0\|_{H^{s+4}}, \cdots, \|v_{2K}\|_{H^{s+4}}
\end{align*}
with $s = \lceil\frac{d+1}{2}\rceil$, and the lower bound of 
\begin{align*}
\min_{x\in\mb T^d}u_0(x)
\quad\mx{and}\quad
\min_{x\in\mb T^d} v_0(x).
\end{align*}
Applying Theorem~\ref{thm: existence}, these constants $C_3$ and $\varepsilon_{2K, 2}$  depend only  on $d,K$,
\begin{align*}
&\sup_{0\leq t\leq 1}\big\{\|\rho_t^{(k)}\|_{H^{2k'+s}}:k,k'\in\mb N,\,\, k+k'\leq 2\big\}
\quad\mx{and}\quad
\big\|\rho_t^\mu\big\|_{\m C^{s+4K+4}}, \big\|\rho_t^{(1)}\big\|_{\m C^{s+4K+2}}, \cdots, \big\|\rho_t^{(2K+1)}\big\|_{\m C^{s+2}}.
\end{align*}

Using the definition of the Schr\"odinger functional $\m I_{\rho_t^\mu, \rho_{t+\varepsilon}^\mu}^\varepsilon$ defined in~\eqref{eqn: dual functional} and Lemma~\ref{lem: estim_dual_lb}, we have
\begin{align*}
\m I_{\rho_t^\mu, \rho_{t+\varepsilon}^\mu}^\varepsilon\big[\wt \phi_\varepsilon^\mu\big]
&= \int_{\mb T^d} \wt\phi_\varepsilon^\mu\,\dd\rho_t^\mu + \int_{\mb T^d} \m T_{\rho_t^\mu}^\varepsilon\big[\wt\phi_\varepsilon^\mu\big]\,\dd\rho_{t+\varepsilon}^\mu\\
&\geq \varepsilon\int_{\mb T^d}\log U_{2K}^\mu\,\dd\rho_t^\mu + \varepsilon\int_{\mb T^d}\log V_{2K}^\mu\,\dd\rho_{t+\varepsilon}^\mu - C_4\varepsilon^{2K+2},
\end{align*}
where $C_4 \geq 0$ is a constant depending only on the upper bound of
\begin{align*}
&\|u_0\|_{\m C^{4K+2}}, \cdots, \|u_{2K}\|_{\m C^{4K+2}},
\|v_0\|_{\m C^{4K+2}}, \cdots, \|v_{2K}\|_{\m C^{4K+2}}
\quad\mx{and}\\
&\sup\big\{\|\rho_t^{(k)}\|_{\m C^{2k'}}: t\in[0,1],\,\, k+k' \leq 2K+1, k,k'\in\mb N\big\}.
\end{align*}
Therefore, we get
\begin{align}\label{eqn: diff_schro_func}
\m I_{\rho_t^\mu, \rho_{t+\varepsilon}^\mu}^\varepsilon\big[\phi_\varepsilon^\mu\big] - \m I_{\rho_t^\mu, \rho_{t+\varepsilon}^\mu}^\varepsilon\big[\wt \phi_\varepsilon^\mu\big]
\leq \big(C_3 + C_4\varepsilon\big)\varepsilon^{2K+1}.
\end{align}
Applying Theorem~\ref{thm: existence} again, we know that $C_4$ only depends on $d,K$
\begin{align*}
\sup_{0\leq t\leq 1}\big\{\|\rho_t^{(k)}\|_{\m C^{2k'}}: k,k'\in\mb N, k+k'\leq 2K+1\big\}
\quad\mx{and}\quad
\big\|\rho_t^\mu\big\|_{\m C^{8K+3}}, \big\|\rho_t^{(1)}\big\|_{\m C^{8K+1}}, \cdots, \big\|\rho_t^{(2K+1)}\big\|_{\m C^{4K+1}}.
\end{align*}




\vspace{0.5em}\noindent
\underline{Step 3: Summary of Steps 1 and 2.}
From the previous steps \eqref{eqn: estimate_through_stability} and~\eqref{eqn: diff_schro_func}, we now know there exist constants $C_5 > 0$ and $\varepsilon_{2K, 3}>0$, such that for each $\varepsilon \le \varepsilon_{2K, 3}$,
\begin{align*}
\mb V_{\rho_{t}^\mu}(\wt\phi_\varepsilon^\mu - \phi_\varepsilon^\mu) 
& \leq \Big[4\varepsilon + C_1 + \frac{8}{3\varepsilon}\Big(2\varepsilon \log\frac{3}{2C_2} + 4\pi^2d\Big) ^2\Big]\big(C_3 + C_4\varepsilon\big)\varepsilon^{2K+1}
\leq C_5\varepsilon^{2K}.
\end{align*}

Apply for $\psi^\mu_\varepsilon, \tilde \psi^\mu_\varepsilon$ exactly the same procedure as in Step 1 and Step 2, then we get 
constants $C_5' > 0$ and $\varepsilon_{2K, 3}'>0$, such that for each $\varepsilon \le \varepsilon_{2K, 3}'$,
\begin{align*}
\mb V_{\rho_{t+\varepsilon}^\mu}(\wt\psi_\varepsilon^\mu - \psi_\varepsilon^\mu) 
\leq C_5'\varepsilon^{2K}.
\end{align*}
Here, the constants $C_5$ and $C_5'$ depend on $L_\rho, d, \alpha, R$,
\begin{align*}
&\sup_{0\leq t\leq 1}\big\{\|\rho_t^{(k)}\|_{H^{2k'+s}}:k,k'\in\mb N,\,\, k+k'\leq 2\big\}
\quad\mx{and}\quad
\big\|\rho_t^\mu\big\|_{\m C^{s+4K+4}}, \big\|\rho_t^{(1)}\big\|_{\m C^{s+4K+2}}, \cdots, \big\|\rho_t^{(2K+1)}\big\|_{\m C^{s+2}},\\
&\sup_{0\leq t\leq 1}\big\{\|\rho_t^{(k)}\|_{\m C^{2k'}}: k,k'\in\mb N, k+k'\leq 2K+1\big\}
\quad\mx{and}\quad
\big\|\rho_t^\mu\big\|_{\m C^{8K+3}}, \big\|\rho_t^{(1)}\big\|_{\m C^{8K+1}}, \cdots, \big\|\rho_t^{(2K+1)}\big\|_{\m C^{4K+1}}.
\end{align*}
where $\alpha$ and $R$ are introduced in Proposition~\ref{prop: 2norm_local_convex}, and $s = \lceil\frac{d+1}{2}\rceil$;
the constants $\varepsilon_{2K, 3}, \varepsilon_{2K,3}'$ depend on
\begin{align*}
&d, K, \big\|\rho_t^\mu\big\|_{\m C^{s+4K+4}}, \big\|\rho_t^{(1)}\big\|_{\m C^{s+4K+2}}, \cdots, \big\|\rho_t^{(2K+1)}\big\|_{\m C^{s+2}}
\quad\mx{and}\quad
\sup_{0\leq t\leq 1}\big\{\|\rho_t^{(k)}\|_{H^{2k'+s}}:k,k'\in\mb N,\,\, k+k'\leq 2\big\}
.
\end{align*}

\vspace{0.5em}\noindent
\underline{Step 4: Converting the variance $\mb V_{\rho_t^\mu}$ to $L^2(\rho_t^\mu)$-norm.}
Note that for any constant $c\in\mb R$, we have
\begin{align}
\begin{aligned}\label{eqn: lhs_stability}
\mb V_{\rho_{t}^\mu}\big(\wt\phi_\varepsilon^\mu - \phi_\varepsilon^\mu\big) 
&= \mb V_{\rho_{t}^\mu}\big(\wt\phi_\varepsilon^\mu - \phi_\varepsilon^\mu - c\big)\\
&= \int_{\mb T^d}\big(\wt\phi_\varepsilon^\mu - \phi_\varepsilon^\mu - c\big)^2\,\dd\rho_{t}^\mu - \bigg(\int_{\mb T^d}(\wt\phi_\varepsilon^\mu - \phi_\varepsilon^\mu - c)\,\dd\rho_{t}^\mu\bigg)^2.
\end{aligned}
\end{align}
Therefore, from Step 3, for $\varepsilon \le \varepsilon_{2K, 3}$,
\begin{align*}
\int_{\mb T^d}\big(\wt\phi_\varepsilon^\mu - \phi_\varepsilon^\mu - c\big)^2\,\dd\rho_{t}^\mu 
\leq \bigg(\int_{\mb T^d}\wt\phi_\varepsilon^\mu - \phi_\varepsilon^\mu - c\,\dd\rho_{t}^\mu\bigg)^2 + C_5\varepsilon^{2K}.
\end{align*}
As $c\in\mb R$ is an arbitrary number, similarly, for $\varepsilon \le \varepsilon_{2K, 3}'$, 
\begin{align*}
\int_{\mb T^d}\big(\wt\psi_\varepsilon^\mu - \psi_\varepsilon^\mu + c\big)^2\,\dd\rho_{t+\varepsilon}^\mu 
\leq \bigg(\int_{\mb T^d}\wt\psi_\varepsilon^\mu - \psi_\varepsilon^\mu + c\,\dd\rho_{t+\varepsilon}^\mu\bigg)^2 + C_5'\varepsilon^{2K}.
\end{align*}
Summing the above two inequalities yields that there is a constant $\varepsilon_{\rm thres}>0$ such that for each $0< \varepsilon \le \varepsilon_{\rm thres}$ we have
\begin{align}\label{eqn:inter-expansion-proof}
\begin{aligned}
&\quad\,\min_{c\in\mb R} \big\|\wt\phi_\varepsilon^\mu - \phi_\varepsilon^\mu - c\big\|_{L^2(\rho_{t}^\mu)}^2 + \big\|\wt\psi_\varepsilon^\mu - \psi_\varepsilon^\mu + c\big\|_{L^2(\rho_{t+\varepsilon}^\mu)}^2\\
&\leq \min_{c\in\mb R} 
\left[\bigg(\int_{\mb T^d}\wt\phi_\varepsilon^\mu - \phi_\varepsilon^\mu - c\,\dd\rho_{t}^\mu\bigg)^2 + \bigg(\int_{\mb T^d}\wt\psi_\varepsilon^\mu - \psi_\varepsilon^\mu + c\,\dd\rho_{t+\varepsilon}^\mu\bigg)^2\right] + (C_5 + C_5')\varepsilon^{2K}\\
&= \frac{1}{2}\bigg(\int_{\mb T^d}\wt\phi_\varepsilon^\mu - \phi_\varepsilon^\mu \,\dd\rho_{t}^\mu + \int_{\mb T^d}\wt\psi_\varepsilon^\mu - \psi_\varepsilon^\mu \,\dd\rho_{t +\varepsilon}^\mu\bigg)^2 + (C_5 + C_5')\varepsilon^{2K}\\
&=\frac{1}{2}\bigg[\eot_{\varepsilon}(\rho_{t}^\mu, \rho_{t+\varepsilon}^\mu) - \bigg(\varepsilon\int_{\mb T^d}\log U_{2K}^{\mu}\,\dd\rho_{t}^\mu + \varepsilon\int_{\mb T^d}\log V_{2K}^{\mu}\,\dd\rho_{t+\varepsilon}^\mu\bigg)\bigg]^2 + (C_5 + C_5')\varepsilon^{2K}\\
&\leq \frac{C_6}{2}\varepsilon^{4K+2} + (C_5 + C_5')\varepsilon^{2K}.
\end{aligned}
\end{align}
Here, the second last line follows from the definition of $(\wt\phi_\varepsilon^\mu, \wt\psi_\varepsilon^\mu)$ and the strong duality \eqref{eqn: strong_dual}. The last line follows from Theorem~\ref{thm: eot_expansion}.
Here, the constant $\varepsilon_{\rm thres}$ depends only on
$K, d, L_\rho$, and 
\begin{align*}
\sup_{0\leq t\leq 1}\big\{\|\rho_t^{(k)}\|_{\m C^{2k'+s}}: k,k'\in\mb N,\,\,k+k'\leq 2\big\}
\quad\mx{and}\quad
\big\|\rho_t^\mu\big\|_{\m C^{s+4K+5}}, \big\|\rho_t^{(1)}\big\|_{\m C^{s+4K+3}}, \cdots, \big\|\rho_t^{(2K+2)}\big\|_{\m C^{s+1}}.
\end{align*}
(Here, recall the definition $s = \lceil\frac{d+1}{2}\rceil$.)
The constant $C_6 > 0$ depends only on \begin{align*}
\sup_{0\leq t\leq 1}\big\{\|\rho_t^{(k)}\|_{\m C^{2k'}}: k,k'\in\mb N,\,\,k+k'\leq 2K+1\big\}
\quad\mx{and}\quad
\big\|\rho_t^\mu\big\|_{\m C^{8K+3}}, \big\|\rho_t^{(1)}\big\|_{\m C^{8K+1}}, \cdots, \big\|\rho_t^{(2K+2)}\big\|_{\m C^{4K-1}}.
\end{align*}

\vspace{0.5em}\noindent
\underline{Step 5: Expansion in $\varepsilon$.} 
We now translate \eqref{eqn:inter-expansion-proof} into the desired polynomial expansion as in \eqref{eqn: expand_arbitrary_time}. For  $\big(\wt\phi_\varepsilon^\mu, \wt\psi_\varepsilon^\mu\big)$ in~\eqref{eqn: proxy_potential}, apply polynomial expansions in $\varepsilon$ that give coefficient functions as follows: 
\begin{align}\label{eqn: explicit_construct}
\begin{aligned}
& f_{1} = \log u_{0},\quad g_{1} = \log v_{0},\\
& f_{k} = \sum_{l=1}^{k-1} \frac{(-1)^{l+1}}{l! u_{0}^l}\sum_{\substack{s_1 + \cdots + s_l = k-1\\1\leq s_1, \cdots, s_l < k}} u_{s_1}\cdots u_{s_l}, \quad \forall\, 2\leq k \leq K-1,\\
& g_{k} = \sum_{l=1}^{k-1} \frac{(-1)^{l+1}}{l! v_{0}^l}\sum_{\substack{s_1 + \cdots + s_l = k-1\\1\leq s_1, \cdots, s_l < k}} v_{s_1}\cdots v_{s_l}, \quad \forall\, 2\leq k \leq K-1.
\end{aligned}
\end{align}
We now plug these back in \eqref{eqn:inter-expansion-proof} and get \eqref{eqn: expand_arbitrary_time}. This completes the proof.
\subsection{Proof of Theorem~\ref{thm: multiSB_stable}}\label{sec: pf_main_result}

In this section, we apply the previous  established
results to derive quantitative stability of multi-marginal Schr\"odinger bridges with respect to the marginals. 
First, we prove \eqref{eqn:KL-potentials} as a consequence of the structure of the EOT problem, including the Markov property of the solution:
\begin{lemma}\label{lem:KL-decomp-EOT}
Let $(\phi^\mu_{j}, \psi_{j}^\mu)$ and $(\phi^\nu_{j}, \psi^{\nu}_{j})$ be the Schr\"odinger potentials for solving $\eot_{\varepsilon_j}(\rho^\mu_{t_{j-1}}, \rho_{t_j}^\mu)$ and $\eot_{\varepsilon_j}(\rho_{t_{j-1}}^\nu, \rho_{t_j}^\nu)$ with $\varepsilon_j = t_{j} - t_{j-1}$, respectively. Then,
    \begin{align}\label{eqn: simplify_KL_MSB}
\KL(R^{\bm\nu^m}\,\|\,R^{\bm\mu^m})
\stackrel{}{=} \sum_{j=1}^m \frac{1}{\varepsilon_j}\bigg[\eot_{\varepsilon_j}(\rho_{t_{j-1}}^\nu, \rho_{t_j}^\nu) - \int_{\mb T^d}\phi_j^\mu\,\dd\rho_{t_{j-1}}^\nu - \int_{\mb T^d}\psi_j^\mu\,\dd\rho_{t_j}^\nu\bigg] + \sum_{j=0}^m  \KL(\rho_{t_j}^\nu\,\|\,\rho_{t_j}^\mu).
\end{align}
\end{lemma} 
\begin{rem}
    Observe that in \eqref{eqn: simplify_KL_MSB} only the Schr\"odinger potentials $\phi_j^\mu,\psi_j^\mu$of $\mu_j$'s appear, which can be controlled by Assumption~\ref{assump: density_expansion}.  
\end{rem}
\begin{proof}

Using the form of the solution to multi-marginal Schr\"odinger bridge problems as in~\eqref{eqn: Markov}, we have
\begin{align*}
\frac{\dd R^{\bm \nu^m}}{\dd W}(X) = \frac{\dd R^{\bm\nu^m}_{t_0,\cdots, t_m}}{\dd W_{t_0,\cdots, t_m}}(X_{t_0}, \cdots, X_{t_m}).
\end{align*}
Due to the Markov property of Brownian motion and $R_{t_0, \cdots, t_m}^{\bm\nu^m}$ as in the representation~\eqref{eqn: MSB_Markov}, we have
\begin{align*}
\frac{\dd R^{\bm\nu^m}_{t_0,\cdots, t_m}}{\dd W_{t_0,\cdots, t_m}}(X_{t_0}, \cdots, X_{t_m})
= \frac{\dd R_{t_0}^{\bm\nu^m}}{\dd W_{t_0}}(X_{t_0}) \frac{\dd R^{\bm\nu^m}_{t_1|t_0}}{\dd W_{t_1|t_0}}(X_{t_0}, X_{t_1}) \cdots \frac{\dd R^{\bm\nu^m}_{t_m|t_{m-1}}}{W_{t_m|t_{m-1}}}(X_{t_{m-1}}, X_{t_m}),
\end{align*}
where $R_{t_{j-1}, t_j}^{\bm\nu^m}$ is the optimal coupling for solving $\eot_{\varepsilon_j}(\rho_{t_{j-1}}^\nu, \rho_{t_j}^\nu)$, and $R_{t_j|t_{j-1}}^{\bm\nu^m}$ is its conditional distribution.
Moreover, using the definition of conditional distribution, we have
\begin{align*}
R^{\bm\nu^m}_{t_j|t_{j-1}}(X_{t_{j-1}}, X_{t_j}) = \frac{R^{\bm\nu^m}_{t_{j-1}, t_j}(X_{t_{j-1}}, X_{t_j})}{R^{\bm\nu^m}_{t_{j-1}}(X_{t_{j-1}})},\quad j\in[m].
\end{align*}
Combining the above pieces  yields
\begin{align*}
\log\frac{\dd R^{\bm \nu^m}(X)}{\dd W (X)} 
&=\log\frac{\dd R_{t_0,\cdots, t_m}^{\bm \nu^m}(X_{t_0}, \cdots, X_{t_m})}{\dd W_{t_0,\cdots, t_m}(X_{t_0}, \cdots, X_{t_m})}\\
&= \sum_{j=1}^m \log\frac{\dd R^{\bm\nu^m}_{t_{j-1}, t_j}(X_{t_{j-1}}, X_{t_j})}{\dd W_{t_{j-1}, t_j}(X_{t_{j-1}}, X_{t_j})} - \sum_{j=1}^{m-1}\log\frac{\dd R^{\bm\nu^m}_{t_j}(X_{t_j})}{\dd W_{t_j}(X_{t_j})}.
\end{align*}
Now apply the straight formula for general probability measures $A, A', B$,  
\begin{align*}
\KL(A \, \|\,A') :=\int \log \left(\frac{\dd A}{\dd A'}\right)\dd R = 
\int \log \left(\frac{\dd A/\dd B}{\dd A'/\dd B}\right)\dd A=\int \Big( \log\frac{\dd A}{\dd B} -\log\frac{\dd A'}{\dd B}\Big)\dd A.
\end{align*}
Then piecing together the above computations and  from the marginal conditions $R_{t_j}^{\bm\mu^m} = \rho_{t_j}^{\mu}, \, R_{t_j}^{\bm\nu^m} = \rho_{t_j}^{\nu},$  we get
\begin{align}\label{eqn:KL-Markov-sum}
&\quad\,\KL(R^{\bm\nu^m}\,\|\,R^{\bm\mu^m})
= \sum_{j=1}^m \KL(R^{\bm\nu^m}_{t_{j-1}, t_j}\,\|\,R^{\bm\mu^m}_{t_{j-1}, t_j}) - \sum_{j=1}^{m-1}\KL(\rho_{t_j}^\nu\,\|\,\rho_{t_j}^\mu)
\end{align}
Now, by the optimality condition~\eqref{eqn: SB_primal_sol}, we know
\begin{align*}
R^{\bm\nu^m}_{t_{j-1}, t_j}(x_{j-1}, x_j) &= e^{\frac{\phi_j^\nu(x_{j-1}) + \psi_j^\nu(x_j) - c_{\varepsilon_j}(x_{j-1}, x_j)}{\varepsilon_j}} \rho_{t_{j-1}}^\nu(x_{j-1})\rho_{t_j}^\nu(x_j),\\
R^{\bm\mu^m}_{t_{j-1}, t_j}(x_{j-1}, x_j) &= e^{\frac{\phi_j^\mu(x_{j-1}) + \psi_j^\mu(x_j) - c_{\varepsilon_j}(x_{j-1}, x_j)}{\varepsilon_j}} \rho_{t_{j-1}}^\mu(x_{j-1})\rho_{t_j}^\mu(x_j).
\end{align*}
Therefore, we have
\begin{align*}
&\quad\,\KL(R^{\bm\nu^m}_{t_{j-1}, t_j}\,\|\,R^{\bm\mu^m}_{t_{j-1}, t_j})\\
&= \int_{\mb T^d\times\mb T^d} \left[\frac{\phi_j^\nu(x_{j-1}) + \psi_j^\nu(x_j) - \phi_j^\mu(x_{j-1}) - \psi_j^\mu(x_j)}{\varepsilon_j} + \log\frac{\rho^\nu_{t_{j-1}}(x_{j-1})}{\rho_{t_{j-1}}^\mu(x_{j-1})} + \log\frac{\rho_{t_j}^\nu(x_j)}{\rho_{t_j}^\mu(x_j)}\,\right]\dd R^{\bm\nu}_{t_{j-1},t_j}\\
&= \frac{1}{\varepsilon_j}\int_{\mb T^d}(\phi_j^\nu - \phi_j^\mu)\,\dd\rho_{t_{j-1}}^\nu + \frac{1}{\varepsilon_j}\int_{\mb T^d} (\psi_j^\nu - \psi_j^\mu)\,\dd\rho_{t_j}^\nu + \KL(\rho^\nu_{t_{j-1}}\,\|\,\rho^\mu_{t_{j-1}}) + \KL(\rho_{t_j}^\nu\,\|\,\rho_{t_j}^\mu)
\end{align*}
Combining this with \eqref{eqn:KL-Markov-sum} we get
\begin{align*}
&\quad\,\KL(R^{\bm\nu}\,\|\,R^{\bm\mu})
= \sum_{j=1}^m\bigg[\frac{1}{\varepsilon_j}\int_{\mb T^d}(\phi_j^\nu - \phi_j^\mu)\,\dd\rho_{t_{j-1}}^\nu + \frac{1}{\varepsilon_j}\int_{\mb T^d}(\psi_j^\nu - \psi_j^\mu)\,\dd\rho_{t_j}^\nu\bigg] + \sum_{j=0}^m  \KL(\rho_{t_j}^\nu\,\|\,\rho_{t_j}^\mu)\\
&\stackrel{}{=} \sum_{j=1}^m \frac{1}{\varepsilon_j}\bigg[\eot_{\varepsilon_j}(\rho_{t_{j-1}}^\nu, \rho_{t_j}^\nu) - \int_{\mb T^d}\phi_j^\mu\,\dd\rho_{t_{j-1}}^\nu - \int_{\mb T^d}\psi_j^\mu\,\dd\rho_{t_j}^\nu\bigg] + \sum_{j=0}^m  \KL(\rho_{t_j}^\nu\,\|\,\rho_{t_j}^\mu).
\end{align*}
Here, the last equality is due to the strong duality~\eqref{eqn: strong_dual}. We complete the proof of the equality~\eqref{eqn: simplify_KL_MSB}.
\end{proof}

We will also use an additional ingredient:
\begin{proposition}[Theorem 1.6 of \citep{conforti2021formula}]\label{lem:Conforti}
We have for any $t$, $\varepsilon >0$,
\begin{align}\label{eqn:Conforti21}
\eot_{\varepsilon}(\rho_{t}^\nu, \rho_{t+\varepsilon}^\nu)
&\leq \W_2^2(\rho_{t}^\nu, \rho_{t+\varepsilon}^\nu) - \frac{\varepsilon}{2}\big[\m H(\rho_{t}^\nu) + \m H(\rho_{t+\varepsilon}^\nu)\big] + \frac{\varepsilon}{8}\int_{t}^{t+\varepsilon}\!\int_{\mb T^d}\|\nabla\log\bar\rho_{s}^\nu\|^2\bar\rho_s^\nu\,\dd x\dd s.
\end{align}
Here, $(\bar\rho_s^\nu)_{s\in[t, t+\varepsilon]}$ is the  Wasserstein-2 constant-speed geodesics connecting  $\rho_{t}^\nu$ and $\rho_{t+\varepsilon}^\nu$  with the  velocity vector field $\nabla\bar\Phi_s^\nu$ satisfying
\[
\partial_s\bar\rho_s^\nu + \nabla\cdot(\bar\rho_s^\nu\nabla\bar\Phi_s^\nu) = 0,\qquad s\in[t, t+\varepsilon].
\]
\end{proposition}
Now, comparing \eqref{eqn: simplify_KL_MSB} and \eqref{eqn:Conforti21}  and recalling the approximation result \eqref{eqn: eot_cost_expand}
motivates the definition of $S_j$ in the following lemma. It 
helps simplify the proof of the main theorem.
Its proof consists of calculations and simple estimates, and  will be deferred to Appendix~\ref{app: simplification}. 
\begin{lemma}\label{lem: simplification}
Let $u_{t_j, k}, v_{t_j, k}$  be given as in \eqref{eqn: eot_cost_expand} with $t= t_{j-1}$ and $\varepsilon=t_j-t_{j-1}$.
Recall the negative self-entropy functional $\m H(\rho)\coloneqq \int_{\mb T^d}\rho(x)\log\rho(x)\,\dd x$ for $\rho\in\ms P_{ac}(\mb T^d)$.
Define
\begin{align*}
S_j&\coloneqq\frac{1}{\varepsilon_j}\bigg[\W_2^2(\rho_{t_{j-1}}^\nu, \rho_{t_j}^\nu) - \frac{\varepsilon_j}{2}\big[\m H(\rho_{t_{j-1}}^\nu) + \m H(\rho_{t_j}^\nu)\big] + \frac{\varepsilon_j}{8}\int_{t_{j-1}}^{t_j}\!\int_{\mb T^d}\|\nabla\log\bar\rho_{t}^\nu\|^2\bar\rho_t^\nu\,\dd x\dd t\bigg]\\
&\qquad - \bigg[\int_{\mb T^d}\left(\log u_{t_{j-1},0} + \varepsilon_j\frac{u_{t_{j-1},1}}{u_{t_{j-1},0}}\right)\,\dd\rho_{t_{j-1}}^\nu + \int_{\mb T^d}
\left(\log v_{t_{j-1},0} + \varepsilon_j\frac{v_{t_{j-1},1}}{v_{t_{j-1},0}}\right)\,\dd\rho_{t_j}^\nu\bigg].
\end{align*}
Then, there are remainder terms $\delta_{1,j}$ and $\delta_{2,j}$ (c.f. Appendix~\ref{app: simplification} for explicit expressions) such that
\begin{align*}
S_j &= \frac{1}{2}\int_{t_{j-1}}^{t_j}\!\int_{\mb T^d}\Big\|\nabla\Phi_{t_{j-1}}^\mu - \nabla\bar \Phi_t^\nu - \frac{1}{2}\nabla\log\frac{\rho_{t_{j-1}}^\mu}{\bar\rho_t^\nu}\Big\|^2\,\dd\bar\rho_t^\nu\dd t 
- \KL(\rho_{t_{j-1}}^\nu\,\|\,\rho_{t_{j-1}}^\mu) - (\varepsilon_j\delta_{1,j} + \delta_{2,j}).
\end{align*}
Moreover, these remainder terms satisfy
\begin{align*}
\sum_{j=1}^m \varepsilon_j \delta_{1,j} &\leq \sup_{j\in[m]}\Lip\big(v_{t_{j-1}, 1}v_{t_{j-1}, 0}^\dagger\big)\cdot \sqrt{\sum_{j=1}^m\frac{\W_2^2(\rho_{t_{j-1}}^\nu, \rho_{t_j}^\nu)}{\varepsilon_j}}\cdot \max_{j\in[m]}\varepsilon_j\\
\sum_{j=1}^m \delta_{2,j} &\leq \frac{1}{4} \left[\sup_{t\in[0,1]} \Big\|\frac{\partial}{\partial t}v_{t,0}\Delta v_{t,0}^\dagger\Big\|_{\m C^0} + \sup_{j\in[m]}\Lip\big(v_{t_{j-1},0}\Delta v_{t_{j-1},0}^\dagger\big)\sqrt{\sum_{j=1}^m \frac{\W_2^2(\rho_{t_{j-1}}^\nu, \rho_{t_j}^\nu)}{\varepsilon_j}}\right]\cdot\max_{j\in[m]} \varepsilon_j.
\end{align*}
\myqed
\end{lemma}

Though $u_{t_{j-1},1}/u_{t_{j-1},0}$ and $v_{t_{j-1},1}/v_{t_{j-1},0}$ appear in the second-order expansion of Schr\"odinger potentials with respect to the regularization coefficient $\varepsilon_j$ as shown in Theorem~\ref{thm: asymp_schro_potential}, we would like to emphasize that the sum of their integrations only depends on $\nabla\Phi_t^\mu$, which is only related to the first-order time derivative $\partial_t\rho_t^\mu$, due to additional cancellation (c.f. Appendix~\ref{app: simplification} for more details).

Now, we are ready to present the proof of the main theorem of this paper.
\paragraph{Proof of Theorem~\ref{thm: multiSB_stable}}
For each $j$, let $\varepsilon_j = t_j - t_{j-1}.$ It follows from Theorem~\ref{thm: asymp_schro_potential} with $K=3$ and the expressions \eqref{eqn: explicit_construct}, that
\begin{align*}
&\quad\,\int_{\mb T^d}\phi_j^\mu\,\dd\rho_{t_{j-1}}^\nu + \int_{\mb T^d}\psi_j^\mu\,\dd\rho_{t_j}^\nu\\
&= \int_{\mb T^d} \left(\varepsilon_j\log u_{t_{j-1},0}+ \varepsilon_j^2 \frac{u_{t_{j-1}, 1}}{u_{t_{j-1}, 0}}\right)\,\dd\rho_{t_{j-1}}^\nu + \int_{\mb T^d} \left(\varepsilon_j\log v_{t_{j-1},0}+ \varepsilon_j^2 \frac{v_{t_{j-1}, 1}}{v_{t_{j-1}, 0}}\right)\,\dd\rho_{t_j}^\nu + O(\varepsilon_j^3),
\end{align*}
where the big-O notation omits a constant depending on
\begin{align*}
\alpha\in(0, 1),\,\, d,\,\, L_\rho,\,\,\mbox{and}\,\, \sup\big\{\|\rho_t^{(k)}\|_{\m C^{\max\{2+\lceil\frac{d+1}{2}\rceil, 12\}+k',\alpha}}: t\in[0, 1],\,\, 2k+k'=14,\,\,k,k'\in\mb N\big\}.
\end{align*}
Let $\varepsilon_{\rm max} \coloneqq \max_{j\in[m]}\varepsilon_j$. Combining with~\eqref{eqn: simplify_KL_MSB}, \eqref{eqn:Conforti21} and the definition of $S_j$ in Proposition~\ref{lem:Conforti} we have 
\begin{align}\label{eqn: temp_KL-ub}
&\KL(R^{\bm\nu^m}\,\|\,R^{\bm\mu^m})
\leq \sum_{j=1}^m \big[S_j + O(\varepsilon_j^2)\big] + \sum_{j=0}^m\KL(\rho_{t_j}^\nu\,\|\,\rho_{t_j}^\mu),
\end{align}
which from Lemma~\ref{lem: simplification} satisfies
\begin{align}\label{eqn: le e max}
\le \KL(\rho_1^\nu\,\|\,\rho_1^\mu) + \frac{1}{2}\sum_{j=1}^m\int_{t_{j-1}}^{t_j}\!\int_{\mb T^d}\Big\|\nabla\Phi_{t_{j-1}}^\mu - \nabla\bar \Phi_t^\nu - \frac{1}{2}\nabla\log\frac{\rho_{t_{j-1}}^\mu}{\bar\rho_t^\nu}\Big\|^2\,\dd\bar\rho_t^\nu\dd t + \sum_{j=1}^m\big(\varepsilon_j|\delta_{1,j}| + |\delta_{2,j}|\big) + O(\varepsilon_{\rm max}).
\end{align}
Here, we used $ \sum_{j=1}^m O(\varepsilon_j^2) \le O(\varepsilon_{\rm max})$.
Now Lemma~\ref{lem: simplification}  also implies
\begin{align*}
\sum_{j=1}^m\big(\varepsilon_j|\delta_{1,j}| + |\delta_{2,j}|\big)
\leq O(\varepsilon_{\max}), 
\end{align*}
where the big-O notation omits a constant depending only on
\begin{align*}
\sum_{j=1}^m \frac{\W_2^2(\rho_{t_{j-1}}^\nu, \rho_{t_j}^\nu)}{\varepsilon_j},\,\,
\sup_{t\in[0,1]} \Big\|\frac{\partial}{\partial t}\Phi_t^\mu\Big\|_{\m C^2},\,\,
\sup_{t\in[0,1]}\|\Phi_t^\mu\|_{\m C^{4}},\,\,
\sup\big\{\|\rho_t^{(k)}\|_{\m C^{4-2k,\alpha}}:t\in[0, 1],\, k=0,1,2\big\},\,\,
\alpha,\,\,
\mx{and}\,\, d
\end{align*}
for any fixed $\alpha \in(0,1)$.
 
To control the term $\Big\|\nabla\Phi_{t_{j-1}}^\mu - \nabla\bar \Phi_t^\nu - \frac{1}{2}\nabla\log\frac{\rho_{t_{j-1}}^\mu}{\bar\rho_t^\nu}\Big\|^2$ in \eqref{eqn: le e max}, observe that  by Cauchy--Schwarz inequality, we have
for any constant $C>0$ that \begin{align*}
\Big\|\nabla\Phi_{t_{j-1}}^\mu - \nabla\bar \Phi_t^\nu - \frac{1}{2}\nabla\log\frac{\rho_{t_{j-1}}^\mu}{\bar\rho_t^\nu}\Big\|^2
\leq \big(1 + \sqrt{C}\varepsilon_j\big)\bigg[\Big\|\nabla\bar \Phi_{t_{j-1}}^\mu - \nabla\bar \Phi_t^\nu - \frac{1}{2}\nabla\log\frac{\rho_{t_{j-1}}^\mu}{\bar\rho_t^\nu}\Big\|^2 + \frac{\big\|\nabla\Phi_{t_{j-1}}^\mu - \nabla\bar\Phi_{t_{j-1}}^\mu\big\|^2}{\sqrt{C}\varepsilon_j}\bigg].
\end{align*}
We will prove that whenever $\varepsilon_j = t_j - t_{j-1} < \bar \varepsilon_{\rm thres}$,
\begin{align}\label{eqn: change_velocity}
\int_{\mb T^d} \big\|\nabla\Phi_{t_{j-1}}^\mu - \nabla\bar\Phi_{t_{j-1}}^\mu\big\|^2\,\dd\rho_{t_{j-1}}^\nu \leq C \varepsilon_j^2
\end{align}
holds for a constant $C>0$ that depends  only on $L_\rho$ and $\sup_{t\in[0, 1]}\|\Phi_t^\mu\|_{\m C^3}$, and for a constant $\bar \varepsilon_{\rm thres} > 0$  that depends only on $\sup_{t\in[0, 1]}\|\Phi_t^\mu\|_{\m C^2}$.
{Assuming~\eqref{eqn: change_velocity}, we get}
\begin{align*}
&\quad\,\int_{\mb T^d}\Big\|\nabla\Phi_{t_{j-1}}^\mu - \nabla\bar \Phi_t^\nu - \frac{1}{2}\nabla\log\frac{\rho_{t_{j-1}}^\mu}{\bar\rho_t^\nu}\Big\|^2\,\dd\bar\rho_t^\nu\\
&\leq \big(1 + \sqrt{C}\varepsilon_j\big)\int_{\mb T^d} \Big\|\nabla\bar \Phi_{t_{j-1}}^\mu - \nabla\bar \Phi_t^\nu - \frac{1}{2}\nabla\log\frac{\rho_{t_{j-1}}^\mu}{\bar\rho_t^\nu}\Big\|^2\,\dd\bar\rho_t^\nu + \sqrt{C}\varepsilon_j + C\varepsilon_j^2.
\end{align*}
Combining with \eqref{eqn: le e max} 
it yields
\begin{align*}
\KL(R^{\bm\nu^m}\,\|\,R^{\bm\mu^m})
\le 
\KL(\rho_1^\nu\,\|\,\rho_1^\mu) + \frac{1}{2}\sum_{j=1}^m\int_{t_{j-1}}^{t_j}\!\int_{\mb T^d}\Big\|\nabla\bar\Phi_{t_{j-1}}^\mu - \nabla\bar \Phi_t^\nu - \frac{1}{2}\nabla\log\frac{\rho_{t_{j-1}}^\mu}{\bar\rho_t^\nu}\Big\|^2\,\dd\bar\rho_t^\nu\dd t + O(\varepsilon_{\rm max})
\end{align*}
which will complete the proof, modulo the proof of \eqref{eqn: change_velocity}.
\qed

\begin{proof}[Proof of estimate~\eqref{eqn: change_velocity}]
Now, we only need to prove~\eqref{eqn: change_velocity}. Without loss of generality, we can assume $j = 1$ and focus on the time interval $[0, t_1]$. We will need to prove:
$$
\int_{\mb T^d} \big\|\nabla\Phi_{0}^\mu - \nabla\bar\Phi_{0}^\mu\big\|^2\,\dd\rho_{0}^\nu \leq C t_1^2.
$$

Note that in general there is no reason why there is a relation between $\nabla\Phi_0^\mu$ and $\nabla\bar\Phi_{0}^\mu$ as in the estimate \eqref{eqn: change_velocity}, because $\nabla \bar\Phi_{0}^\mu$ depends only on the marginals $\rho^\mu_0, \rho^\mu_{t_1}$, while $\nabla\Phi_{0}^\mu$ (without the bar) is the vector field determined by the {\em given}  flow itself.
This is where the regularity assumption on the flow ($\rho_t$ and $\Phi_t^\mu$),  Assumption~\ref{assump: density_expansion}, is relevant; in other words, without such an assumption we cannot expect an estimate like ~\eqref{eqn: change_velocity}, nor the estimate in Theorem~\ref{thm: multiSB_stable}. However, even with such a regularity assumption, since $\nabla\Phi_{0}^\mu$ and $\nabla\bar \Phi_{0}^\mu$ share only  the initial and terminal marginals $\rho^\mu_0, \rho^\mu_{t_1}$, we will carry out estimates on their integrals individually then compare them, rather than trying to get pointwise estimates on their differences.

Using the fact that $L_\rho \leq \rho_t^\mu(x), \rho_t^\nu(x) \leq L_\rho^{-1}$ for all $x\in\mb T^d$ and $t\in[0, 1]$, we have
\begin{align}\label{eqn: upper-bd-deltav}
\begin{aligned}
&\quad\,\frac{1}{2}\int_{\mb T^d}\big\|\nabla\Phi_0^\mu - \nabla\bar\Phi_0^\mu\big\|^2\,\dd\rho_0^\nu
\leq \frac{1}{2L_\rho^2}\int_{\mb T^d}\big\|\nabla\Phi_0^\mu - \nabla\bar\Phi_0^\mu\big\|^2\,\dd\rho_0^\mu\\
&= \frac{1}{2L_\rho^2}\left[\int_{\mb T^d}\big\|\nabla\bar\Phi_0^\mu\big\|^2\,\dd\rho_0^\mu - 2\int_{\mb T^d}\big\langle\nabla\bar\Phi_0^\mu, \nabla\Phi_0^\mu\big\rangle\,\dd\rho_0^\mu + \int_{\mb T^d}\big\|\nabla\Phi_0^\mu\big\|^2\,\dd\rho_0^\mu\right].
\end{aligned}
\end{align}

\vspace{0.5em}
\noindent\underline{Step 1: Control the first integration in~\eqref{eqn: upper-bd-deltav}.}
Let $T^\mu$ be the optimal transport map from $\rho_0^\mu$ to $\rho_{t_1}^\mu$. Since $(\nabla\bar\Phi_t^\mu)_{t\in[0, t_1]}$ is the velocity vector field of constant-speed Wasserstein-2 geodesics $(\bar\rho_t^\mu)_{t\in[0, t_1]}$,  we have
\begin{align}\label{eqn: velocity2transport}
\nabla\bar\Phi_0^\mu(x) = \frac{T^\mu(x) - x}{t_1}.
\end{align}
Therefore, we have
\begin{align*}
\frac{1}{2}\int_{\mb T^d}\big\|\nabla\bar\Phi_0^\mu\big\|^2\,\dd\rho_0^\mu = \frac{1}{2t_1^2}\int_{\mb T^d}\|T^\mu(x) - x\|^2\,\dd\rho_0^\mu = \frac{1}{t_1^2}\W_2^2\big(\rho_0^\mu, \rho_{t_1}^\mu\big).
\end{align*}
Next, we aim to connect $\W_2^2\big(\rho_0^\mu, \rho_{t_1}^\mu\big)$ with the velocity vector field $\{\nabla\Phi_t^\mu: 0\leq t\leq t_1\}$ using Assumption~\ref{assump: density_expansion}. Consider the ordinary differential equation
\begin{align*}
\frac{\dd X_t}{\dd t} = \nabla\Phi_t^\mu(X_t)
\quad\mx{with}\quad X_0\sim \rho_0^\mu.
\end{align*}
Then, we have $X_t\sim\rho_t^\mu$ for all $t\in[0, 1]$. Note that
\begin{align*}
\frac{\dd}{\dd t}\nabla\Phi_t^\mu(X_t)
&= \nabla(\partial_t \Phi_t^\mu)(X_t) + \nabla^2\Phi_t^\mu(X_t)\frac{\dd X_t}{\dd t}
= \nabla(\partial_t \Phi_t^\mu)(X_t) + \nabla^2\Phi_t^\mu(X_t)\nabla\Phi_t^\mu(X_t).
\end{align*}
Thus, the Wasserstein-2 distance can be upper bounded by
\begin{align*}
\W_2^2\big(\rho_0^\mu, \rho_{t_1}^\mu\big) 
&\leq \frac{1}{2}\mb E\|X_{t_1} - X_0\|^2
= \frac{1}{2}\mb E\bigg\|\int_0^{t_1}\frac{\dd X_t}{\dd t}\,\dd t\bigg\|^2
= \frac{1}{2}\mb E\bigg\|\int_0^{t_1} \nabla\Phi_t^\mu(X_t)\,\dd t\bigg\|^2\\
&= \frac{1}{2}\mb E\bigg\| t_1\nabla\Phi_0^\mu(X_0) + \int_0^{t_1}\!\!\int_{0}^t\frac{\dd}{\dd s}\nabla\Phi_s^\mu(X_s)\,\dd s\dd t\bigg\|^2\\
&= \frac{1}{2}\mb E\bigg\|t_1\nabla\Phi_0^\mu(X_0) + \frac{t_1^2}{2}\big[\nabla(\partial_t\Phi_0^\mu)(X_0) + \nabla^2\Phi_0^\mu(X_0)\nabla\Phi_0^\mu(X_0)\big] + A(t_1)\bigg\|^2,
\end{align*}
where for simplicity, we let
\begin{align*}
A(t_1) = \int_0^{t_1}\!\!\int_0^t\!\!\int_0^s \frac{\dd^2}{\dd u^2}\nabla\Phi_u^\mu(X_u)\,\dd u\dd s\dd t = O(t_1^3).
\end{align*}
We also observe that the above estimate also gives 
$$\W_2^2\big(\rho_0^\mu, \rho_{t_1}^\mu\big)  = O(t_1^2).$$
Here, the big-O notation in these  has  hidden constants depending only on the regularity assumption, Assumption~\ref{assump: density_expansion}.

Combining together, we have derived 
\begin{align}\label{eqn: W2-to-velocity}
\begin{aligned}
\frac{1}{2}\int_{\mb T^d}\big\|\nabla\bar\Phi_0^\mu\big\|^2\,\dd\rho_0^\mu
&\leq \frac{1}{2}\mb E\bigg\|\nabla\Phi_0^\mu(X_0) + \frac{t_1}{2}\big[\nabla(\partial_t\Phi_0^\mu)(X_0) + \nabla^2\Phi_0^\mu(X_0)\nabla\Phi_0^\mu(X_0)\big] + \frac{A(t_1)}{t_1}\bigg\|^2\\
&= \frac{1}{2}\int_{\mb T^d}\big\|\nabla\Phi_0^\mu\big\|^2\,\dd\rho_0^\mu + \frac{t_1}{2}\int_{\mb T^d}\big\langle\nabla\Phi_0^\mu, \nabla(\partial_t\Phi_0^\mu) + (\nabla^2\Phi_0^\mu)\nabla\Phi_0^\mu\big\rangle\,\dd\rho_0^\mu + O(t_1^2).
\end{aligned}
\end{align}

\vspace{0.5em}
\noindent\underline{Step 2: Control the second term in~\eqref{eqn: upper-bd-deltav}.} 
Next, we will control the cross-product term 
$\int_{\mb T^d}\big\langle\nabla\bar\Phi_0^\mu, \nabla\Phi_0^\mu\big\rangle\,\dd\rho_0^\mu$ in~\eqref{eqn: upper-bd-deltav}. Define
\begin{align*}
H_t \coloneqq \int_{\mb T^d}\Phi_0^\mu\,\dd\rho_t^\mu,\quad t\in[0, t_1].
\end{align*}
We will analyze $H_{t_1} - H_0$ in two different approaches. On the one hand, we have
\begin{align}
H_{t_1} - H_0
&= \int_{\mb T^d}\Phi_0^\mu\,\dd\rho_{t_1}^\mu - \int_{\mb T^d}\Phi_0^\mu\,\dd\rho_0^\mu
= \int_{\mb T^d}\left(\Phi_0^\mu\big(T^\mu(x)\big) - \Phi_0^\mu(x)\right)\,\dd\rho_0^\mu\notag\\
&= \int_{\mb T^d} \big\langle\nabla\Phi_0^\mu(x), T^\mu(x) - x\big\rangle + \frac{1}{2}\big\langle T^\mu(x) - x, \nabla^2\Phi_0^\mu(x)\big(T^\mu(x) - x\big)\big\rangle\,\dd\rho_0^\mu + O\big(\|T^\mu - \id\|_{L^3(\rho_0^\mu)}^3\big)\notag\\
&= t_1\int_{\mb T^d}\big\langle\nabla\Phi_0^\mu, \nabla\bar\Phi_0^\mu\big\rangle\,\dd\rho_0^\mu + \frac{t_1^2}{2}\int_{\mb T^d}\big\langle\nabla\bar\Phi_0^\mu, (\nabla^2\Phi_0^\mu)\nabla\bar\Phi_0^\mu\big\rangle\,\dd\rho_0^\mu + O\big(\|T^\mu - \id\|^3_{L^3(\rho_0^\mu)}\big).\label{eqn: Ht-H0_approach1}
\end{align}
Here, the big-O notation omits a constant only depending on $\|\Phi_0^\mu\|_{\m C^3(\mb T^d)}$.
On the other hand, we have
\begin{align*}
H_{t_1} - H_0 
&= \int_0^{t_1}\frac{\dd H_t}{\dd t}\,\dd t
= \int_0^{t_1}\!\!\int_{\mb T^d}\Phi_0^\mu(x)\partial_t\rho_t^\mu(x)\,\dd x\dd t
\stackrel{\ri}{=} \int_0^{t_1}\!\!\int_{\mb T^d}\big\langle\nabla\Phi_0^\mu, \nabla\Phi_t^\mu\big\rangle\,\dd\rho_t^\mu\dd t\\
&= t_1\int_{\mb T^d}\big\|\nabla\Phi_t^\mu\big\|^2\,\dd\rho_0^\mu + \int_0^{t_1}\!\!\int_0^t\frac{\dd}{\dd s}\bigg[\int_{\mb T^d}\big\langle\nabla\Phi_0^\mu, \nabla\Phi_s^\mu\big\rangle\,\dd\rho_s^\mu\bigg]\,\dd s\dd t.
\end{align*}
Here, (i) follows from the continuity equation and integration by parts. Note that
\begin{align*}
\frac{\dd}{\dd s}\bigg[\int_{\mb T^d}\big\langle\nabla\Phi_0^\mu, \nabla\Phi_s^\mu\big\rangle\,\dd\rho_s^\mu\bigg]
&= \int_{\mb T^d}\big\langle\nabla\Phi_0^\mu, \nabla(\partial_s\Phi_s^\mu)\big\rangle\,\dd\rho_s^\mu + \int_{\mb T^d}\big\langle\nabla\Phi_0^\mu, \nabla\Phi_s^\mu\big\rangle\,\partial_s\rho_s^\mu\\
&= \int_{\mb T^d}\big\langle\nabla\Phi_0^\mu, \nabla(\partial_s\Phi_s^\mu)\big\rangle\,\dd\rho_s^\mu + \int_{\mb T^d}\Big\langle \nabla\big[\big(\nabla\Phi_0^\mu\big)^\top\nabla\Phi_s^\mu\big], \nabla\Phi_s^\mu\Big\rangle\,\dd\rho_s^\mu.
\end{align*}
Thus, we have
\begin{align}
H_{t_1} - H_0 
&= t_1\int_{\mb T^d}\big\|\nabla\Phi_t^\mu\big\|^2\,\dd\rho_0^\mu + \frac{t_1^2}{2}\bigg[\int_{\mb T^d}\big\langle\nabla\Phi_0^\mu, \nabla(\partial_t\Phi_0^\mu)\big\rangle\,\dd\rho_0^\mu + \int_{\mb T^d}\Big\langle \nabla\big\|\nabla\Phi_0^\mu\big\|^2, \nabla\Phi_0^\mu\Big\rangle\,\dd\rho_0^\mu\bigg]+ O(t_1^3)\notag\\
&= t_1\int_{\mb T^d}\big\|\nabla\Phi_t^\mu\big\|^2\,\dd\rho_0^\mu + \frac{t_1^2}{2}\bigg[\int_{\mb T^d}\left(\big\langle\nabla\Phi_0^\mu, \nabla(\partial_t\Phi_0^\mu)\big\rangle + 2\big\langle \big(\nabla^2\Phi_0^\mu\big)\nabla\Phi_0^\mu, \nabla\Phi_0^\mu\big\rangle\right)\,\dd\rho_0^\mu\bigg]+ O(t_1^3)\label{eqn: Ht-H0_approach2},
\end{align}
where the big-O notation omits a constant depending on $\sup_{t\in[0, t_1]}\|\Phi_t^\mu\|_{\m C^3(\mb T^d)}$ and $L_\rho$. Recall  that $\W_2(\rho_0^\mu, \rho_{t_1}^\mu) = O(t_1)$. Comparing~\eqref{eqn: Ht-H0_approach1} and~\eqref{eqn: Ht-H0_approach2} yields
\begin{align}\label{eqn: cross-term}
\begin{aligned}
&\int_{\mb T^d}\big\langle\nabla\Phi_0^\mu, \nabla\bar\Phi_0^\mu\big\rangle\,\dd\rho_0^\mu + \frac{t_1}{2}\int_{\mb T^d}\big\langle\nabla\bar\Phi_0^\mu, (\nabla^2\Phi_0^\mu)\nabla\bar\Phi_0^\mu\big\rangle\,\dd\rho_0^\mu + t_1^{-1}O\big(\|T^\mu - \id\|_{L^3(\rho_0^\mu)}^3\big)\\
&= \int_{\mb T^d}\big\|\nabla\Phi_t^\mu\big\|^2\,\dd\rho_0^\mu + \frac{t_1}{2}\bigg[\int_{\mb T^d}\big\langle\nabla\Phi_0^\mu, \nabla(\partial_t\Phi_0^\mu)\big\rangle + 2\big\langle \big(\nabla^2\Phi_0^\mu\big)\nabla\Phi_0^\mu, \nabla\Phi_0^\mu\big\rangle\,\dd\rho_0^\mu\bigg]+ O(t_1^2).
\end{aligned}
\end{align}

\vspace{0.5em}
\noindent\underline{Step 3: Estimate $\|T^\mu - \id\|_{L^3(\rho_0^\mu)}^3$.}
Using the estimate~\eqref{eqn: W2-to-velocity} and~\eqref{eqn: cross-term}, there exists a constant $C>0$ depending only on $L_\rho$ and $\sup_{t\in[0, 1]}\|\Phi_0^\mu\|_{\m C^2(\mb T^d)}$, such that
\begin{align*}
\int_{\mb T^d}\big\|\nabla\Phi_0^\mu - \nabla\bar\Phi_0^\mu\big\|^2\,\dd\rho_0^\mu
\leq Ct_1.
\end{align*}
Therefore, using the above estimate and~\eqref{eqn: velocity2transport}, we have
\begin{align*}
\int_{\mb T^d}\|T^\mu - \id\|^3\,\dd\rho_0^\mu
&= t_1^3\int_{\mb T^d}\big\|\nabla\bar\Phi_0^\mu\big\|^3\,\dd\rho_0^\mu
\leq 4t_1^3\int_{\mb T^d}\big\|\nabla\Phi_0^\mu\big\|^3 + \big\|\nabla\Phi_0^\mu - \nabla\bar\Phi_0^\mu\big\|^3\,\dd\rho_0^\mu\\
&\leq 4t_1^3\int_{\mb T^d}\big\|\nabla\Phi_0^\mu\big\|^3\,\dd\rho_0^\mu + 4Ct_1^4\big\|\nabla\Phi_0^\mu - \nabla\bar\Phi_0^\mu\big\|_{L^\infty}.
\end{align*}
Note that we have the following $L^\infty$ estimate
\begin{align*}
\big\|\nabla\Phi_0^\mu - \nabla\bar\Phi_0^\mu\big\|_{L^\infty}
\leq \|\Phi_0^\mu\|_{\m C^1(\mb T^d)} + t_1^{-1}\|T^\mu - \id\|_{L^\infty}
\leq \|\Phi_0^\mu\|_{\m C^1(\mb T^d)} + t_1^{-1}2\pi\sqrt{d}.
\end{align*}
Thus, combining with $\rho_0^\mu(x) \leq L_\rho^{-1}$, we have
\begin{align}\label{eqn: L^3-estimate}
\int_{\mb T^d} \|T^\mu - \id\|^3\,\dd\rho_0^\mu
\leq \big[4L_\rho^{-1}\|\Phi_0^\mu\|_{\m C^1}^3 + 4C\big(t_1\|\Phi_0^\mu\|_{\m C^1} + 2\pi\sqrt{d}\big)\big] t_1^3.
\end{align}

\vspace{0.5em}
\noindent\underline{Step 4. An $O(t_1^2)$-order estimate.}
Using the estimate~\eqref{eqn: cross-term} and~\eqref{eqn: L^3-estimate}, we get
\begin{align*}
\begin{aligned}
&\int_{\mb T^d}\big\langle\nabla\Phi_0^\mu, \nabla\bar\Phi_0^\mu\big\rangle\,\dd\rho_0^\mu + \frac{t_1}{2}\int_{\mb T^d}\big\langle\nabla\bar\Phi_0^\mu, (\nabla^2\Phi_0^\mu)\nabla\bar\Phi_0^\mu\big\rangle\,\dd\rho_0^\mu + O\big(t_1^2\big)\\
&= \int_{\mb T^d}\big\|\nabla\Phi_t^\mu\big\|^2\,\dd\rho_0^\mu + \frac{t_1}{2}\bigg[\int_{\mb T^d}\big\langle\nabla\Phi_0^\mu, \nabla(\partial_t\Phi_0^\mu)\big\rangle + 2\big\langle \big(\nabla^2\Phi_0^\mu\big)\nabla\Phi_0^\mu, \nabla\Phi_0^\mu\big\rangle\,\dd\rho_0^\mu\bigg]+ O(t_1^2).
\end{aligned}
\end{align*}
Using the above result and the estimate~\eqref{eqn: W2-to-velocity}, there exists a constant $C'>0$ depending only on $L_\rho$ and $\sup_{t\in[0, 1]}\|\Phi_t^\mu\|_{\m C^3(\mb T^d)}$, such that
\begin{align*}
&\int_{\mb T^d}\big\|\nabla\Phi_0^\mu - \nabla\bar\Phi_0^\mu\big\|^2\,\dd\rho_0^\mu
\leq t_1\int_{\mb T^d}\left(\big\langle \nabla\bar\Phi_0^\mu, (\nabla^2\Phi_0^\mu)\nabla\bar\Phi_0^\mu\big\rangle - \big\langle \nabla\Phi_0^\mu, (\nabla^2\Phi_0^\mu)\nabla\Phi_0^\mu\big\rangle\right)\,\dd\rho_0^\mu + C't_1^2\\
&\stackrel{}{=} t_1 \int_{\mb T^d}\left(\big[\nabla\bar\Phi_0^\mu - \nabla\Phi_0^\mu\big]^\top \big(\nabla^2\Phi_0^\mu\big)\big[\nabla\bar\Phi_0^\mu - \nabla\Phi_0^\mu\big] + 2\big[\nabla\bar\Phi_0^\mu - \nabla\Phi_0^\mu\big]^\top(\nabla^2\Phi_0^\mu)\nabla\Phi_0^\mu\right)\,\dd\rho_0^\mu + C't_1^2\\
&\leq t_1\sup_x\matnorm{\nabla^2\Phi_0^\mu}_{\rm op}\int_{\mb T^d}\left(\big\|\nabla\Phi_0^\mu - \nabla\bar\Phi_0^\mu\big\|^2 + 2\big\|\nabla\Phi_0^\mu\big\|\cdot\big\|\nabla\Phi_0^\mu - \nabla\bar\Phi_0^\mu\big\|\right)\,\dd\rho_0^\mu + C't_1^2.
\end{align*}
By using $\delta_\Phi^2$ to denote the left-hand side, we get
\begin{align*}
\delta_\Phi^2 \leq t_1\sup_x\matnorm{\nabla^2\Phi_0^\mu}_{\rm op}\bigg[\delta_\Phi^2  + 2\sup_x\big\|\nabla\Phi_0^\mu(x)\big\|\delta_\Phi\bigg] + C't_1^2.
\end{align*}
This implies
\begin{align*}
\frac{\delta_\Phi}{t_1} \leq \frac{2\sup_{x}\matnorm{\nabla^2\Phi_0^\mu(x)}_{\rm op}\sup_x\|\nabla\Phi_0^\mu(x)\|}{1 - t_1\sup_x\matnorm{\nabla^2\Phi_0^\mu(x)}_{\rm op}} + \sqrt{\frac{C'}{1 - t_1\sup_x\matnorm{\nabla^2\Phi_0^\mu(x)}}_{\rm op}}
\end{align*}
whenever $t_1\sup_x\matnorm{\nabla^2\Phi_0^\mu(x)}_{\rm op} < 1$. Thus, we know there is are constants $\bar \varepsilon_{\rm thres}> 0$ depending only on $\|\Phi_0^\mu\|_{\m C^2}$ and $C'' > 0$ depending only on $L_\rho, \sup_{t\in[0, 1]}\|\Phi_t^\mu\|_{\m C^3}$, such that
\begin{align*}
\int_{\mb T^d}\big\|\nabla\Phi_0^\mu - \nabla\bar\Phi_0^\mu\big\|^2\,\dd\rho_0^\mu \leq C'' t_1^2
\qquad\mx{for all}\,\,\, \,t_1\leq \bar\varepsilon_{\rm thres}.
\end{align*}
This completes the proof of~\eqref{eqn: change_velocity}, thus, of Theorem~\ref{thm: multiSB_stable}.
\end{proof}

\section*{Acknowledgement}
GS is supported by Canadian Institutes of Health Research (CIHR) and Michael Smith Health Research BC Scholar Program. 
YHK is partially supported by the Natural Sciences and Engineering Research Council of Canada (NSERC), with Discovery Grant RGPIN-2019-03926 and RGPIN-2025-06747. Both GS and YHK are partially supported by Exploration Grant (NFRFE-2019-00944) from the New Frontiers in Research Fund (NFRF). GS and YHK are members of the Kantorovich Initiative (KI), which is supported by the PIMS Research Network (PRN) program of the Pacific Institute for the Mathematical Sciences (PIMS). We thank PIMS for their generous support; report identifier PIMS-20260415-PRN01

\bibliography{ref}
\bibliographystyle{plainnat}
\newpage
\appendix
\section{Proof of Corollaries}

\subsection{Proof of Corollary~\ref{prop: vs_conforti}}\label{app: pf vs_conforti}

We first collect a few properties of Holder space. The detailed proofs are omitted.
\begin{proposition}\label{prop: Holder}
Let $k\in\mb N$ and $\alpha\in(0, 1)$. The Holder space $\m C^{k, \alpha}(\mb T^d)$ has the following properties.
\begin{itemize}
\item For any $f, g\in\m C^{k, \alpha}(\mb T^d)$, we have $fg\in\m C^{k,\alpha}(\mb T^d)$. Moreover, we have the estimate
\begin{align*}
\|fg\|_{\m C^{k,\alpha}} \leq C_1(k, \alpha, d)\|f\|_{\m C^{k,\alpha}}\|g\|_{\m C^{k,\alpha}}.
\end{align*}

\item For any $f\in\m C^{k,\alpha}(\mb T^d)$ and $g\in\m C^{k,\alpha}(\mb T^d; \mb T^d)$, we have
\begin{align*}
\|f\circ g\|_{\m C^{k,\alpha}} \leq 
\begin{cases}
C_2\big(\alpha, d, \lip(g)\big) \|f\|_{\m C^{k,\alpha}},\qquad\qquad & k=0\\
C_2'\big(k,\alpha,d,\|g\|_{\m C^{k,\alpha}}\big)\|f\|_{\m C^{k,\alpha}},\qquad\qquad & k \geq 1
\end{cases}
\quad.
\end{align*}

\item Let $A(x)$ be a symmetric matrix-valued function on $\mb T^d$, such that $A(x) \geq \lambda I_d$ is uniformly positive definite for all $x\in\mb T^d$. Define the Holder norm of $A(x)$ by
\begin{align*}
\|A\|_{\m C^{k,\alpha}} \coloneqq
\max_{|\beta| \leq k} \sup_{x\in\mb T^d}\matnorm{A(x)}_{\rm op} + \max_{\|\beta\| = k}\sup_{x\neq y} \frac{\matnorm{D^\beta A(x) - D^\beta A(y)}_{\rm op}}{|x-y|^\alpha}.
\end{align*}
Then, we have
\begin{align*}
\|A^{-1}\|_{\m C^{k,\alpha}} \leq C_3(k,\alpha, d, \lambda, \|A\|_{\m C^{k,\alpha}}).
\end{align*}
\myqed
\end{itemize}
\end{proposition}

\noindent
We will prove that there exists a constant $C>0$ depending only on 
\begin{align*}
d,\,\,
L_\rho, \,\,
\sup_{t\in[t_{j-1}, t_j]}\|\rho_t^\mu\|_{\m C^{2,1/2}}, 
\sup_{t\in[t_{j-1}, t_j]}\|\partial_t\rho_t^\mu\|_{\m C^{1,1/2}},
\quad\mx{and}\quad
\sup_{t\in[t){j-1}, t_j]}\|\Phi_t^\mu\|_{\m C^1},
\end{align*}
such that
\begin{align}\label{eqn: same-marginal-cts-discrete-1piece}
\int_{t_{j-1}}^{t_j}\!\int_{\mb T^d}\Big\|\nabla\bar \Phi_{t_{j-1}}^\mu - \nabla\bar\Phi_t^\mu - \frac{1}{2}\nabla\log\frac{\rho_{t_{j-1}}^\mu}{\bar\rho_t^\mu}\Big\|^2\,\dd \bar\rho_t^\mu\dd t 
\leq C(t_j - t_{j-1})^2
\end{align}
for all $j\in[m]$. 

Without loss of generality, we assume $j = 1$.
Let $T^\mu$ be the optimal transport map from $\rho_0^\mu$ to $\rho_{t_1}^\mu$. Recall that $(\bar\rho_t^\mu)_{0\leq t\leq t_1}$ is the constant-speed geodesics connecting $\rho_0^\mu$ and $\rho_{t_1}^\mu$. Let $T_t$ denote the optimal transport map from $\bar \rho_0^\mu$ to $\bar \rho_{t}^\mu$, defined by
\begin{align}\label{eqn: ot_interpolate}
T_t(x) = \frac{t}{t_1}T^\mu(x) + \Big(1 - \frac{t}{t_1}\Big)x,
\qquad t\in[0, t_1].
\end{align}
Since $\nabla\bar\Phi_t^\mu$ is the velocity vector field of piecewise Wasserstein geodesics $(\bar\rho_t^\mu)_{t\in[0, 1]}$, we have
\begin{align}\label{eqn: geodesics_vecfield}
\nabla\bar\Phi_t^\mu\Big(\frac{t}{t_1}T^\mu(x) + \frac{t_1 - t}{t_1} x\Big) = \frac{T^\mu(x) - x}{t_1},
\qquad t\in[0, t_1].
\end{align}
Applying Cauchy--Schwarz inequality yields
\begin{align*}
\Big\|\nabla\bar \Phi_{0}^\mu - \nabla\bar\Phi_t^\mu - \frac{1}{2}\nabla\log\frac{\rho_{{0}}^\mu}{\bar\rho_t^\mu}\Big\|^2
\leq 2\big\|\nabla\bar \Phi_{0}^\mu - \nabla\bar\Phi_t^\mu\big\|^2
+ \frac{1}{2}\Big\|\nabla\log\frac{\rho_{0}^\mu}{\bar\rho_t^\mu}\Big\|^2,
\quad t\in[0, t_1].
\end{align*}
Therefore, we only need to control the integration these two terms separately. 

Before proceeding, we will introduce the following regularity results of the solution to Monge--Ampere equations. The proofs are postponed to the end of this section.

\begin{lemma}\label{lem: Kanpot-HessLB}
Let $\phi_t$ and $\psi_t$ be the Kantorovich potentials for solving the optimal transport problem between $\rho_0^\mu$ and $\rho_t^\mu$. That is, 
\begin{align*}
(\id - \nabla\phi_t)_\# \rho_0^\mu = \rho_t^\mu
\qquad\mx{and}\qquad
(\id - \nabla\psi_t)_\# \rho_t^\mu = \rho_0^\mu,
\end{align*}
where the sums $\id - \nabla\phi_t$ and $\id - \nabla\psi_t$ are
to be intended modulo $2\pi\mb Z^d$.
Then for any $\alpha\in(0, 1)$, there exists $\sigma_{\rm min},\sigma_{\rm max} > 0$ depending on $d$, $L_\rho$,  $\|\rho_0^\mu\|_{\m C^{0,\alpha}}$, and $\|\rho_t^\mu\|_{\m C^{0,\alpha}}$, 
such that
\begin{align*}
\sigma_{\rm min} I_d \leq  I_d - \nabla^2\phi_t(x) \leq \sigma_{\rm max}I_d
\qquad\mx{and}\qquad
\sigma_{\rm min} I_d \leq  I_d - \nabla^2\phi_t(x) \leq \sigma_{\rm max}I_d
\end{align*}
hold for every $(t, x)\in[0, 1]\times \mb T^d$.
\myqed
\end{lemma}

The above statement provides the regularity of $\phi_t$ and $\psi_t$ for each fixed $t$. The following result then focuses on the time derivative $\dot\psi_t \coloneqq \partial_t\psi_t$. This kind of results is strongly related to linearized Monge--Ampere equation. We refer to~\citep{loeper2005regularity} for more discussion.

\begin{lemma}\label{lem: jtsmooth-Kanpot}
For any $k\in\mb N$ and $\alpha\in(0, 1)$, $t\mapsto \psi_t$ is a continuous map from $[0, 1]$ to $\m C^{k+2,\alpha}(\mb T^d)$. As a result, the map $t\mapsto \phi_t$ is also smooth. Moreover, $\dot\psi_t$ satisfies the following uniform elliptic PDE,
\begin{align*}
\nabla\cdot\Big[\rho_0^\mu\big(x - \nabla\psi_t(x)\big) \cof\big(I_d - \nabla^2\psi_t(x)\big) \nabla u(x)\Big] = -\partial_t\rho_t^\mu(x).
\end{align*}
Here, for a positive definite matrix $A$, we let $\cof(A) \coloneqq \det(A)A^{-1}$ be the cofactor matrix of $A$.
\myqed
\end{lemma}

The last key ingredient concerns the estimate on how the transport map $T^\mu$ is close to the identity map when the time gap $t_1$ is small.
\begin{lemma}\label{lem: uniform-in-time-OTmap}
There are constants $C = C(\alpha, d, \sup_{t\in[0, t_1]}\|\rho_t^\mu\|_{\m C^{1,\alpha}}, L_\rho)$ and $C' = C'(\alpha, d, \sup_{t\in[0, t_1]}\|\rho_t^\mu\|_{\m C^{2,\alpha}}, L_\rho)$, such that
\begin{align}
\matnorm{JT^\mu(x) - I_d}_{\rm op} &\leq Ct_1\sup_{0\leq t\leq t_1}\|\partial_t \rho_t^\mu\|_{\m C^{0,\alpha}},\label{eqn: smooth1}\\
\sum_{k=1}^d \Big\|\frac{\partial JT^\mu(x)}{\partial x_k}\Big\|_{\rm F}^2 &\leq  C't_1^2 \sup_{0\leq t\leq t_1}\|\partial_t\rho_t^\mu\|_{\m C^{1,\alpha}}^2.\label{eqn: smooth2}
\end{align}
\myqed
\end{lemma}

\vspace{0.5em}
\noindent
Now, we are ready for the proof of the estimate~\eqref{eqn: same-marginal-cts-discrete-1piece} with $j=1$.

\vspace{0.5em}\noindent
\underline{Step 1.} We will prove that there exists a constant $C_1 > 0$ depending on $\alpha$, $d$, $L_\rho$, $\sup_{t\in[0, t_1]}\|\partial_t\rho_t^\mu\|_{\m C^{0,\alpha}}$, $\sup_{t\in[0,t_1]}\|\Phi_t^\mu\|_{\m C^1}$, and $\sup_{t\in[0, t_1]}\|\rho_t^\mu\|_{\m C^{1,\alpha}}$, such that
\begin{align*}
\int_{\mb T^d}\big\|\nabla\bar \Phi_{0}^\mu - \nabla\bar\Phi_t^\mu\big\|^2\,\dd\bar\rho_t^\mu \leq C_1t_1^2,
\qquad t\in[0, t_1].
\end{align*}
First, we need to simplify the above notation using the change of variable formula. It gives
\begin{align}
\int_{\mb T^d}\big\|\nabla\bar \Phi_{0}^\mu - \nabla\bar\Phi_t^\mu\big\|^2\,\dd\bar\rho_t^\mu
&= \int_{\mb T^d}\big\|\nabla\bar\Phi_0^\mu(T_t(x)) - \nabla\bar\Phi_t^\mu(T_t(x))\big\|^2\,\dd\rho_0^\mu(x)\notag \\
&= \int_{\mb T^d}\bigg\|\frac{T^\mu(T_t(x)) - T_t(x)}{t_1} - \frac{T^\mu(x) - x}{t_1}\bigg\|^2\,\dd\rho_0^\mu(x).\label{eqn: velocity2transport1}
\end{align}
Here, the first line follows from $\bar\rho_t^\mu = (T_t)_\#\rho_0^\mu$, and the second line is due to~\eqref{eqn: geodesics_vecfield}. 
By mean-value theorem, there exists $\xi_{t,x}\in\mb T^d$ on the segment connecting $x$ and $T_t(x)$, such that
\begin{align*}
T^\mu\big(T_t(x)\big) - T^\mu(x)
&= JT^\mu(\xi_{t,x})\big(T_t(x) - x\big),
\end{align*}
where $JT^\mu$ is the Jacobian matrix of the transport map $T^\mu$.
Combining with~\eqref{eqn: velocity2transport1}, we get
\begin{align*}
&\int_{\mb T^d}\big\|\nabla\bar \Phi_{0}^\mu - \nabla\bar\Phi_t^\mu\big\|^2\,\dd\bar\rho_t^\mu
= \frac{1}{t_1^2}\int_{\mb T^d}\Big\|\big(JT^\mu(\xi_{t,x}) - I_d\big)\big(T^\mu(x) - x\big)\Big\|^2\,\dd\rho_0^\mu(x)\\
&\leq \frac{\sup_x\matnorm{JT^\mu(x) - I_d}_{\rm op}^2}{t_1^2}\int_{\mb T^d}\big\|T^\mu(x) - x\big\|^2\,\dd\rho_0^\mu(x)
= \frac{\sup_x\matnorm{JT^\mu(x) - I_d}_{\rm op}^2}{t_1^2}\W_2^2\big(\rho_0^\mu, \rho_{t_1}^\mu\big)\\
&\leq C^2\sup_{0\leq t\leq t_1}\|\partial_t\rho_t^\mu\|_{\m C^{0,\alpha}}^2 \W_2^2(\rho_0^\mu, \rho_{t_1}^\mu).
\end{align*}
Here, the last inequality follows from Lemma~\ref{lem: uniform-in-time-OTmap}. Moreover, Benamous--Brenier formula implies
\begin{align*}
\W_2^2(\rho_0^\mu, \rho_{t_1}^\mu)
\leq \frac{t_1}{2}\int_0^{t_1}\!\!\!\int_{\mb T^d}\|\nabla\Phi_t^\mu\|^2\,\dd\rho_t^\mu\dd t
\leq \frac{t_1^2}{2}\sup_{0\leq t\leq t_1}\|\Phi_t^\mu\|_{\m C^1}^2.
\end{align*}
Therefore, we have
\begin{align*}
\int_{\mb T^d}\big\|\nabla\bar \Phi_{0}^\mu - \nabla\bar\Phi_t^\mu\big\|^2\,\dd\bar\rho_t^\mu
\leq \frac{C^2t_1^2}{2} \sup_{0\leq t\leq t_1}\|\partial_t\rho_t^\mu\|_{\m C^{0,\alpha}}^2\sup_{0\leq t\leq t_1}\|\Phi_t\|_{\m C^1}^2.
\end{align*}

\vspace{0.5em}
\noindent\underline{Step 2: control the relative Fisher information.}
Since $\bar\rho_t^\mu = (T_t)_\#\rho_0^\mu$, the change of measure formula implies 
\begin{align}\label{eqn: relative-FI}
\int_{\mb T^d}\Big\|\nabla\log\frac{\rho_{0}^\mu}{\bar\rho_t^\mu}(x)\Big\|^2\,\dd\bar\rho_t^\mu(x)
= \int_{\mb T^d}\Big\|\nabla\log\frac{\rho_0^\mu}{\bar\rho_t^\mu}\big(T_t(x)\big)\Big\|^2\,\dd\rho_0^\mu(x).
\end{align}
for $t\in[0, t_1]$.  
Also, the Monge--Ampere equation implies
\begin{align*}
\log \rho_0^\mu(x) &= \log\bar\rho_t^\mu\big(T_t(x)\big) + \log\det JT_t(x),\quad x\in\mb T^d.
\end{align*}
Taking derivatives on both sides with respect to $x$, we get
\begin{align*}
\nabla\log\rho_0^\mu(x) 
&= \nabla(\log\bar\rho_t^\mu\circ T_t)(x) + \nabla\log\det JT_t(x)\\
&= JT_t(x) \nabla\log\bar\rho_t^\mu\big(T_t(x)\big) + \nabla\log\det JT_t(x).
\end{align*}
Here, the last line follows from the chain rule. This implies
\begin{align*}
\nabla\log\bar\rho_t^\mu\big(T_t(x)\big) = \big[JT_t(x)\big]^{-1}\big[\nabla\log\rho_0^\mu(x) - \nabla\log\det JT_t(x)\big].
\end{align*}
Applying the above equality to~\eqref{eqn: relative-FI} yields
\begin{align}\label{eqn: upperbd_RFI}
\begin{aligned}
&\quad\,\int_{\mb T^d}\Big\|\nabla\log\frac{\rho_0^\mu}{\bar\rho_t^\mu}(x)\Big\|^2\,\dd\bar\rho_t^\mu
= \int_{\mb T^d}\Big\| \nabla\log\rho_0^\mu\big(T_t(x)\big) - \big[JT_t(x)\big]^{-1}\big[\nabla\log\rho_0^\mu(x) - \nabla\log\det JT_t(x)\big]\Big\|^2\,\dd\rho_0^\mu\\
&\leq 3 \int_{\mb T^d}\big\|\nabla\log\rho_0^\mu\big(T_t(x)\big) - \nabla\log\rho_0^\mu(x)\big\|^2 + \big\|\nabla\log\rho_0^\mu(x) - \big[JT_t(x)\big]^{-1}\nabla\log\rho_0^\mu(x)\big\|^2\\
&\qquad\qquad\qquad\qquad\qquad\qquad\qquad\qquad\qquad\qquad\qquad\qquad + \big\|\big[JT_t(x)\big]^{-1}\nabla\log\det JT_t(x)\big\|^2\,\dd\rho_0^\mu,
\end{aligned}
\end{align}
where the last line follows from the Cauchy--Schwarz inequality. Now, we will control these three terms separately.

\underline{Step 2.1:} For the first term in~\eqref{eqn: upperbd_RFI}, we have
\begin{align*}
&\quad\,\int_{\mb T^d}\big\|\nabla\log\rho_0^\mu\big(T_t(x)\big) - \nabla\log\rho_0^\mu(x)\big\|^2\,\dd\rho_0^\mu
\leq \sup_{x}\matnorm{\nabla^2\log\rho_0^\mu}_{\rm op}\int_{\mb T^d}\|T_t(x) - x\|^2\,\dd\rho_0^\mu(x)\\
&= \frac{t^2}{t_1^2} \sup_{x}\matnorm{\frac{\nabla^2\rho_0^\mu}{\rho_0^\mu} - \Big(\frac{\nabla\rho_0^\mu}{\rho_0^\mu}\Big)\Big(\frac{\nabla\rho_0^\mu}{\rho_0^\mu}\Big)^\top}_{\rm op} \W_2^2(\rho_0^\mu, \rho_{t_1}^\mu)
\leq \frac{t^2}{t_1^2}\Big(L_\rho^{-1}\|\rho_0^\mu\|_{\m C^2} + L_\rho^{-2}\|\rho_0^\mu\|_{\m C^1}^2\Big)\W_2^2\big(\rho_0^\mu, \rho_{t_1}^\mu\big)\\
&\leq \frac{t^2}{2}\sup_{0\leq t\leq t_1}\|\Phi_t^\mu\|_{\m C^1}^2\Big(L_\rho^{-1}\|\rho_0^\mu\|_{\m C^2} + L_\rho^{-2}\|\rho_0^\mu\|_{\m C^1}^2\Big).
\end{align*}

\underline{Step 2.2:} For the second term in~\eqref{eqn: upperbd_RFI}, we have
\begin{align*}
\int_{\mb T^d}\big\|\nabla\log\rho_0^\mu(x) - \big[JT_t(x)\big]^{-1}\nabla\log\rho_0^\mu(x)\big\|^2\,\dd\rho_0^\mu
&\leq \sup_x \matnorm{I_d - [JT_t(x)]^{-1}}_{\rm op}^2\int_{\mb T^d}\big\|\nabla\log\rho_0^\mu(x)\big\|^2\,\dd\rho_0^\mu.
\end{align*}
Due to~\eqref{eqn: ot_interpolate}, we have
\begin{align*}
JT_t(x) = \frac{t_1 -t}{t_1} I_d + \frac{t}{t_1}JT^\mu(x).
\end{align*}
Also note that
\begin{align*}
\matnorm{I_d - [JT_t(x)]^{-1}}_{\rm op}
= \matnorm{JT_t(x)^{-1}\big[JT_t(x) - I_d\big]}_{\rm op}
\leq \matnorm{JT_t(x)^{-1}}_{\rm op}\matnorm{JT_t(x) - I_d}_{\rm op}.
\end{align*}
Applying Lemma~\ref{lem: Kanpot-HessLB}, the first term is bounded by
\begin{align}\label{eqn: invHess-opnorm-ub}
\matnorm{JT_t(x)^{-1}}_{\rm op} 
\leq \frac{1}{1 + \frac{t}{t_1}(\sigma_{\rm min}-1)},
\end{align}
and we also have
\begin{align*}
\matnorm{JT_t(x) - I_d}_{\rm op} 
= \frac{t}{t_1}\matnorm{JT^\mu(x) - I_d}_{\rm op}
\leq C t\sup_{0\leq t\leq t_1}\|\partial_t\rho_t^\mu\|_{\m C^{0,\alpha}}.
\end{align*}
Thus, we get
\begin{align*}
\int_{\mb T^d}\big\|\nabla\log\rho_0^\mu(x) - \big[JT_t(x)\big]^{-1}\nabla\log\rho_0^\mu(x)\big\|^2\,\dd\rho_0^\mu\leq \bigg(\frac{Ct\sup_{0\leq t\leq t_1}\|\partial_t\rho_t\|_{\m C^{0,\alpha}}}{1 + \frac{t}{t_1}(\sigma_{\rm min}-1)}\bigg)^2\int_{\mb T^d}\big\|\nabla\log\rho_0^\mu(x)\big\|^2\,\dd\rho_0^\mu.
\end{align*}

\underline{Step 2.3:} For the third term in~\eqref{eqn: upperbd_RFI}, we have
\begin{align*}
&\quad\,\int_{\mb T^d}\big\|[JT_t(x)]^{-1} \nabla\log\det JT_t(x)\big\|^2\,\dd\rho_0^\mu
\leq \sup_{x}\matnorm{[JT_t(x)]^{-1}}_{\rm op}^2\int_{\mb T^d}\big\| \nabla\log\det JT_t(x)\big\|^2\,\dd\rho_0^\mu\\
&=\sup_{x}\matnorm{[JT_t(x)]^{-1}}_{\rm op}^2 \sum_{k=1}^d\int_{\mb T^d}\Big[\tr\Big(JT_t(x)^{-1}\frac{\partial JT_t(x)}{\partial x_k}\Big)\Big]^2\,\dd\rho_0^\mu\\
&= \frac{t^2}{t_1^2}\sup_{x}\matnorm{[JT_t(x)]^{-1}}_{\rm op}^2\sum_{k=1}^d\int_{\mb T^d}\Big[\tr\Big(JT_t(x)^{-1}\frac{\partial JT^\mu(x)}{\partial x_k}\Big)\Big]^2\,\dd\rho_0^\mu\\
&\leq \frac{t^2}{t_1^2}\sup_{x}\matnorm{[JT_t(x)]^{-1}}_{\rm op}^2\sup_x\big\|[JT_t(x)]^{-1}\big\|_{\rm F}^2 \sum_{k=1}^d \int_{\mb T^d}\Big\|\frac{\partial JT^\mu(x)}{\partial x_k}\Big\|_{\rm F}^2\,\dd\rho_0^\mu.
\end{align*}
Since $JT_t(x)$ is positive symmetric, applying~\eqref{eqn: invHess-opnorm-ub} yields
\begin{align*}
\big\|JT_t(x)^{-1}\big\|_{\rm F}^2
\leq d\matnorm{JT_t(x)^{-1}}_{\rm op}^2
\leq \frac{d}{[1 + \frac{t}{t_1}(\sigma_{\rm min}-1)]^2}
\end{align*}
Combining with Lemma~\ref{lem: uniform-in-time-OTmap}, we get
\begin{align*}
\int_{\mb T^d}\big\|[JT_t(x)]^{-1} \nabla\log\det JT_t(x)\big\|^2\,\dd\rho_0^\mu
\leq C't^2\frac{d\sup_{0\leq t\leq t_1}\|\partial_t\rho_t\|_{\m C^{1,\alpha}}^2}{[1+\frac{t}{t_1}(\sigma_{\rm min}-1)]^4}.
\end{align*}

Combining all the pieces above yields the result.
\qed

\subsubsection{Proof of Lemma~\ref{lem: jtsmooth-Kanpot}}

\begin{proof}
For simplicity, define the function space
\begin{align*}
\m C_0^{k,\alpha}(\mb T^d) \coloneqq \bigg\{f\in\m C^{k,\alpha}(\mb T^d): \int_{\mb T^d} f\,\dd x = 0\bigg\}.
\end{align*}
Consider the functional $\m A: [0, 1]\times \m C_0^{k+2,\alpha}(\mb T^d) \to \m C_0^{k,\alpha}(\mb T^d)$, defined by
\begin{align*}
\m A(t, \psi) = \rho_0^\mu \big(x - \nabla\psi(x)\big)\det\big[I_d - \nabla^2\psi(x)\big] - \rho_t^\mu(x).
\end{align*}
Note that $\psi_t$ is the solution to the Monge--Ampere equation $\m A(t, \psi_t) = 0$. Now, we will use implicit function theorem to prove the regularity of the map $t\mapsto \psi_t$. 

First, we claim that $\m A$ is jointly $\m C^\infty$. Since $\m A$ is additively separable in $t$ and $\psi$, we only need to show that the maps $t\mapsto \rho_t^\mu$ and $\psi\mapsto \rho_0^\mu(\id - \nabla\psi)\det[I_d - \nabla^2\psi]$ are $\m C^\infty$. The former is given by Assumption~\ref{assump: density_expansion}. 
For the latter, note that $\psi\mapsto \rho_0^\mu(\id -\nabla\psi)$ is a smooth map from $\m C_0^{k+2,\alpha}(\mb T^d)$ to $\m C^{k+1,\alpha}(\mb T^d)$, and the map $\psi\mapsto \det[I_d - \nabla^2\psi]$ is also smooth from $\m C_0^{k+2,\alpha}(\mb T^d)$ to $\m C^{k,\alpha}(\mb T^d)$.

Next, given $t\in[0, 1]$ we will show that the Frechet derivative $\nabla_\psi\m A(t, \psi_t): \m C_0^{k+2,\alpha}(\mb T^d)\to\m C_0^{k,\alpha}(\mb T^d)$ is an invertible operator. Note that
\begin{align*}
\nabla_\psi\m A(t, \psi_t)[u] 
&= -\big\langle \nabla\rho_0^\mu(\id - \nabla\psi_t), \nabla u\big\rangle \det\big[I_d - \nabla^2\psi_t\big] \\
&\qquad\qquad- \rho_0^\mu(\id - \nabla\psi_t) \det\big(I_d - \nabla^2\psi_t\big)\tr\Big[(I_d - \nabla^2\psi_t)^{-1}\nabla^2u\Big]\\
&= - \nabla\cdot\Big[\rho_0^\mu(\id - \nabla\psi_t)\det(I_d - \nabla^2\psi_t)\big(I_d - \nabla^2\psi_t\big)^{-1}\nabla u\Big]
\end{align*}
By Lemma~\ref{lem: Kanpot-HessLB}, the matrix $(I_d - \nabla^2\psi_t)^{-1}$ is smooth and uniformly positive definite. Thus, for any function $f\in\m C_0^{k, \alpha}(\mb T^d)$, the following PDE with respect to $u$,
\begin{align*}
\nabla_\psi\m A(t, \psi_t)[u] = f
\end{align*}
has a unique solution in $\m C_0^{k+2,\alpha}(\mb T^d)$. 

Now, using implicit function theorem, for any $t_0\in(0, 1)$, there exists an open interval containing $t_0$ where $t\mapsto \psi_t$ is a $\m C^\infty$ map. For the cases $t_0 = 0$ and $t_0 = 1$, since $\rho_0^\mu(x) > L_\rho$ and $\rho_1^\mu(x) > L_\rho$ for all $x\in\mb T^d$, we can extend the density flow $(\rho_t^\mu: 0\leq t\leq 1)$ to $(\rho_t^\mu: -\varepsilon\leq t\leq 1+\varepsilon)$ for an $\varepsilon > 0$ small enough. Then, we can use the implicit function theorem again at $t_0 = 0$ and $t_0 = 1$. We have concluded that the map $t\mapsto \psi_t$ is smooth on $t\in[0, 1]$.
\end{proof}

\subsubsection{Proof of Lemma~\ref{lem: uniform-in-time-OTmap}}
Recall that $I_d - \nabla^2\phi_{t_1}(x) = JT^\mu(x)$ and $\nabla^2\phi_0(x) = 0$. Therefore, we have
\begin{align*}
\matnorm{JT^\mu(x) - I_d}_{\rm op}
= \matnorm{\nabla^2\phi_{t_1}(x) - \nabla^2\phi_0(x)}_{\rm op}
\leq t_1 \sup_{0\leq t\leq t_1}\matnorm{\nabla^2\dot\phi_t(x)}_{\rm op}
< t_1\sup_{0\leq t\leq t_1}\big\|\nabla\dot\phi_t\big\|_{\m C^{1,\alpha}}
\end{align*}
and also
\begin{align*}
\Big\|\frac{\partial JT^\mu(x)}{\partial x_k}\Big\|_{\rm F}
= \Big\|\frac{\partial \nabla^2\phi_{t_1}(x)}{\partial x_k} - \frac{\partial \nabla^2\phi_0(x)}{\partial x_k}\Big\|_{\rm F}
\leq t_1 \sup_{0\leq t\leq t_1}\Big\|\frac{\partial\nabla^2\dot\phi_t(x)}{\partial x_k}\Big\|_{\rm F}
\leq t_1 d\big\|\nabla\dot\phi_t\big\|_{\m C^{2,\alpha}}.
\end{align*}
Now, the rest of the proof aims to control $\nabla\dot\phi_t$ in Holder norm.

Note that $(\id - \nabla\psi_t)\circ (\id - \nabla\phi_t) = \id$, i.e.
\begin{align*}
\big(x - \nabla\phi_t(x)\big) - \nabla\psi_t\big(x - \nabla\phi_t(x)\big) = x
\end{align*}
holds for all $t\in[0, t_1]$ and $x\in\mb T^d$. Taking the time derivative on both sides yields
\begin{align*}
\nabla\dot\phi_t(x) + \nabla\dot\psi_t\big(x - \nabla\phi_t(x)\big) - \nabla^2\psi_t\big(x - \nabla\phi_t(x)\big)\nabla\dot\phi_t(x) = 0
\end{align*}
Therefore, we get
\begin{align*}
\nabla\dot\phi_t = -\Big[I_d - \nabla^2\psi_t\big(x - \nabla\phi_t(x)\big)\Big]^{-1} \nabla\dot\psi_t\big(x - \nabla\phi_t(x)\big).
\end{align*}
Now, for any $k\in\mb Z_+$ and $\alpha\in(0, 1)$, we have
\begin{align*}
\big\|\nabla\dot\phi_t\big\|_{\m C^{k,\alpha}}
&= \Big\|\big[I_d - \nabla^2\psi_t\big(x - \nabla\phi_t(x)\big)\big]^{-1} \nabla\dot\psi_t\big(x - \nabla\phi_t(x)\big)\Big\|_{\m C^{k,\alpha}}\\
&\leq C_1(k,\alpha, d)\Big\|\big[I_d - \nabla^2\psi_t\big(x - \nabla\phi_t(x)\big)\big]^{-1} \Big\|_{\m C^{k,\alpha}}\Big\|\nabla\dot\psi_t\big(x - \nabla\phi_t(x)\big)\Big\|_{\m C^{k,\alpha}}\\
&\leq C_1(k,\alpha, d)C_3\Big(k,\alpha, d, \sigma_{\rm min},\big\|I_d - \nabla^2\psi_t\big(x - \nabla\phi_t(x)\big)\big\|_{\m C^{k,\alpha}}\Big)
C_2'\big(k,\alpha,d,\|\id - \nabla\phi_t\|_{\m C^{k,\alpha}}\big)\|\nabla\dot\psi_t\|_{\m C^{k,\alpha}}
\end{align*}
Here, the last two inequalities follow from Proposition~\ref{prop: Holder}. Again, we know
\begin{align*}
\big\|\nabla^2\psi_t\big(\id - \nabla\phi_t\big)\big\|_{\m C^{k,\alpha}}
\leq C_2'\big(k,\alpha, d, \|\id - \nabla\phi_t\|_{\m C^{k,\alpha}}\big)\big\|\nabla^2\psi_t\big\|_{\m C^{k,\alpha}}
\end{align*}
Therefore, there is a function $C_{k,\alpha,d}\big(\sigma_{\rm min}, \|\phi_t\|_{\m C^{k+1,\alpha}}, \|\psi_t\|_{\m C^{k+2,\alpha}}\big)$, such that
\begin{align*}
\big\|\nabla\dot\phi_t\big\|_{\m C^{k,\alpha}}
\leq C_{k,\alpha,d}\big(\sigma_{\rm min}, \|\phi_t\|_{\m C^{k+1,\alpha}}, \|\psi_t\|_{\m C^{k+2,\alpha}}\big) \big\|\nabla\dot\psi_t\big\|_{\m C^{k,\alpha}}.
\end{align*}

Applying Lemma~\ref{lem: jtsmooth-Kanpot} and Lemma~\ref{lem: Schauder}, there is another constant $C'_{k,\alpha,d}(\sigma_{\rm min}, \sigma_{\rm max}, \|\psi_t\|_{\m C^{k+2}}, \|\rho_0\|_{\m C^{k,\alpha}}) > 0$, such that
\begin{align*}
\big\|\nabla\dot\psi_t\big\|_{\m C^{k,\alpha}} \leq C'_{k,\alpha, d}\big(\sigma_{\rm min}, \sigma_{\rm max}, \|\psi_t\|_{\m C^{k+2,\alpha}}, \|\rho_0\|_{\m C^{k,\alpha}}\big)\|\partial_t\rho_t^\mu\|_{\m C^{k-1,\alpha}}.
\end{align*}
Therefore, we have
\begin{align*}
\big\|\nabla\dot\phi_t\big\|_{\m C^{k,\alpha}}
\leq C_{k,\alpha,d}\big(\sigma_{\rm min}, \|\phi_t\|_{\m C^{k+1,\alpha}}, \|\psi_t\|_{\m C^{k+2,\alpha}}\big) C'_{k,\alpha,d}\big(\sigma_{\rm min}, \sigma_{\rm max}, \|\psi_t\|_{\m C^{k+2,\alpha}}, \|\rho_0\|_{\m C^{k,\alpha}}\big)\|\partial_t\rho_t^\mu\|_{\m C^{k-1,\alpha}}.
\end{align*}
Due to the regularity of the solution to the Monge--Ampere equation, $\|\phi_t\|_{\m C^{k+2,\alpha}}$ and $\|\psi_t\|_{\m C^{k+2,\alpha}}$ can be upper bounded by constants depending only on $k,\alpha, d, \|\rho_0^\mu\|_{\m C^{k,\alpha}}, \|\rho_t^\mu\|_{\m C^{k,\alpha}}$, and $L_\rho$. Taking $k=1$ and $k=2$ leads to the upper bounds~\eqref{eqn: smooth1} and~\eqref{eqn: smooth2} respectively.
\qed

\subsection{Proof of Corollary~\ref{coro: Brenier_map}}\label{app: pf Brenier_map}

Since $\phi_\varepsilon^{\rm self}\in\m C^\infty(\mb T^d)$ for any given $\varepsilon > 0$, there exists a constant $\beta\geq 0$ (possibly depending on $\varepsilon$), such that $\phi_\varepsilon^{\rm self}$ is $\beta$ semi-convex, i.e. $-\phi_\varepsilon^{\rm self}$ is $\beta$ semi-concave. Let $\alpha$ and $R$ be the constants in Proposition~\ref{prop: 2norm_local_convex}, so that $\td(0_d,\cdot)^2$ is $\alpha$-strongly convex for all $x\in B_{\mb T^d}(0_d, R)$.
Now, for any $x_0\in\mb T^d$, define
\begin{align*}
\varphi^{x_0}_\varepsilon\coloneqq \phi_\varepsilon^{\rm self} + \frac{\beta}{\alpha}\td(x_0,\cdot)^2
\end{align*}
Then, $\varphi_\varepsilon^{x_0}$ is a convex function on $B_{\mb T^d}(x_0, R)$.

Now, define $\{(u_k, u_k): k=0,1,\cdots, S\} = \schro(\rho, 0, \cdots, 0)$, where we let $S = 3K$ for simplicity. Take
\begin{align*}
f_1 &= \log u_0,
\quad\mx{and}\quad 
f_k =  \sum_{l=1}^{k-1} \frac{(-1)^{l+1}}{lu_0^l} \sum_{\substack{s_1 + \cdots + s_l = k-1\\1\leq s_1,\cdots,s_l<k}} u_{s_1}\cdots u_{s_l}
\quad\forall k>1.
\end{align*}
For simplicity, let $F_{S-1}\coloneqq \sum_{k=1}^{S-1} \varepsilon^k f_k$.
Then, Theorem~\ref{thm: asymp_schro_potential} implies
\begin{align*}
2\sqrt{\mb V_\rho\big(\phi_\varepsilon - F_{S-1}\big)}
&= \min_{c\in\mb R}\Big\{\big\|\phi_\varepsilon + c - F_{S-1}\big\|_{L^2(\rho)} + \big\|\phi_\varepsilon - c - F_{S-1}\big\|_{L^2(\rho)}\Big\} \lesssim O(\varepsilon^S).
\end{align*}
Note that there also exists $\varepsilon_0(S, \rho) > 0$, such that $U_{S-1}$ is also $\beta$ semi-convex whenever $\varepsilon \leq \varepsilon_0(S, \rho)$.

By Proposition~\ref{prop: var_decompose}, there exists an open covering $\m F$ consisting of only balls with radius no greater than $R$. Moreover, $\m F$ satisfies the Boman chain condition with constants $A, B, C > 1$. Then, we have
\begin{align*}
&\int_{\mb T^d}\big\|\nabla\phi_\varepsilon^{\rm self} - \nabla U_{S-1}\big\|^2\,\dd\rho
\leq \sum_{Q_{x_0,r}\in\m F}\int_{Q_{x_0,r}}\Big\|\nabla\varphi_\varepsilon^{x_0} - \nabla\Big[\frac{\beta}{\alpha}\td(x_0, \cdot)^2 + U_{S-1}\Big]\Big\|^2\,\dd\rho\\
&= \sum_{Q_{x_0,r}\in\m F}\rho(Q_{x_0,r}) \int_{Q_{x_0,r}}\Big\|\nabla\varphi_\varepsilon^{x_0} - \nabla\Big[\frac{\beta}{\alpha}\td(x_0, \cdot)^2 + U_{S-1}\Big]\Big\|^2\,\dd\rho_{Q_{x_0,r}}\\
&\stackrel{\ri}{\lesssim} \sum_{Q_{x_0,r}\in\m F}\rho(Q_{x_0,r}) \Big[\mb V_{\rho_{Q_{x_0,r}}}\Big(\varphi_\varepsilon^{x_0} - \frac{\beta}{\alpha}\td(x_0, \cdot)^2 - U_{S-1}\Big)\Big]^{\frac{1}{3}}\\
&\stackrel{\rii}{\leq} \Big[\sum_{Q_{x_0,r}\in\m F}\rho(Q_{x_0,r})\Big]^{\frac{2}{3}}\Big[\sum_{Q_{x_0,r}\in\m F}\rho(Q_{x_0,r})\mb V_{\rho_{Q_{x_0,r}}}\Big(\varphi_\varepsilon^{x_0} - \frac{\beta}{\alpha}\td(x_0, \cdot)^2 - U_{S-1}\Big)\Big]^{\frac{1}{3}}.
\end{align*}
Here, (i) follows from Eqn. (4.4) and Proposition 4.3 in~\cite{kitagawa2025stability}, where a constant depending on $L_\rho$ and $d$ is omitted, and (ii) follows from Holder's inequality. 

Note that the second term in the above inequality is
\begin{align*}
&\sum_{Q\in\m F}\rho(Q)\cdot\inf_{c_Q\in\mb R}\int_Q \big(\phi_\varepsilon^{\rm self} - U_{S-1} - c_Q\big)^2\,\dd\rho_Q
= \sum_{Q\in\m F}\inf_{c_Q\in\mb R}\int_{\mb T^d} \big(\phi_\varepsilon^{\rm self} - U_{S-1} - c_Q\big)^2 1_Q\,\dd\rho\\
&\leq \inf_{c\in\mb R}\int_{\mb T^d} \big(\phi_\varepsilon^{\rm self} - U_{S-1} - c\big)^2 \sum_{Q\in\m F}1_Q\,\dd\rho
\stackrel{\ri}{\leq} A\inf_{c\in\mb R} \int_{\mb T^d}\big(\phi_\varepsilon^{\rm self} - U_{S-1} - c_Q\big)^2\,\dd\rho\\
&= A \mb V_\rho\big(\phi_\varepsilon^{\rm self} - U_{S-1}\big).
\end{align*}
Here, (i) follows from the first point of Boman chain condition:
\begin{align*}
\sum_{Q\in\m F} \rho(Q) = \int_{\mb T^d}\sum_{Q\in\m F} 1_Q\,\dd\rho \leq A,
\end{align*}
which only depends on $d$ and $L_\rho$ according to Proposition~\ref{prop: var_decompose}.
Therefore, we get 
\begin{align*}
\int_{\mb T^d}\big\|\nabla\phi_\varepsilon^{\rm self} - \nabla U_{S-1}\big\|^2\,\dd\rho
\lesssim \big[\mb V_\rho\big(\phi_\varepsilon^{\rm self} - U_{S-1}\big)\big]^{\frac{1}{3}} = O\big(\varepsilon^{\frac{2S}{3}}\big) = O(\varepsilon^{2K}).
\end{align*}
This implies
\begin{align*}
\int_{\mb T^d}\Big\|\nabla\phi_\varepsilon^{\rm self} - \sum_{k=1}^{K-1}\varepsilon^k \nabla f_k\Big\|^2\,\dd\rho \leq O(\varepsilon^{2K}).
\end{align*}
The proof is completed.

\section{Technical Lemmas}
\subsection{Proof of Lemma~\ref{lem: int_remainder}}\label{sec: pf_int_remainder}
By the definitions of $U_K$ and $V_K$ in~\eqref{eqn: proxy_Schro_sys}, we have
\begin{align*}
\frac{R_\varepsilon(y)}{\rho_\varepsilon(y)} 
&= 1 - \Big[\sum_{k=0}^K \varepsilon^k v_k(y)\Big] \int_{\mb T^d}\m K_\varepsilon(y-x)\Big[\sum_{k=0}^K \varepsilon^k u_k(x)\Big]\rho_0(x)\,\dd x\\
&= 1 - \sum_{k=0}^K \varepsilon^k\sum_{l=0}^kv_l(y)\m K_\varepsilon\ast(u_{k-l}\rho_0)(y) - \sum_{k=K+1}^{2K}\varepsilon^k \sum_{l=k-K}^Kv_l(y)\m K_\varepsilon\ast(u_{k-l}\rho_0)(y),
\end{align*}
where $\m K_\varepsilon\ast(u_{k-l}\rho_0)$ is the convolution of the Gaussian kernel $\m K_\varepsilon$ and $u_{k-l}\rho_0$.
Using the initial condition~\eqref{eqn: init}, we have $\lim_{\varepsilon\to 0^+}R_\varepsilon = \rho_0 - \rho_0 v_0u_0\rho_0 = 0$, where the convergence is in the sense of $L^2(\mb T^d)$ and pointwise.
For every $p\in\mb N^\ast$ , the $p$-th order derivative of $\frac{R_\varepsilon}{\rho_\varepsilon}$ with respect to $\varepsilon$ is 
\begin{align}\label{eqn: p-deri}
&\,\,\,\,\partial_\varepsilon^p \Big(\frac{R_\varepsilon}{\rho_\varepsilon}\Big)
= - \sum_{s=0}^p\binom{p}{s}\sum_{k=0}^K\big(\partial_\varepsilon^s\varepsilon^k\big)\sum_{l=0}^{k}v_l\partial_\varepsilon^{p-s}\m K_\varepsilon\ast(u_{k-l}\rho_0) - \sum_{s=0}^p\binom{p}{s}\sum_{k=K+1}^{2K}\big(\partial_\varepsilon^s\varepsilon^k\big)\sum_{l=k-K}^{K}v_l\partial_\varepsilon^{p-s}\m K_\varepsilon\ast(u_{k-l}\rho_0)\\ \nonumber
&\stackrel{\ri}{=} - \sum_{s=0}^p\binom{p}{s}\sum_{k=0}^K\big(\partial_\varepsilon^s\varepsilon^k\big)\sum_{l=0}^{k}\frac{v_l}{2^{p-s}}\m K_\varepsilon\ast\big[\Delta^{p-s}(u_{k-l}\rho_0)\big] - \sum_{s=0}^p\binom{p}{s}\sum_{k=K+1}^{2K}\big(\partial_\varepsilon^s\varepsilon^k\big)\sum_{l=k-K}^{K}\frac{v_l}{2^{p-s}}\m K_\varepsilon\ast\big[\Delta^{p-s}(u_{k-l}\rho_0)\big].
\end{align}
Here, in (i), we use the identity $\partial_\varepsilon[\m K_\varepsilon\ast (u_{k-l}\rho_0)] = \frac{1}{2}(\Delta\m K_\varepsilon)\ast (u_{k-l}\rho_0) = \frac{1}{2}\m K_\varepsilon\ast[\Delta (u_{k-l}\rho_0)]$ when $u_{k-l}\rho_0\in H^2(\mb T^d)$.
When $1\leq p\leq K$, taking the limit as $\varepsilon\to 0^+$ yields
\begin{align*}
\lim_{\varepsilon\to 0^+}\partial_\varepsilon^p \Big(\frac{R_\varepsilon}{\rho_\varepsilon}\Big) 
&\stackrel{\ri}{=} - \sum_{s=0}^p\binom{p}{s}s!\cdot \sum_{l=0}^{s}\frac{v_l \Delta^{p-s}(u_{s-l}\rho_0)}{2^{p-s}}
= - p!\sum_{s=0}^p\sum_{l=0}^{s}\frac{v_l \Delta^{p-s}(u_{s-l}\rho_0)}{2^{p-s}(p-s)!}.
\end{align*}
Here, the convergence is in $L^2(\mb T^d)$ sense in (i), where we use the facts that if $\Delta^{p-s}(u_{k-l}\rho_0)\in L^2(\mb T^d)$, then $\m K_\varepsilon\ast [\Delta^{p-s}(u_{k-l}\rho_0)] \stackrel{L^2(\mb T^d)}{\longrightarrow} \Delta^{p-t-s}(u_{k-l}\rho_0)$ and that $u_0\rho_0,\cdots,u_{2K}\rho_0\in H^{2p}(\mb T^d)$. Note that
\begin{align*}
&\quad\,\sum_{s=0}^p\sum_{l=0}^{s}\frac{v_l \Delta^{p-s}(u_{s-l}\rho_0)}{2^{p-s}(p-s)!}
= \sum_{l=0}^pv_l\sum_{s=l}^p\frac{\Delta^{p-s}(u_{s-l}\rho_0)}{2^{p-s}(p-s)!} \stackrel{}{=}\sum_{l=0}^pv_l\sum_{t=0}^{p-l}\frac{\Delta^{p-l-t}(u_{t}\rho_0)}{2^{p-l-t}(p-l-t)!}\stackrel{\ri}{=} \sum_{l=0}^p v_lv_{p-l}^\dagger
\stackrel{\rii}{=} 0,
\end{align*}
where (i) is due to Equation~\eqref{eqn: iter_dagger}, and (ii) is due to Equation~\eqref{eqn: iter_fourier}. Therefore, we have $\lim_{\varepsilon\to 0^+}\partial_\varepsilon^p\big(\frac{R_\varepsilon}{\rho_\varepsilon}\big)= 0$ in $L^2(\mb T^d)$ sense for $0\leq p\leq K$. 
Since $\frac{R_\varepsilon}{\rho_\varepsilon}$ is smooth with respect to $\varepsilon\geq0$, Taylor expansion implies
\begin{align*}
\Big\|\frac{R_\varepsilon}{\rho_\varepsilon}\Big\|_{L^2(\mb T^d)} &= \bigg\|\int_0^\varepsilon \partial_\varepsilon^{K+1}\Big(\frac{R_\varepsilon}{\rho_\varepsilon}\Big)\Big|_{\varepsilon = t} \cdot \frac{(\varepsilon-t)^{K}}{K!}\,\dd t\bigg\|_{L^2(\mb T^d)}\\
&\leq \sup_{0\leq t\leq \varepsilon}\Big\|\partial_\varepsilon^{K+1}\Big(\frac{R_\varepsilon}{\rho_\varepsilon}\Big)\Big|_{\varepsilon=t}\Big\|_{L^2(\mb T^d)} \cdot \int_0^\varepsilon \frac{(\varepsilon-t)^K}{K!}\,\dd t
= \frac{\varepsilon^{K+1}}{(K+1)!}\cdot \sup_{0\leq t\leq \varepsilon}\Big\|\partial_\varepsilon^{K+1}\Big(\frac{R_\varepsilon}{\rho_\varepsilon}\Big)\Big|_{\varepsilon=t}\Big\|_{L^2(\mb T^d)}.
\end{align*}
Note that the supremum can be bounded by a constant only depending on $\|u_0\|_{H^{2K+2}}, \cdots, \|u_K\|_{H^{2K+2}}$, $\|v_0\|_{\m C^0}, \cdots, \|v_K\|_{\m C^0}$, and $\|\rho_0\|_{\m C^{2K+2}}$, e.g. because of \eqref{eqn: p-deri}.

The remaining steps in estimating $\|\frac{Q_\varepsilon}{\rho_0}\|_{L^2(\mb T^d)}$ are similar, so we only present the key results from the calculation. Again, using the definition~\eqref{eqn: proxy_Schro_sys} of the functions $U_K$ and $V_K$, we have
\begin{align*}
\frac{Q_\varepsilon(x) }{\rho_0(x)}
&= 1 - \sum_{k=0}^K\varepsilon^k\sum_{l=0}^k u_l(x)\m K_\varepsilon\ast(v_{k-l}\rho_\varepsilon)(x) - \sum_{k=K+1}^{2K}\varepsilon^k\sum_{l=k-K}^K u_l(y)\m K_\varepsilon\ast(v_{k-l}\rho_\varepsilon)(x).
\end{align*}
It then follows that $\lim_{\varepsilon\to 0^+}Q_\varepsilon = 0$, and for all $1\leq p\leq K$, we have
\begin{align*}
\lim_{\varepsilon\to 0^+}\partial_\varepsilon^p\Big(\frac{Q_\varepsilon}{\rho_0}\Big)
&= -\sum_{s=0}^p \sum_{t=0}^{p-s}\binom{p}{s, t, p-s-t}s!\sum_{l=0}^s \frac{u_l}{2^t} (p-s-t)!\Delta^{t}(v_{s-l}\rho_{p-s-t})\\
&= -p!\sum_{s=0}^p\sum_{t=0}^{p-s}\sum_{l=0}^s \frac{u_l\Delta^t(v_{s-l}\rho_{p-s-t})}{2^t t!} 
= -p!\sum_{l=0}^p u_l\Big[\sum_{s=l}^p\sum_{t=0}^{p-s}\frac{\Delta^t(v_{s-l}\rho_{p-s-t})}{2^tt!}\Big]\\
&\stackrel{\ri}{=} -p!\sum_{l=0}^p u_lu_{p-l}^\dagger 
\stackrel{\rii}{=} 0
\end{align*}
Here, (i) follows from~\eqref{eqn: iter_dagger}, and (ii) follows from~\eqref{eqn: iter_fourier}. We also use the fact that $v_i\rho_j \in H^{2(p-i-j)}(\mb T^d)$ for all integers $i, j$ satisfying $0 \leq i+j\leq p$. Therefore, by the Taylor expansions, we conclude that $\|Q_\varepsilon\|_{L^2(\mb T^d)} = O(\varepsilon^{K+1})$, where the big-O notation omits a constant depending only on $\|u_0\|_{\m C^0}, \cdots, \|u_K\|_{\m C^0}$, $\|v_0\|_{H^{2K+2}}, \cdots, \|v_K\|_{H^{2K+2}}$, $\|\rho_0\|_{\m C^0}$ and $\sup_{0\leq t\leq \varepsilon}\|\rho_t^{(k)}\|_{\m C^{2k'}}$ for all $k,k'\in\mb N$ and $k+k'\leq K+1$.
This completes the proof.


\subsection{Proof of Lemma~\ref{lem: exist_Fredholm}}\label{sec: exist_Fredholm}
\noindent \underline{Step 1: existence for $\varepsilon \geq 0$.} The existence when $\varepsilon = 0$ is obvious as $R_0=Q_0=0$, so we just consider the case $\varepsilon > 0$. Define a functional $L_\varepsilon: \m C(\mb T^d)\times \m C(\mb T^d)/_\sim \to \m C(\mb T^d)\times\m C(\mb T^d)$ by $L_\varepsilon(f, g) = (\m K_\varepsilon \ast g+ f, \m K_\varepsilon \ast f + g)$. Here, the equivalence $\sim$ is defined by $(f, g) \sim (f-c, g+c)$ for any constant $c\in\mb R$. Then, Equation~\eqref{eqn: Fredholm} is equivalent to $L_\varepsilon(r_U, r_V) = (Q_\varepsilon, R_\varepsilon)$. Note that $L_\varepsilon = \id + K_\varepsilon$, where $\id(f, g) = (f, g)$ and $K_\varepsilon(f, g) = (\m K_\varepsilon\ast g, \m K_\varepsilon\ast f)$.

To prove $L_\varepsilon$ is \emph{injective}, if there exists $(f, g)$ and $(u, v)$ such that $L_\varepsilon(f, g) = L_\varepsilon(u, v)$, i.e.
\begin{align}\label{eqn: inject}
\m K_\varepsilon \ast f + g = \m K_\varepsilon\ast u + v
\qquad\mx{and}\qquad \m K_\varepsilon \ast g + f = \m K_\varepsilon \ast v + u,
\end{align}
we have
\begin{align*}
f+g - u - v = -\m K_\varepsilon\ast(f+g - u - v),   
\end{align*}
which implies that
\begin{align*}
0\leq \|f+g - u - v\|^2_{L^2(\mb T^d)} = \langle f+g - u - v, -\m K_\varepsilon\ast(f+g - u - v)\rangle_{L^2(\mb T^d)} \leq 0.
\end{align*}
So, we have $f+g = u+v$. Combining this equation with~\eqref{eqn: inject} yields $\m K_\varepsilon\ast (f-u) = f-u$, i.e. $f-u$ is the eigenfunction of the Gaussian convolution operator, indicating that $f-u$ is a constant. So, we know $(f, g)\sim(u, v)$, indicating that $L_\varepsilon$ is an injective on $\m C(\mb T^d)\times\m C(\mb T^d)/_\sim$.

Let $\m X := \m C(\mb T^d)\times\m C(\mb T^d)/_\sim$ with the norm 
\begin{align*}
\|(f, g)\|_{\m X} \coloneqq \sup_{c\in\mb R}\|f - c\|_\infty + \|g + c\|_{\infty}
\end{align*}
be a Banach space, and treat $L_\varepsilon$ as a transformation on $\m X$. By definition, if $L_\varepsilon(f, g) \sim (0, 0)$, then there exists a constant $c\in\mb R$, such that
\begin{align*}
\m K_\varepsilon\ast g + f = -c
\quad\mx{and}\quad
\m K_\varepsilon\ast f + g = c.
\end{align*}
So, we have $(f+g) + \m K_\varepsilon\ast (f+g) = 0$, indicating that $f+g = 0$ as shown in the previous arguments. Therefore, we have $\m K_\varepsilon\ast f-f = c$, indicating that
\begin{align*}
c = \int_{\mb T^d} \m K_\varepsilon\ast f(x) - f(x)\,\dd x = \int_{\mb T^d}\!\int_{\mb T^d}\m K_\varepsilon(x-y) f(y)\,\dd x\dd y - \int_{\mb T^d}f(x)\,\dd x = 0.
\end{align*}
Thus, the above argument implies $\m K_\varepsilon \ast f = -g = f$, i.e. $f$ is a constant. So, we know $(f, g)\sim (0, 0)$. It is obvious that $K_\varepsilon$ is a compact operator on $\m C(\mb T^d)\times \m C(\mb T^d)$, and thus a compact operator on $\m X$. Then by Fredholm alternative theorem~\cite[Theorem 6.6,][]{brezis2011functional}, $L_\varepsilon = \id + K_\varepsilon$ is an invertible operator on $\m X$.

Now, we know there exists $r_U^\varepsilon, r_V^\varepsilon\in\m C(\mb T^d)$ and a constant $c\in\mb R$, such that
\begin{align*}
\m K_\varepsilon\ast r_U^\varepsilon + r_V^\varepsilon = R_\varepsilon - c
\quad\mx{and}\quad
\m K_\varepsilon\ast r_V^\varepsilon + r_U^\varepsilon = Q_\varepsilon + c.
\end{align*}
Therefore, we know
\begin{align*}
\int_{\mb T^d}R_\varepsilon(x)\,\dd x - c
&= \int\!\!\int_{\mb T^d\times\mb T^d}\m K_\varepsilon(x-y)r_U^\varepsilon(y)\,\dd y\dd x + \int_{\mb T^d} r_V^\varepsilon(x)\,\dd x\\
&=\int_{\mb T^d} r_U^\varepsilon(y)\,\dd y + \int_{\mb T^d}r_V^\varepsilon(x)\,\dd x\\
&= \int\!\!\int_{\mb T^d\times\mb T^d}\m K_\varepsilon(x-y)r_V^\varepsilon(y)\,\dd y\dd x + \int_{\mb T^d} r_U^\varepsilon(x)\,\dd x
= \int_{\mb T^d}Q_\varepsilon(x) + c.
\end{align*}
Therefore, we have
\begin{align*}
2c = \int_{\mb T^d}R_\varepsilon(x)\,\dd x - \int_{\mb T^d}Q_\varepsilon(x)\,\dd x.
\end{align*}
However, we have
\begin{align}\label{eqn: same_integral}
\begin{aligned}
\int_{\mb T^d}R_\varepsilon(x)\,\dd x
&= \int_{\mb T^d}\rho_\varepsilon(x) - \rho_\varepsilon(x) V_K(x)\cdot\m K_\varepsilon\ast (\rho_0 U_K)(x)\,\dd x\\
&= 1 - \big\langle \rho_\varepsilon V_K, \m K_\varepsilon\ast(\rho_0 U_K)\rangle_{L^2(\mb T^d)}\\
&=1 - \big\langle \m K_\varepsilon\ast(\rho_\varepsilon V_K), \rho_0 U_K\rangle_{L^2(\mb T^d)}\\
&= \int_{\mb T^d} Q_\varepsilon(x)\,\dd x.
\end{aligned}
\end{align}
So, we know $c=0$, and there exists a unique solution $(r_U^\varepsilon, r_V^\varepsilon)\in\m C(\mb T^d)\times\m C(\mb T^d)$, up to shifting a constant, to Equation~\eqref{eqn: Fredholm}.

\vspace{0.5em}
\noindent\underline{Step 2: explicit form via Fourier series.}
As $R_\varepsilon, Q_\varepsilon, r_U^\varepsilon, r_V^\varepsilon, \m K_\varepsilon \in\m C(\mb T^d)$, for any $z\in\mb Z^d$, computing the Fourier series of both sides in~\eqref{eqn: Fredholm} implies
\begin{align*}
e^{-\frac{\varepsilon\|z\|^2}{2}} \wht r_U^\varepsilon(z) + \wht r_V^\varepsilon(z) &= \wht R_\varepsilon(z)\\
e^{-\frac{\varepsilon\|z\|^2}{2}} \wht r_V^\varepsilon(z) + \wht r_U^\varepsilon(z) &= \wht Q_\varepsilon(z),
\end{align*}
i.e.
\begin{align}\label{eqn: Fourier_sol}
\wht r_U^\varepsilon(z) = \frac{\wht Q_\varepsilon(z) - e^{-\frac{\varepsilon\|z\|^2}{2}} \wht R_\varepsilon(z) }{1 - e^{-{\varepsilon\|z\|^2}{}}}
\quad\mx{and}\quad
\wht r_V^\varepsilon(z) = \frac{\wht R_\varepsilon(z) - e^{-\frac{\varepsilon\|z\|^2}{2}} \wht Q_\varepsilon(z) }{1 - e^{-{\varepsilon\|z\|^2}{}}},\quad\forall\, z\in\mb Z^d-\{0_d\}.
\end{align}
By~\eqref{eqn: same_integral}, we have
\begin{align*}
\wht R_\varepsilon(0_d) 
&= \frac{1}{(2\pi)^d}\int_{\mb T^d} R_\varepsilon(x) e^{-i 0_d^\top x}\,\dd x 
= \frac{1}{(2\pi)^d}\int_{\mb T^d} Q_\varepsilon(x) e^{-i 0_d^\top x}\,\dd x = \wht Q_\varepsilon(0_d).
\end{align*}
Then, we can take
\begin{align*}
r_U^\varepsilon(x) = \sum_{z\in\mb Z^d} \wht r_U^\varepsilon(z)e^{i z^\top x}
\quad\mx{and}\quad 
r_V^\varepsilon(x) = \sum_{z\in\mb Z^d} \wht r_V(z)e^{i z^\top x},
\end{align*}
with $\wht r_U^\varepsilon(0) = \wht r_V^\varepsilon(0) = \frac{1}{2}\wht R_\varepsilon(0) = \frac{1}{2}\wht Q_\varepsilon(0)$.

\vspace{0.5em}
\noindent\underline{Step 3: continuity and limits when $\varepsilon\to 0^+$.} Note that both $Q_\varepsilon$ and $R_\varepsilon$ are smooth with respect to $\varepsilon > 0$.
When $K = 0$, we have
\begin{align*}
R_\varepsilon = \rho_\varepsilon - \rho_\varepsilon v_0\cdot \m K_\varepsilon\ast(\rho_0 u_0)
\quad\mx{and}\quad
Q_\varepsilon = \rho_0 - \rho_0 u_0\cdot\m K_\varepsilon\ast(\rho_\varepsilon v_0).
\end{align*}
Then, we have
\begin{align}\label{eqn: RQ_epsilon0}
\lim_{\varepsilon\to 0^+} R_\varepsilon = \lim_{\varepsilon\to 0^+} Q_\varepsilon = \rho_0 - \rho_0^2 v_0 u_0 = 0,
\end{align}
and
\begin{align*}
\frac{\partial R_\varepsilon}{\partial\varepsilon}\bigg|_{\varepsilon=0} &= \rho_1 - \rho_1v_0\rho_0 u_0 - \rho_0v_0\cdot\frac{1}{2}\Delta(\rho_0 u_0) \stackrel{\ri}{=} -\frac{1}{2}u_0^\dagger\Delta v_0^\dagger,\\
\frac{\partial Q_\varepsilon}{\partial\varepsilon}\bigg|_{\varepsilon=0} &= -\rho_0 u_0\rho_1 v_0 - \rho_0 u_0\cdot\frac{1}{2}\Delta(\rho_0 v_0) \stackrel{\rii}{=} -\rho_1 - \frac{1}{2}v_0^\dagger\Delta u_0^\dagger.
\end{align*}
Here both (i) and (ii) use the definition $u_0v_0\rho_0 = u_0u_0^\dagger = v_0v_0^\dagger = 1$. Using Equation~\eqref{eqn: init_dagger} yields
\begin{align*}
\frac{\partial Q_\varepsilon}{\partial\varepsilon}\bigg|_{\varepsilon=0} - \frac{\partial R_\varepsilon}{\partial\varepsilon}\bigg|_{\varepsilon=0}
= \frac{1}{2}[u_0^\dagger\Delta v_0^\dagger - v_0^\dagger\Delta u_0^\dagger] - \rho_1 = 0.
\end{align*}
Therefore, combining Equation~\eqref{eqn: Fourier_sol} with Equation~\eqref{eqn: RQ_epsilon0} yields
\begin{align}\label{eqn: limit_Fourier}
\lim_{\varepsilon\to 0^+}\wht r_U^\varepsilon(z) = \frac{\wht R_0(z)}{2} = 0
\quad\mx{and}\quad
\lim_{\varepsilon\to 0^+} \wht r_V^\varepsilon(z) = \frac{\wht Q_0(z)}{2} = 0
\end{align}
for all $z\in\mb Z^d - \{0_d\}$. Moreover, we have $\lim_{\varepsilon\to 0^+} \wht r_U^\varepsilon(0_d) = \lim_{\varepsilon\to 0^+} \wht r_V^\varepsilon(0_d) = 0$. Therefore, we have $\lim_{\varepsilon\to 0^+}(r_U^\varepsilon, r_V^\varepsilon) = (0, 0)$. Note that this is a solution of Equation~\eqref{eqn: Fredholm} when $\varepsilon = 0$, so we know the solution $r_U^\varepsilon$ and $r_V^\varepsilon$ are continuous with respect to $\varepsilon \geq 0$.

When $K\geq 1$, we can similarly derive
\begin{align*}
\frac{\partial R_\varepsilon}{\partial\varepsilon}\bigg|_{\varepsilon = 0} &= -\rho_0^2 u_0v_1 - \rho_0^2 v_0u_1 - \frac{1}{2} u_0^\dagger\Delta v_0^\dagger,\\
\frac{\partial Q_\varepsilon}{\partial\varepsilon}\bigg|_{\varepsilon = 0} &= -\rho_0^2u_0v_1 - \rho_0^2 v_0u_1 - \frac{1}{2}v_0^\dagger\Delta u_0^\dagger - \rho_1.
\end{align*}
This implies the same limit as in Equation~\eqref{eqn: limit_Fourier}, which leads to $\lim_{\varepsilon \to 0^+}(r_U^\varepsilon, r_V^\varepsilon) = (0, 0)$. Thus, we have finished the proof of continuity of $(r_U^\varepsilon, r_V^\varepsilon)$ with respect to $\varepsilon\geq 0$.

\vspace{0.5em}
\noindent\underline{Step 4: control $L^2$ bound.}
Note that for every $z\in\mb Z^d-\{0_d\}$, we have
\begin{align}\label{eqn: fouriersq_ub}
\begin{aligned}
&|\wht r_U^\varepsilon(z)|^2 + |\wht r_V^\varepsilon(z)|^2
\stackrel{\ri}{=} \frac{1}{(1 - e^{-\varepsilon\|z\|^2})^2}\Big(\big|\wht Q_\varepsilon(z) - e^{-\frac{\varepsilon\|z\|^2}{2}}\wht R_\varepsilon(z)\big|^2 + \big|\wht R_\varepsilon(z) - e^{-\frac{\varepsilon\|z\|^2}{2}}\wht Q_\varepsilon(z)\big|^2\Big)\\
&= \frac{1}{(1 - e^{-\varepsilon\|z\|^2})^2}\Big[\big(1 + e^{-\frac{\varepsilon\|z\|^2}{2}}\big)^2\Big(\big|\wht R_\varepsilon(z)\big|^2 + \big|\wht Q_\varepsilon(z)\big|^2\Big) - 2e^{-\frac{\varepsilon\|z\|^2}{2}}\big|\wht R_\varepsilon(z) - \wht Q_\varepsilon(z)\big|^2\Big]\\
&\leq \big[1 - e^{-\frac{\varepsilon\|z\|^2}{2}}\big]^{-2}\Big(\big|\wht R_\varepsilon(z)\big|^2 + \big|\wht Q_\varepsilon(z)\big|^2\Big)
\leq \big[1 - e^{-\frac{\varepsilon}{2}}\big]^{-2}\Big(\big|\wht R_\varepsilon(z)\big|^2 + \big|\wht Q_\varepsilon(z)\big|^2\Big)\\
&\stackrel{\rii}{\leq} \frac{16}{\varepsilon^2}\Big(\big|\wht R_\varepsilon(z)\big|^2 + \big|\wht Q_\varepsilon(z)\big|^2\Big).
\end{aligned}
\end{align}
Here, (i) follows from~\eqref{eqn: Fourier_sol}, and (ii) holds whenever $\varepsilon\in(0,2]$.
By Parseval's identity~\cite{stein2011fourier}, we have
\begin{align*}
&\quad\,\|r_U^\varepsilon\|_{L^2(\mb T^d)}^2 + \|r_V^\varepsilon\|^2_{L^2(\mb T^d)}
= \sum_{z\in\mb Z^d} |\wht r_U^\varepsilon(z)|^2 + |\wht r_V^\varepsilon(z)|^2\\
&\stackrel{\ri}{\leq} \frac{1}{2}|\wht R_\varepsilon(0)|^2 + \frac{16}{\varepsilon^2}\sum_{z\in\mb Z^d-\{0_d\}}\big[|\wht R_\varepsilon(z)|^2 + |\wht Q_\varepsilon(z)|^2\big]\\
&\stackrel{\rii}{\leq} \frac{16}{\varepsilon^2}\Big[\|R_\varepsilon\|^2_{L^2(\mb T^d)} + \|Q_\varepsilon\|^2_{L^2(\mb T^d)}\Big].
\end{align*}
Here, (i) follows from the estimate~\eqref{eqn: fouriersq_ub}, and (ii) is due to Parseval's identity again.

\vspace{0.5em}
\noindent\underline{Step 5: control $\m C^0$-norm of $r_U^\varepsilon$ and $r_V^\varepsilon$.}
By taking $s = \lceil \frac{d+1}{2}\rceil$ as the smallest integer greater than $d/2$, we have~\cite[Chapter 7.6,][]{cerda2010linear} 
\begin{align*}
\|r_U^\varepsilon\|_{H^s(\mb T^d)}^2 \lesssim_s \sum_{z\in\mb Z^d} \big(1 + \|z\|^2\big)^s|\wht r_U^\varepsilon(z)|^2
\quad\mx{and}\quad
\|r_V^\varepsilon\|_{H^s(\mb T^d)}^2 \lesssim_s \sum_{z\in\mb Z^d} \big(1 + \|z\|^2\big)^s|\wht r_V^\varepsilon(z)|^2
\end{align*}
where the notation $\lesssim_s$ omits universal constants $C_s > 0$ depending only on $s$ varying from lines to lines. Therefore, by Sobolev embedding theorem, it holds that 
\begin{align}\label{eqn: Linf_estimate}
\begin{aligned}
\|r_U^\varepsilon\|_{L^\infty(\mb T^d)}^2 + \|r_V^\varepsilon\|_{L^\infty(\mb T^d)}^2
&\lesssim_{d, s} \|r_U^\varepsilon\|_{H^s(\mb T^d)}^2 + \|r_V^\varepsilon\|_{H^s(\mb T^d)}^2 
\lesssim_{s} \sum_{z\in\mb Z^d} \big(1 + \|z\|^2\big)^s\Big(|\wht r_U^\varepsilon(z)|^2 + |\wht r_V^\varepsilon(z)|^2 \Big)\\
&\stackrel{\ri}{\leq} \frac{1}{2}\big|\wht R_\varepsilon(0)\big|^2 + \sum_{z\in\mb Z^d-\{0_d\}}\big(1 + \|z\|^2\big)^s\cdot\frac{16}{\varepsilon^2}\Big(\big|\wht R_\varepsilon(z)\big|^2 + \big|\wht Q_\varepsilon(z)\big|^2\Big)\\
&\leq \sum_{z\in\mb Z^d}\big(1 + \|z\|^2\big)^s\cdot\frac{16}{\varepsilon^2}\Big(\big|\wht R_\varepsilon(z)\big|^2 + \big|\wht Q_\varepsilon(z)\big|^2\Big)\\
&\lesssim_s \frac{16}{\varepsilon^2}\Big[\|R_\varepsilon\|_{H^s(\mb T^d)}^2 + \|Q_\varepsilon\|_{H^s(\mb T^d)}^2\Big].
\end{aligned}
\end{align}
Here, (i) follows from the estimate~\eqref{eqn: fouriersq_ub} and $\wht r_U^\varepsilon(0_d) = \wht r_V^\varepsilon(0_d) = \frac{1}{2}\wht R_\varepsilon(0_d)$.

Now, let us control $\|R_\varepsilon\|_{H^s(\mb T^d)}$ and $\|Q_\varepsilon\|_{H^s(\mb T^d)}$. Following the similar approach as in Lemma~\ref{lem: int_remainder}, we have
\begin{align*}
\|R_\varepsilon\|_{H^s(\mb T^d)}
&= \bigg\|\int_0^\varepsilon\partial_\varepsilon^2 R_\varepsilon\big|_{\varepsilon=t}\cdot (\varepsilon- t)\,\dd t\bigg\|_{H^s(\mb T^d)}
\leq \frac{\varepsilon^2}{2}\sup_{0\leq t\leq \varepsilon} \big\|\partial_\varepsilon^2 R_\varepsilon |_{\varepsilon = t}\big\|_{H^s(\mb T^d)},
\end{align*}
and similarly $\|Q_\varepsilon\|_{H^s(\mb T^d)} \leq \frac{\varepsilon^2}{2}\sup_{0\leq t\leq \varepsilon}\|\partial_\varepsilon^2 Q_\varepsilon|_{\varepsilon=t}\|_{H^s(\mb T^d)}$.
$H^s(\mb T^d)$ is a Banach algebra since $s > d/2$, indicating that
\begin{align*}
\|fg\|_{H^s} \leq C_{s,d}\|f\|_{H^s}\|g\|_{H^s}
\end{align*}
holds for all $f,g\in H^s(\mb T^d)$ with a universal constant $C_{s,d} > 0$.
Particularly, we have
\begin{align*}
\big\|\partial_\varepsilon^2 R_\varepsilon\big\|_{H^s(\mb T^d)}
&= \bigg\|\rho_\varepsilon\partial_\varepsilon^2\frac{R_\varepsilon}{\rho_\varepsilon} + 2\partial_\varepsilon\rho_\varepsilon \partial_\varepsilon\frac{R_\varepsilon}{\rho_\varepsilon} + \frac{R_\varepsilon}{\rho_\varepsilon}\partial_\varepsilon^2\rho_\varepsilon\bigg\|_{H^s(\mb T^d)}\\
&\lesssim_{s, d} \|\rho_\varepsilon\|_{H^s(\mb T^d)}\bigg\|\partial_s^2\frac{R_\varepsilon}{\rho_\varepsilon}\bigg\|_{H^s(\mb T^d)} + \|\partial_\varepsilon\rho_\varepsilon\|_{H^s(\mb T^d)}\bigg\|\partial_s\frac{R_\varepsilon}{\rho_\varepsilon}\bigg\|_{H^s(\mb T^d)} + \|\partial_\varepsilon^2\rho_\varepsilon\|_{H^s(\mb T^d)}\bigg\|\frac{R_\varepsilon}{\rho_\varepsilon}\bigg\|_{H^s(\mb T^d)}
\end{align*}

Combining with Equation~\eqref{eqn: p-deri} in Section~\ref{sec: pf_int_remainder}, we can control $\sup_{t\in[0, \varepsilon]}\|\partial_t^2 R_t\|_{H^s}$ through a constant depending only on $\sup\big\{\|\rho_t^{(k)}\|_{H^{2k'+s}}: t\in[0,\varepsilon], \,\,k,k'\in\mb N,\,\, k+k'\leq 2\big\}$, $\|u_0\|_{H^{s+4}}, \cdots, \|u_K\|_{H^{s+4}}$, and $\|v_0\|_{H^{s+4}}, \cdots, \|v_K\|_{H^{s+4}}$. A similar upper bound also applies to $\sup_{t\in[0, \varepsilon]}\|\partial_t^2 Q_t\|_{H^s}$.

Recall that $s = \lceil\frac{d+1}{2}\rceil$. Combining the above estimate with the upper bound in Equation~\eqref{eqn: Linf_estimate}, the proof is completed.
\qed

\subsection{Proof of Lemma~\ref{lem: Schauder}}\label{sec:Schauder-Lemma}
Since $\mb T^d$ is a compact flat smooth manifold without boundaries, we can directly use the Schauder interior estimate (see e.g. Corollary 2.29 in~\cite{fernandez2022regularity}) to get that
\begin{align}\label{eqn: Schauder1}
\|u\|_{\m C^{s+2,\alpha}} \leq C'\big(\|u\|_{\m C^0} + \|f\|_{\m C^{s,\alpha}}\big)
\end{align}
holds for some constant $C' = C'(\alpha, s, d, L_\rho, \|\rho\|_{\m C^{s+1,\alpha}}) > 0$. The key is to further estimate $\|u\|_{\m C^0}$. In fact, we will show that there exists a constant $C'' = C''(\alpha, d, L_\rho, \|\rho\|_{\m C^{1,\alpha}}) > 0$, such that
\begin{align}\label{eqn: Schauder2}
\|u\|_{\m C^0} \leq C''\|f\|_{\m C^{0,\alpha}}.
\end{align}

To prove this result, define
\begin{align*}
\m S &\coloneqq \bigg\{u\in\m C^{2,\alpha}: \|u\|_{\m C^{0,\alpha}}=1, \int_{\mb T^d} u\,\dd x  =0\bigg\}\\
\m P_M &\coloneqq \big\{\rho\in\m C^{1,\alpha}: 0< L_\rho \leq \rho(x)\leq L_\rho^{-1} < \infty\,\,\mx{and}\,\,\|\rho\|_{\m C^{1,\alpha}} \leq M\big\}.
\end{align*}
Moreover, define
\begin{align*}
\epsilon(\rho)\coloneqq \inf_{u\in\m S}\big\|\nabla\cdot(\rho\nabla u)\big\|_{\m C^{0,\alpha}}
\quad\mx{and}\quad \epsilon_0 = \inf_{\rho\in\m P_M} \epsilon(\rho).
\end{align*}
With these notation, we know
\begin{align*}
\Big\|\frac{f}{\|u\|_{\m C^{0,\alpha}}}\Big\|_{\m C^{0,\alpha}} &= \Big\|\nabla\cdot\Big(\rho_0\nabla\frac{u}{\|u\|_{\m C^{0,\alpha}}}\Big)\Big\|_{\m C^{0,\alpha}} \geq \epsilon(\rho_0) \geq \epsilon_0.
\end{align*}
Therefore, we have 
\begin{align*}
\|u\|_{\m C^{0}} \leq \|u\|_{\m C^{0,\alpha}} \leq \epsilon_0^{-1}{\|f\|_{\m C^{0,\alpha}}}.
\end{align*}
Since $\epsilon_0$ only depends on $\alpha$, $d$, $L_\rho$ and $M$, we only need to prove $\epsilon_0 > 0$.

Suppose we have $\epsilon_0 = 0$. By definition, there exists a sequence $\{\rho_n\}_{n\in Z_+} \subset \m P_M$, so that $\epsilon(\rho_n) \leq \frac{1}{n}$. By the definition of $\epsilon(\rho_n)$, there exists $u_n\in\m S$, such that
\begin{align}\label{eqn: f0a-norm}
\|f_n\|_{\m C^{0,\alpha}} \leq \epsilon(\rho_n) + \frac{1}{n} \leq \frac{2}{n},
\end{align}
where we let $f_n \coloneqq \nabla\cdot(\rho_n\nabla u_n)$ for simplicity. Combining with the Schauder interior estimate~\eqref{eqn: Schauder1}, we have
\begin{align*}
\|u_n\|_{\m C^{2,\alpha}} \leq C'\big[\|u_n\|_{\m C^{0}} + \|f_n\|_{\m C^{0,\alpha}}\big]
\stackrel{\ri}{\leq} C'\Big[1 + \frac{2}{n}\Big] \leq 3C'.
\end{align*}
Here, (i) follows from $u_n\in\m S$ and the estimate~\eqref{eqn: f0a-norm}.

For the sequence $\{f_n\}_{n\in\mb Z_+}$, using the estimate~\eqref{eqn: f0a-norm} again, we know $\{f_n\}_{n\in\mb Z_+}$ is uniformly bounded in $\m C^0(\mb T^d)$ and satisfies
\begin{align*}
|f_n(x) - f_n(x')| \leq \|f_n\|_{\m C^{0,\alpha}} \td(x,x')^\alpha \leq 2\td(x,x')^\alpha,
\end{align*}
indicating that $\{f_n\}_{n\in\mb Z_+}$ is equicontinuous. By Arzela--Ascoli Theorem, there exists a subsequence (we still use $\{f_n\}_{n\in\mb Z_+}$ for simplicity) such that $f_n\to f^\ast$ in $\m C^0$. Similarly, since $\|\rho_n\|_{\m C^{1,\alpha}}$ is uniformly bounded, using the above argument twice, we know there exists $\rho^\ast\in\m C^{1}$ such that $\rho_n\to \rho^\ast$ in $\m C^1$. We can also assume $u_n\to u^\ast$ in $\m C^2$ for some $u^\ast$. Therefore, we have
\begin{align*}
\nabla\cdot(\rho^\ast\nabla u^\ast) 
= \lim_{n\to\infty} \nabla\cdot(\rho_n\nabla u_n) = \lim_{n\to\infty} f_n = f^\ast.
\end{align*}
However, we know $f^\ast = 0$ due to the estimate~\eqref{eqn: f0a-norm}. This implies $u^\ast$ is a constant, which is impossible as we need
\begin{align*}
\int_{\mb T^d}u^\ast\,\dd x = \lim_{n\to\infty}\int_{\mb T^d} u_n\,\dd x = 0
\end{align*}
and $\|u^\ast\|_{\m C^{0,\alpha}} = \lim_{n\to\infty}\|u_n\|_{\m C^{0,\alpha}} = 1$, leading to a contradiction.

Now, we know $\epsilon_0 > 0$, indicating that~\eqref{eqn: Schauder2} is true. We finis the proof.

\subsection{Proof of Lemma~\ref{lem: DensityRatioLB}}\label{app: DensityRatioLB}

Recall that
\begin{align*}
\gamma[\phi](y) = \frac{\gamma(y)e^{\m T_\mu^\varepsilon[\phi](y)}}{\int_{\mb T^d}\gamma(y)e^{\m T_\mu^\varepsilon[\phi](y)}\,\dd y}.
\end{align*}
Therefore, we have
\begin{align*}
\frac{\nu_Q(y)}{\gamma[\phi](y)}
= \frac{\nu(y)}{\nu(Q)\gamma(y)}\cdot \frac{e^{\m T_\mu^\varepsilon[\phi](y)}}{\int_{\mb T^d}e^{\m T_\mu^\varepsilon[\phi](y)}\,\dd\gamma(y)},
\end{align*}
indicating that
\begin{align*}
\inf_{y\in Q}\frac{\nu(y)}{\gamma(y)} \cdot \nu(Q)^{-1} e^{-\osc(\m T_\mu^\varepsilon[\phi])}
\leq\frac{\nu(y)}{\gamma[\phi](y)}
\leq \sup_{y\in Q}\frac{\nu(y)}{\gamma(y)} \cdot\nu(Q)^{-1} e^{\osc(\m T_\mu^\varepsilon[\phi])},
\end{align*}
where $\osc(f) \coloneqq \sup_{x\in\mb T^d} f(x) - \inf_{x\in\mb T^d}f(x)$ is the oscillation of the function $f$.
Note that
\begin{align}
\osc(\m T_\mu^\varepsilon[\phi]) 
&\stackrel{\ri}{\leq} \sup_{x, y\in\mb T^d} c_\varepsilon(x, y) - \inf_{x, y\in\mb T^d} c_\varepsilon(x, y)\notag\\
&=\osc(c_\varepsilon)= \varepsilon\Big[\sup_{x\in\mb T^d}\log\m K_\varepsilon(x) - \inf_{y\in\mb T^d}\log\m K_\varepsilon(y)\Big]
\stackrel{\rii}{<} 4\pi^2d.\label{eqn: osc_cost}
\end{align}
Here, (i) is due to Equation (17) in~\citep{chizat2022trajectory}, and (ii) is due to Proposition~\ref{prop: bound_Gauss_ker} and $\varepsilon < \frac{\pi^2}{2}$. Moreover, we have
\begin{align*}
m_\nu e^{-\frac{r^2}{\alpha}}\vol(Q)
\leq\frac{\nu(y)}{\gamma(y)} = \nu(y) e^{\frac{1}{\alpha}\td(x_0, y)^2} \int_Q e^{-\frac{1}{\alpha}\td(x_0, z)^2}\,\dd z 
\leq M_\nu e^{\frac{r^2}{\alpha}}\vol(Q).
\end{align*}
Therefore, combining the pieces above yields
\begin{align*}
m_\nu e^{-\frac{r^2}{\alpha} - 4\pi^2 d}\cdot \frac{\vol(Q)}{\nu(Q)}
\leq\frac{\nu_Q(y)}{\gamma[\phi](y)} 
\leq M_\nu e^{\frac{r^2}{\alpha} + 4\pi^2 d}\cdot\frac{\vol(Q)}{\nu(Q)}.
\end{align*}
\subsection{Proof of Lemma~\ref{lem: integ_lb}}\label{app: integ_lb}
Recall that
\begin{align*}
\pi^\nu[\phi_\varepsilon + \lambda\chi](x) 
&=\int_{\mb T^d}\pi_y[\phi_\varepsilon + \lambda\chi](x)\,\dd\nu(y)\\
&= \int_{\mb T^d}\frac{\mu(x)e^{\frac{\phi_\varepsilon(x) + \lambda\chi(x) - c_\varepsilon(x, y)}{\varepsilon}}}{\int_{\mb T^d}e^{\frac{\phi_\varepsilon(x) + \lambda\chi(x) - c_\varepsilon(x,y)}{\varepsilon}}\,\dd\mu(x)}\,\dd\nu(y)
= \mu(x)e^{\frac{[\phi_\varepsilon+\lambda\chi](x)}{\varepsilon}}e^{-\frac{\m T_\nu^\varepsilon\circ\m T_\mu^\varepsilon[\phi_\varepsilon+\lambda\chi](x)}{\varepsilon}}.
\end{align*}
So, we have $\pi^\nu[\phi_\varepsilon + \lambda\chi] \ll \mu$ and $\mu\ll \pi^\nu[\phi_\varepsilon + \lambda\chi]$. Applying~\cite[Lemma 4.1,][]{chizat2025sharper} implies
\begin{align}\label{eqn: var_LB}
\mb V_{X\sim\pi^\nu[\phi_\varepsilon + \lambda\chi]}[\chi(X)]
\geq \frac{1}{2}\mb V_{X\sim\mu}[\chi(X)] - 4\|\chi\|_{L^\infty(\mb T^d)}^2\KL(\mu\,\|\,\pi^\nu[\phi_\varepsilon + \lambda\chi]).
\end{align}
Note that
\begin{align*}
&\quad\,\KL(\mu\,\|\,\pi^\nu[\phi_\varepsilon + \lambda\chi]) 
= \mb E_\mu\Big[\log\frac{\mu}{\pi^\nu[\phi_\varepsilon + \lambda\chi]}\Big]
=\frac{1}{\varepsilon}\mb E_\mu\Big[\m T_\nu^\varepsilon\circ\m T_\mu^\varepsilon[\phi_\varepsilon+\lambda\chi] - [\phi_\varepsilon+\lambda\chi]\Big]\\
&= \frac{1}{\varepsilon}\bigg[\int\m T_\nu^\varepsilon\circ\m T_\mu^\varepsilon[\phi_\varepsilon+\lambda\chi]\,\dd\mu + \int \m T_\mu^\varepsilon[\phi_\varepsilon+\lambda\chi]\,\dd\nu\bigg] - \frac{1}{\varepsilon}\bigg[\int\phi_\varepsilon+\lambda\chi\,\dd\mu + \int\m T_\mu^\varepsilon[\phi_\varepsilon+\lambda\chi]\,\dd\nu\bigg]\\
&=\frac{1}{\varepsilon}\Big[\overline{\m I}_{\mu, \nu}^\varepsilon\big[\m T_\mu^\varepsilon[\phi_\varepsilon+\lambda\chi]\big] - \m I_{\mu, \nu}^\varepsilon[\phi_\varepsilon+\lambda\chi]\Big]\\
&\leq \frac{1}{\varepsilon}\big[\m I_{\mu, \nu}^\varepsilon[\phi_\varepsilon] - \m I_{\mu, \nu}^\varepsilon[\phi_\varepsilon + \lambda\chi]\big].
\end{align*}
Here, the last inequality is again due to the strong duality~\eqref{eqn: strong_dual} and the optimality of $\phi_\varepsilon$, i.e.
\begin{align*}
\m I_{\mu, \nu}^\varepsilon[\phi_\varepsilon]
= \eot_\varepsilon(\mu, \nu)
= \overline{\m I}_{\mu, \nu}^\varepsilon[\psi_\varepsilon] \geq \overline{\m I}_{\mu, \nu}^\varepsilon[\m T_\mu^\varepsilon\big[\phi_\varepsilon + \lambda\chi]\big].
\end{align*}
Note that $\m I_{\mu, \nu}^\varepsilon$ is concave, so we have
\begin{align*}
\m I_{\mu, \nu}^\varepsilon[\phi_\varepsilon + \lambda\chi] \geq (1-\lambda)\m I_{\mu, \nu}^\varepsilon[\phi_\varepsilon] + \lambda\m I_{\mu, \nu}^\varepsilon(\phi_\varepsilon + \chi).
\end{align*}
Therefore, we have
\begin{align*}
\KL(\mu\,\|\,\pi^\nu[\phi_\varepsilon + \lambda\chi]) \leq \frac{\lambda}{\varepsilon} \Big[\m I_{\mu, \nu}^\varepsilon[\phi_\varepsilon] - \m I_{\mu, \nu}^\varepsilon[\phi_\varepsilon + \chi]\Big].
\end{align*}
Plugging the above inequality in~\eqref{eqn: var_LB} yields the result.

\subsection{Proof of Lemma~\ref{lem: simplification}}\label{app: simplification}
The proof consists of a few steps.

\vspace{0.5em}
\noindent\underline{Step 1: terms in the first square brackets.}
Note that
\begin{align*}
\m H(\bar\rho_{t_j}^\nu) - \m H(\bar \rho_{t_{j-1}}^\nu)
&= \int_{t_{j-1}}^{t_j}
\frac{\dd}{\dd t} \m H(\bar\rho_t^\nu) \,\dd t
= \int_{t_{j-1}}^{t_j}\!\int_{\mb T^d}(1+\log\bar\rho_t^\nu)\cdot\partial_t\bar\rho_t^\nu\,\dd x\dd t\\
&= -\int_{t_{j-1}}^{t_j}\!\int_{\mb T^d}(1+\log\bar\rho_t^\nu)\nabla\cdot(\bar\rho_t^\nu\nabla\bar \Phi_t^\nu)\,\dd x\dd t
= \int_{t_{j-1}}^{t_j}\!\int_{\mb T^d}\big\langle \nabla\bar \Phi_t^\nu, \nabla\log\bar\rho_t^\nu\big\rangle\,\dd\bar\rho_t^\nu\dd t,
\end{align*}
and also Benamou--Brenier's formula implies
\begin{align}\label{eqn: BB-formula}
\W_2^2(\rho_{t_{j-1}}^\nu, \rho_{t_j}^\nu)
&= \frac{\varepsilon_j}{2}\int_{t_{j-1}}^{t_j}\!\!\int_{\mb T^d}\|\nabla\bar \Phi_t^\nu\|^2\bar\rho_t^\nu\,\dd x\dd t.
\end{align}
Therefore, we have
\begin{align*}
&\quad\,\frac{1}{\varepsilon_j}\bigg[\W_2^2(\rho_{t_{j-1}}^\nu, \rho_{t_j}^\nu) - \frac{\varepsilon_j}{2}\big[\m H(\rho_{t_{j-1}}^\nu) + \m H(\rho_{t_j}^\nu)\big] + \frac{\varepsilon_j}{8}\int_{t_{j-1}}^{t_j}\!\int_{\mb T^d}\|\nabla\log\bar\rho_{t}^\nu\|^2\bar\rho_t^\nu\,\dd x\dd t\bigg]\\
&= \frac{1}{2}\int_{t_{j-1}}^{t_j}\!\int_{\mb T^d}\|\nabla\bar \Phi_t^\nu\|^2\,\dd\bar\rho_t^\nu\dd t - \m H(\rho_{t_j}^\nu) + \frac{1}{2}\big[\m H(\rho_{t_j}^\nu) - \m H(\rho_{t_{j-1}}^\nu)\big] +  \frac{1}{8}\int_{t_{j-1}}^{t_j}\!\int_{\mb T^d}\|\nabla\log\bar\rho_t^\nu\|^2\,\dd\bar\rho_t^\nu\dd t\\
&= \frac{1}{2}\int_{t_{j-1}}^{t_j}\!\int_{\mb T^d}\|\nabla\bar \Phi_t^\nu\|^2\,\dd\bar\rho_t^\nu\dd t - \m H(\rho_{t_j}^\nu) + \frac{1}{2}\int_{t_{j-1}}^{t_j}\!\int_{\mb T^d}\big\langle\nabla\bar \Phi_t^\nu, \nabla\log\bar\rho_t^\nu\big\rangle\,\dd\bar\rho_t^\nu\dd t +  \frac{1}{8}\int_{t_{j-1}}^{t_j}\!\int_{\mb T^d}\|\nabla\log\bar\rho_t^\nu\|^2\,\dd\bar\rho_t^\nu\dd t\\
&= \frac{1}{2}\int_{t_{j-1}}^{t_j}\!\int_{\mb T^d}\Big\|\nabla\bar \Phi_t^\nu + \frac{1}{2}\nabla\log\bar\rho_t^\nu\Big\|^2\,\dd\bar\rho_t^\nu\dd t - \m H(\rho_{t_j}^\nu).
\end{align*}

\vspace{0.5em}
\noindent\underline{Step 2: terms in the second square brackets.}
Simple rearranging leads to
\begin{align*}
&\quad\,\int_{\mb T^d}\log u_{t_{j-1},0} + \varepsilon_j\frac{u_{t_{j-1},1}}{u_{t_{j-1},0}}\,\dd\rho_{t_{j-1}}^\nu + \int_{\mb T^d}\log v_{t_{j-1},0} + \varepsilon_j\frac{v_{t_{j-1},1}}{v_{t_{j-1},0}}\,\dd\rho_{t_j}^\nu\\
&= \bigg(\int_{\mb T^d}\log u_{t_{j-1},0}\,\dd\rho_{t_{j-1}}^\nu
+ \int_{\mb T^d}\log v_{t_{j-1},0} \,\dd\rho_{t_j}^\nu\bigg)
+ \varepsilon_j\bigg(\int_{\mb T^d}\frac{u_{t_{j-1},1}}{u_{t_{j-1},0}}\,\dd\rho_{t_{j-1}}^\nu
+ \int_{\mb T^d}\frac{v_{t_{j-1},1}}{v_{t_{j-1},0}}\,\dd\rho_{t_j}^\nu\bigg).
\end{align*}

\underline{Step 2.1: terms in the first parentheses.}
 Recall that we have $\partial_t\rho_t^\mu + \nabla\cdot(\rho_t^\mu\nabla\Phi_t^\mu) = 0$. Also, Theorem~\ref{thm: existence} implies
\begin{align*}
\nabla\cdot(\rho_t^\mu\nabla\log u_{t,0}) = \partial_t\rho_t^\mu - \frac{\Delta\rho_t^\mu}{2}
\quad\mx{and}\quad
\nabla\cdot(\rho_t^\mu\nabla\log v_{t,0}) = -\partial_t\rho_t^\mu - \frac{\Delta\rho_t^\mu}{2},
\end{align*}
so, we have
\begin{align}\label{eqn: 1st-order_sol}
\log u_{t,0} = -\Phi_t^\mu - \frac{1}{2}\log\rho_t^\mu
\quad\mx{and}\quad
\log v_{t,0} = \Phi_t^\mu - \frac{1}{2}\log\rho_t^\mu.
\end{align}
Therefore, we have
\begin{align*}
&\quad\,\int_{\mb T^d}\log u_{t_{j-1},0}\,\dd\rho_{t_{j-1}}^\nu + \int_{\mb T^d}\log v_{t_{j-1},0}\,\dd\rho_{t_j}^\nu\\
&\stackrel{\ri}{=} \int_{\mb T^d} - \Phi_{t_{j-1}}^\mu - \frac{1}{2}\log\rho_{t_{j-1}}^\mu\,\dd\rho_{t_{j-1}}^\nu + \int_{\mb T^d} \Phi_{t_{j-1}}^\mu - \frac{1}{2}\log\rho_{t_{j-1}}^\mu\,\dd\rho_{t_j}^\nu\\
&= \int_{\mb T^d}\Phi_{t_{j-1}}^\mu + \frac{1}{2}\log\rho_{t_{j-1}}^\mu\,\dd[\rho_{t_j}^\nu - \rho_{t_{j-1}}^\nu] - \int_{\mb T^d}\log\rho_{t_{j-1}}^\mu\,\dd\rho_{t_j}^\nu\\
&= \int_{t_{j-1}}^{t_j}\!\int_{\mb T^d}\Big(\Phi_{t_{j-1}}^\mu + \frac{1}{2}\log\rho_{t_{j-1}}^\mu\Big)\partial_t\bar\rho_t^\nu\,\dd x\dd t - \int_{\mb T^d}\log\rho_{t_{j-1}}^\mu\,\dd\rho_{t_j}^\nu\\
&\stackrel{\rii}{=} \int_{t_{j-1}}^{t_j}\!\int_{\mb T^d}\Big\langle\nabla\Phi_{t_{j-1}}^\mu + \frac{1}{2}\nabla\log\rho_{t_{j-1}}^\mu, \nabla\bar\Phi_t^\nu\Big\rangle\,\dd\bar\rho_t^\nu\dd t - \int_{\mb T^d}\log\rho_{t_{j-1}}^\mu\,\dd\rho_{t_j}^\nu
\end{align*}
Here, (i) follows from~\eqref{eqn: 1st-order_sol}, and (ii) follows from the continuity equation $\partial_t\bar\rho_t^\nu + \nabla\cdot(\bar\rho_t^\nu\nabla\bar\Phi_t^\nu) = 0$ and integration by parts.
Note that we have
\begin{align*}
\int_{\mb T^d}\log\rho_{t_{j-1}}^\mu\,\dd[\rho_{t_j}^\nu - \rho_{t_{j-1}}^\nu]
&= \int_{t_{j-1}}^{t_j}\!\int_{\mb T^d}\log\rho_{t_{j-1}}^\mu\cdot \partial_t\bar\rho_t^\nu\,\dd x\dd t
= \int_{t_{j-1}}^{t_j}\!\int_{\mb T^d} \big\langle\nabla\log\rho_{t_{j-1}}^\mu, \nabla\bar\Phi_t^\nu\big\rangle\,\dd\bar\rho_t^\nu\dd t.
\end{align*}
Here, the last equality again follows from the continuity equation and integration by parts. Therefore, 
\begin{align*}
\int_{\mb T^d}\log\rho_{t_{j-1}}^\mu\,\dd\rho_{t_j}^\nu 
= \int_{\mb T^d}\log\rho_{t_{j-1}}^\mu\,\dd\rho_{t_{j-1}}^\nu + \int_{t_{j-1}}^{t_j}\!\int_{\mb T^d} \big\langle\nabla\log\rho_{t_{j-1}}^\mu, \nabla\bar\Phi_t^\nu\big\rangle\,\dd\bar\rho_t^\nu\dd t.
\end{align*}
To sum up, we get
\begin{align*}
\int_{\mb T^d}\log u_{t_{j-1},0}\,\dd\rho_{t_{j-1}}^\nu + \int_{\mb T^d}\log v_{t_{j-1},0}\,\dd\rho_{t_j}^\nu
&= \int_{t_{j-1}}^{t_j}\!\int_{\mb T^d}\Big\langle\nabla\Phi_{t_{j-1}}^\mu + \frac{1}{2}\nabla\log\rho_{t_{j-1}}^\mu, \nabla\bar\Phi_t^\nu\Big\rangle\,\dd\bar\rho_t^\nu\dd t\\
&\qquad -\int_{\mb T^d}\log\rho_{t_{j-1}}^\mu\,\dd\rho_{t_{j-1}}^\nu - \int_{t_{j-1}}^{t_j}\!\int_{\mb T^d} \big\langle\nabla\log\rho_{t_{j-1}}^\mu, \nabla\bar\Phi_t^\nu\big\rangle\,\dd\bar\rho_t^\nu\dd t.
\end{align*}

\underline{Step 2.2: terms in the second parentheses.}
By definition~\eqref{eqn: asymp-Schro}, we have 
\begin{align*}
0 &= v_{t,0}^\dagger v_{t,1} + v_{t,0}v_{t,1}^\dagger
= v_{t,0}^\dagger v_{t,1} + v_{t,0}\Big[\frac{\Delta(u_{t,0}\rho_t)}{2} + u_{t,1}\rho_t\Big]\\
&= v_{t,0}^\dagger v_{t,1} + v_{t,0}\rho_t u_{t,1} + \frac{1}{2} v_{t,0}\Delta(u_{t,0}\rho_t)\\
&= v_{t,0}^\dagger v_{t,1} + u_{t,0}^\dagger u_{t,1} + \frac{1}{2}v_{t,0}\Delta v_{t,0}^\dagger.
\end{align*}
Here, in the last two equalities, we use $u_{t,0}^\dagger = 1/u_{t,0} = \rho_t v_{t,0}$ and $v_{t, 0}^\dagger = 1/v_{t,0} = u_{t,0}\rho_{t}$. So, we have
\begin{align*}
&\quad\,\int_{\mb T^d}\frac{u_{t_{j-1},1}}{u_{t_{j-1}, 0}}\,\dd\rho_{t_{j-1}}^\nu + \int_{\mb T^d}\frac{v_{t_{j-1},1}}{v_{t_{j-1},0}}\,\dd\rho_{t_j}^\nu
= \int_{\mb T^d}{u_{t_{j-1},1}}{u_{t_{j-1}, 0}^\dagger}\,\dd\rho_{t_{j-1}}^\nu + \int_{\mb T^d}{v_{t_{j-1},1}}{v_{t_{j-1},0}^\dagger}\,\dd\rho_{t_j}^\nu\\
&= \underbrace{\int_{\mb T^d} v_{t_{j-1},1}v_{t_{j-1},0}^\dagger\,\dd\big[\rho_{t_j}^\nu - \rho_{t_{j-1}}^\nu\big]}_{\eqqcolon\delta_{1,j}} - \frac{1}{2}\int_{\mb T^d} \frac{\Delta v_{t_{j-1},0}^\dagger}{v_{t_{j-1}, 0}^\dagger}\,\dd\rho_{t_{j-1}}^\nu.
\end{align*}
Since $\log v_{t_{j-1},0}^\dagger = -\log v_{t_{j-1}, 0} = -\Phi_{t_{j-1}}^\mu + \frac{1}{2}\log\rho_{t_{j-1}}^\mu$, which indicates that
\begin{align*}
\frac{\Delta v_{t_{j-1}, 0}^\dagger}{v_{t_{j-1},0}^\dagger} 
= \Delta\log v_{t_{j-1}, 0}^\dagger + \big\|\nabla\log v_{t_{j-1},0}^\dagger\big\|^2
= \Big[\frac{1}{2}\Delta\log\rho_{t_{j-1}}^\mu - \Delta\Phi_{t_{j-1}}^\mu\Big] + \Big\|\frac{1}{2}\nabla\log\rho_{t_{j-1}}^\mu - \nabla\Phi_{t_{j-1}}^\mu\Big\|^2,
\end{align*}
we have
\begin{align*}
&\quad\,\varepsilon_j\bigg[\int_{\mb T^d}\frac{u_{t_{j-1},1}}{u_{t_{j-1}, 0}}\,\dd\rho_{t_{j-1}}^\nu + \int_{\mb T^d}\frac{v_{t_{j-1},1}}{v_{t_{j-1},0}}\,\dd\rho_{t_j}^\nu - \delta_{1,j}\bigg]\\
&= -\frac{\varepsilon_j}{2}\int_{\mb T^d}\Delta\log v_{t_{j-1},0}^\dagger + \big\|\nabla\log v_{t_{j-1},0}^\dagger\big\|^2\,\dd\rho_{t_{j-1}}^\nu\\
&= - \frac{\varepsilon_j}{2}\int_{\mb T^d}\Big[\frac{1}{2}\Delta\log\rho_{t_{j-1}}^\mu - \Delta\Phi_{t_{j-1}}^\mu\Big] + \Big\|\frac{1}{2}\nabla\log\rho_{t_{j-1}}^\mu - \nabla\Phi_{t_{j-1}}^\mu\Big\|^2\,\dd\rho_{t_{j-1}}^\nu\\
&= - \frac{1}{2}\int_{t_{j-1}}^{t_j}\!\int_{\mb T^d}\Big[\frac{1}{2}\Delta\log\rho_{t_{j-1}}^\mu - \Delta\Phi_{t_{j-1}}^\mu\Big] + \Big\|\frac{1}{2}\nabla\log\rho_{t_{j-1}}^\mu - \nabla\Phi_{t_{j-1}}^\mu\Big\|^2\,\dd\bar\rho_{t}^\nu\dd t + \delta_{2,j},
\end{align*}
where we define the error term
\begin{align*}
\delta_{2,j}&\coloneqq
\frac{1}{2}\int_{t_{j-1}}^{t_j}\!\int_{\mb T^d}\Big[\frac{1}{2}\Delta\log\rho_{t_{j-1}}^\mu - \Delta\Phi_{t_{j-1}}^\mu\Big] + \Big\|\frac{1}{2}\nabla\log\rho_{t_{j-1}}^\mu - \nabla\Phi_{t_{j-1}}^\mu\Big\|^2\,\dd\bar\rho_{t}^\nu\dd t\\
&\qquad\qquad- \frac{\varepsilon_j}{2}\int_{\mb T^d}\Big[\frac{1}{2}\Delta\log\rho_{t_{j-1}}^\mu - \Delta\Phi_{t_{j-1}}^\mu\Big] + \Big\|\frac{1}{2}\nabla\log\rho_{t_{j-1}}^\mu - \nabla\Phi_{t_{j-1}}^\mu\Big\|^2\,\dd\rho_{t_{j-1}}^\nu.
\end{align*}
Using integration by parts, we have
\begin{align*}
\int_{t_{j-1}}^{t_j}\!\int_{\mb T^d}\Big[\frac{1}{2}\Delta\log\rho_{t_{j-1}}^\mu - \Delta\Phi_{t_{j-1}}^\mu\Big]\,\dd\bar\rho_{t}^\nu\dd t
= -\int_{t_{j-1}}^{t_j}\!\int_{\mb T^d}\Big\langle\frac{1}{2}\nabla\log\rho_{t_{j-1}}^\mu - \nabla\Phi_{t_{j-1}}^\mu, \nabla\log\bar\rho_t^\nu\Big\rangle\,\dd\bar\rho_t^\nu\dd t.
\end{align*}
Therefore, we get 
\begin{align*}
&\quad\,\varepsilon_j\bigg[\int_{\mb T^d}\frac{u_{t_{j-1},1}}{u_{t_{j-1}, 0}}\,\dd\rho_{t_{j-1}}^\nu + \int_{\mb T^d}\frac{v_{t_{j-1},1}}{v_{t_{j-1},0}}\,\dd\rho_{t_j}^\nu\bigg]\\
&= -\frac{1}{2}\int_{t_{j-1}}^{t_j}\!\int_{\mb T^d}\Big\|\frac{1}{2}\nabla\log\rho_{t_{j-1}}^\mu - \nabla\Phi_{t_{j-1}}^\mu\Big\|^2 - \Big\langle\frac{1}{2}\nabla\log\rho_{t_{j-1}}^\mu - \nabla\Phi_{t_{j-1}}^\mu, \nabla\log\bar\rho_t^\nu\Big\rangle\,\dd\bar\rho_{t}^\nu\dd t + \varepsilon_j\delta_{1,j} + \delta_{2,j} .
\end{align*}

\vspace{0.5em}
\noindent\underline{Step 3: combining all pieces.} Using the results derived so far, we directly have
\begin{align*}
S_j &= \frac{1}{2}\int_{t_{j-1}}^{t_j}\!\int_{\mb T^d}\Big\|\nabla\Phi_{t_{j-1}}^\mu - \nabla\bar \Phi_t^\nu - \frac{1}{2}\nabla\log\frac{\rho_{t_{j-1}}^\mu}{\bar\rho_t^\nu}\Big\|^2\,\dd\bar\rho_t^\nu\dd t 
- \KL(\rho_{t_{j-1}}^\nu\,\|\,\rho_{t_{j-1}}^\mu) - (\varepsilon_j\delta_{1,j} + \delta_{2,j}).
\end{align*}
We complete the proof of the equality.

\vspace{0.5em}
\noindent\underline{Step 4: control remainder terms.}
To control $\delta_{1,j}$, we can use the Kantorovich--Rubinstein duality and the fact that $\W_1\leq \W_2$ to get
\begin{align*}
\delta_{1,j} \leq \Lip\big(v_{t_{j-1}, 1} v_{t_{j-1},0}^\dagger\big) \W_2(\rho_{t_{j-1}}^\nu, \rho_{t_j}^\nu).
\end{align*}
To control $\delta_{2,j}$, first recall that for every $t\in[0,1)$, we have
\begin{align}\label{eqn: vt-Delta-vt}
v_{t,0}\Delta v_{t,0}^\dagger = \Big[\frac{1}{2}\Delta\log\rho_t^\mu - \Delta\Phi_t^\mu\Big] + \Big\|\frac{1}{2}\nabla\log\rho_t^\mu - \nabla\Phi_t^\mu\Big\|^2.
\end{align}
Therefore, we have
\begin{align*}
\delta_{2,j} 
&= \frac{1}{2} \int_{t_{j-1}}^{t_j}\!\int_{\mb T^d} v_{t_{j-1},0}\Delta v_{t_{j-1},0}^\dagger\,\dd[\bar\rho_t^\nu - \rho_{t_{j-1}}^\nu]\dd t\\
&\stackrel{\ri}{\leq} \frac{1}{2}\Lip\big(v_{t_{j-1},0}\Delta v_{t_{j-1},0}^\dagger\big)\int_{t_{j-1}}^{t_j}\W_2(\bar\rho_t^\nu, \rho_{t_{j-1}}^\nu)\,\dd t\\
&\stackrel{\rii}{=} \frac{\varepsilon_j}{4}\W_2(\rho_{t_{j-1}}^\nu, \rho_{t_j}^\nu)\,\Lip\big(v_{t_{j-1},0}\Delta v_{t_{j-1},0}^\dagger\big).
\end{align*}
Here, (i) again follows from the Kantorovich--Rubinstein duality and the fact that $\W_1\leq \W_2$, while (ii) is due to $\W_2(\bar\rho_t^\nu, \rho_{t_{j-1}}^\nu) = \frac{t - t_{j-1}}{t_j - t_{j-1}}\W_2(\rho_{t_j}^\nu, \rho_{t_{j-1}}^\nu)$.
\section{Facts of Gaussian Kernel on $\mb T^d$}
\begin{proposition}[Uniform bound of $\m K_\varepsilon$]\label{prop: bound_Gauss_ker}
For every $\varepsilon < \frac{\pi^2}{2}$, we have
\begin{align*}
\frac{1}{(2\pi\varepsilon)^\frac{d}{2}} e^{-\frac{\td(x, 0_d)^2}{2\varepsilon}}
\leq \m K_\varepsilon(x)
\leq 2^{d+1}e^{5d}\cdot \frac{1}{(2\pi\varepsilon)^\frac{d}{2}} e^{-\frac{\td(x, 0_d)^2}{2\varepsilon}}.
\end{align*}
\end{proposition}
\begin{proof}
For the lower bound, simply note that
\begin{align*}
\m K_\varepsilon(x)
= \frac{1}{(2\pi\varepsilon)^\frac{d}{2}}\sum_{k\in\mb Z^d}e^{-\frac{\|x - 2\pi k\|^2}{2\varepsilon}}
\geq \frac{1}{(2\pi\varepsilon)^\frac{d}{2}}\min_{k\in\mb Z^d}e^{-\frac{\|x-2\pi k\|^2}{2\varepsilon}} = \frac{1}{(2\pi\varepsilon)^\frac{d}{2}}e^{-\frac{\td(x, 0_d)^2}{2\varepsilon}}.
\end{align*}
For the upper bound, note that
\begin{align*}
\|x - 2\pi k\| \geq 2\pi\|k\| - \|x\| \geq 2\pi\|k\| - 2\pi\sqrt{d} \geq \pi\|k\|
\end{align*}
holds for every $k\in\mb Z^d$ such that $\|k\| \geq 2\sqrt{d}$. So, we know
\begin{align*}
\sum_{k\in\mb Z^d: \|k\| \geq 2\sqrt{d}} e^{-\frac{\|x-2\pi k\|^2}{2\varepsilon}}
&\leq \sum_{k\in\mb Z^d: \|k\| \geq 2\sqrt{d}} e^{-\frac{\pi^2\|k\|^2}{2\varepsilon}}
= \sum_{s = 4d}^\infty e^{-\frac{\pi^2 s}{2\varepsilon}}\cdot\#\{k\in\mb Z^d: \|k\|^2 = s\}\\
&\stackrel{\ri}{\leq} \sum_{s=4d}^\infty e^{-\frac{\pi^2 s}{2\varepsilon}} \cdot 2^d\binom{s+d-1}{d-1}
\stackrel{\rii}{\leq} 2^d\sum_{s=4d}^\infty e^{-\frac{\pi^2 s}{2\varepsilon}} \Big(\frac{e(s+d-1)}{d-1}\Big)^{d-1}\\
&\leq 2^d \sum_{s=4d}^\infty e^{-\frac{\pi^2 s}{2\varepsilon}} e^{d-1}e^s
= 2^de^{d-1}\cdot\frac{e^{-4d(\frac{\pi^2}{2\varepsilon}-1)}}{1 - e^{1-\frac{\pi^2}{2\varepsilon}}}\\
&\stackrel{\riii}{<} 2^d e^{2d} e^{-\frac{\td(x, 0_d)^2}{2\varepsilon}}
\end{align*}
for every $\varepsilon < \frac{\pi^2}{2}$.
Here, in (i), we use the fact that
\begin{align}\label{eqn: count}
\#\{k\in\mb Z^d: \|k\|^2 = s\}
\leq 2^d\cdot\#\{k\in\mb N^d: k_1 + \cdots + k_d = s\} = 2^d\binom{s+d-1}{d-1};
\end{align}
in (ii), we use the inequality $\binom{n}{k}\leq (\frac{en}{k})^k$; in (iii), we use $\td(x, 0_d) \leq \pi\sqrt{d}$.
We also have
\begin{align*}
\sum_{k\in\mb Z^d: \|k\| < 2\sqrt{d}} e^{-\frac{\|x-2\pi k\|^2}{2\varepsilon}}
&\leq e^{-\frac{\td(x, 0_d)^2}{2\varepsilon}}\cdot \#\{k\in\mb Z^d: \|k\| < 2\sqrt{d}\}\\
&\stackrel{\ri}{\leq} e^{-\frac{\td(x, 0_d)^2}{2\varepsilon}}\cdot \Big[1 + 2^d\sum_{s=1}^{4d}\Big(\frac{e(s+d-1)}{d-1}\Big)^{d-1}\Big]\\
&= e^{-\frac{\td(x, 0_d)^2}{2\varepsilon}}\cdot\Big[1 + 2^d e^{d-1}\sum_{s=1}^{4d}e^s\Big]\\
&< 2^de^{5d}e^{-\frac{\td(x, 0_d)^2}{2\varepsilon}}
\end{align*}
Here, in (i) we use~\eqref{eqn: count} again. Therefore, we have
\begin{align*}
\frac{1}{(2\pi\varepsilon)^\frac{d}{2}} e^{-\frac{\td(x, 0_d)^2}{2\varepsilon}}
\leq \m K_\varepsilon(x)
\leq 2^{d+1}e^{5d}\cdot \frac{1}{(2\pi\varepsilon)^\frac{d}{2}} e^{-\frac{\td(x, 0_d)^2}{2\varepsilon}}
\end{align*}
for all $\varepsilon < \frac{\pi^2}{2}$.
\end{proof}



\section{Controlling Holder Norm}
In this appendix, we aim to provide more details on the derivation of~\eqref{eqn: xxx} in the proof of Theorem~\ref{thm: existence}, where we have defined
\begin{align*}
A_{k-1} &\coloneqq \sum_{l=0}^{k-1} \frac{\Delta^{k-l}(u_l\rho_0)}{2^{k-l}(k-l)!},\\
B_{k-1} &\coloneqq \sum_{l=0}^{k-1}\sum_{i=0}^l \frac{\Delta^{k-l}(v_i\rho_{l-i})}{2^{k-l}(k-l)!} + \sum_{i=0}^{k-1}v_i\rho_{k-i},\\
C_{k-1} &\coloneqq -u_0\sum_{i=1}^{k-1}u_{k-i}u_i^\dagger,\\
D_{k-1} &\coloneqq -v_0\sum_{i=1}^{k-1}v_{k-i}v_i^\dagger.
\end{align*}
We also have (see Equation~\eqref{eqn: bridge})
\begin{align*}
v_0^\dagger v_k + u_0^\dagger u_k = u_0^\dagger C_{k-1} - u_0B_{k-1} = v_0^\dagger D_{k-1} - v_0A_{k-1} \eqqcolon S_{k-1},
\end{align*}
where $S_{k-1}$ only relates to variables with indices smaller than $k$. Note that $C_0 = D_0 = 0$, so we have
\begin{align*}
S_0 = -v_0A_0 = -u_0B_0.
\end{align*}
The goal is to reformulate~\eqref{eqn: uB=vA}, i.e.
\begin{align*}
\sum_{i=0}^k u_{k-i}B_i = \sum_{i=0}^k v_{k-i}A_i,
\end{align*}
as a equation of $u_k$.
\subsection{Derivation of Equation~\eqref{eqn: xxx}}\label{app: calculation}
First, note that the above equation is exactly
\begin{align}\label{eqn: eqn1}
\big(v_kA_0 + v_0A_k\big) - \big(u_kB_0 + u_0B_k\big) = \sum_{i=1}^{k-1} u_{k-i}B_i - \sum_{i=1}^{k-1} v_{k-i}A_i,
\end{align}
where the right-hand side does not depend on $u_k, v_k, u_k^\dagger$, or $v_k^\dagger$. Now, we only need to calculate the left-hand side.

Let us consider the first term in~\eqref{eqn: eqn1}. Note that
\begin{align*}
v_kA_0 + v_0A_k 
&= v_kA_0 + v_0\sum_{l=0}^k \frac{\Delta^{k-l+1}(u_l\rho_0)}{2^{k-l+1}(k-l+1)!}
= \Big[v_k A_0 + \frac{v_0\Delta(u_k\rho_0)}{2}\Big] + v_0\sum_{l=0}^{k-1}\frac{\Delta^{k-l+1}(u_l\rho_0)}{2^{k-l+1}(k-l+1)!}.
\end{align*}
Also, we have
\begin{align*}
A_0 v_k &= A_0 v_0 \cdot v_0^\dagger v_k
= A_0v_0\big(S_{k-1} - u_0^\dagger u_k\big)
= S_0\big(u_0^\dagger u_k - S_{k-1}\big),
\end{align*}
and
\begin{align*}
\frac{v_0\Delta(u_k\rho_0)}{2} &= \frac{v_0}{2}\Delta[(u_ku_0^\dagger)v_0^\dagger]
= \frac{v_0}{2}\big[v_0^\dagger\Delta(u_ku_0^\dagger) + 2\big\langle\nabla v_0^\dagger, \nabla(u_k u_0^\dagger)\big\rangle + u_ku_0^\dagger \Delta v_0^\dagger\big]\\
&= \frac{1}{2}\Delta(u_ku_0^\dagger) + \big\langle\nabla\log v_0^\dagger, \nabla(u_k u_0^\dagger)\big\rangle + \frac{v_0\Delta v_0^\dagger}{2}\cdot u_ku_0^\dagger.
\end{align*}
Therefore, we have
\begin{align*}
v_kA_0 + v_0A_k
&= \frac{1}{2}\Delta(u_ku_0^\dagger) + \big\langle\nabla\log v_0^\dagger, \nabla(u_ku_0^\dagger)\big\rangle + \Big[\frac{v_0\Delta v_0^\dagger}{2} + S_0\Big]u_ku_0^\dagger\\
&\qquad\qquad\qquad\qquad\qquad\qquad\quad + \bigg[v_0\sum_{l=0}^{k-1}\frac{\Delta^{k-l+1}(u_l\rho_0)}{2^{k-l+1}(k-l+1)!} - S_0S_{k-1}\bigg]
\end{align*}

Now, let us calculate the second term in~\eqref{eqn: eqn1}. Note that we have
\begin{align*}
&u_0B_k + u_kB_0
= u_0\bigg[\sum_{l=0}^k\sum_{i=0}^l \frac{\Delta^{k+1-l}(v_i\rho_{l-i})}{2^{k+1-l}(k+1-l)!} + \sum_{i=0}^k v_i\rho_{k-i+1}\bigg] + u_kB_0\\
&=\Big[u_kB_0 + \frac{u_0\Delta(v_k\rho_0)}{2} + u_0v_k\rho_1\Big] + u_0\bigg[\sum_{l=0}^{k-1}\sum_{i=0}^l \frac{\Delta^{k+1-l}(v_0\rho_{l-i})}{2^{k+1-l}(k+1-l)!} + \sum_{i=0}^{k-1}\frac{\Delta(v_i\rho_{k-i})}{2} + \sum_{i=0}^{k-1}v_i\rho_{k-i+1}\bigg].
\end{align*}
Only the first term above involves $u_k$ and $v_k$. Note that
\begin{align*}
\frac{u_0}{2}\Delta(v_k\rho_0) 
&= \frac{u_0}{2}\Delta(v_kv_0^\dagger u_0^\dagger)
= \frac{u_0}{2}\Delta\big[u_0^\dagger(S_{k-1} - u_0^\dagger u_k)\big]\\
&= -\frac{u_0}{2}\Delta[u_0^\dagger(u_0^\dagger u_k)] + \frac{u_0}{2}\Delta(u_0^\dagger S_{k-1})\\
&=-\frac{u_0}{2}\big[u_0^\dagger\Delta(u_0^\dagger u_k) + 2\big\langle \nabla u_0^\dagger, \nabla(u_0^\dagger u_k)\big\rangle + u_0^\dagger u_k\Delta u_0^\dagger\big] + \frac{u_0}{2}\Delta(u_0^\dagger S_{k-1})\\
&=-\frac{1}{2}\Delta(u_0^\dagger u_k) - \big\langle\nabla\log u_0^\dagger, \nabla(u_0^\dagger u_k)\big\rangle - \frac{u_0\Delta u_0^\dagger}{2} u_0^\dagger u_k + \frac{u_0}{2}\Delta(u_0^\dagger S_{k-1}),
\end{align*}
and we also have
\begin{align*}
u_0v_k\rho_1 
&= u_0\rho_1v_0 v_0^\dagger v_k = u_0\rho_1v_0(S_{k-1} - u_0^\dagger u_k)
= -u_0v_0\rho_1 u_0^\dagger u_k + u_0v_0\rho_1 S_{k-1}.
\end{align*}
Therefore, we have
\begin{align*}
u_0B_k + u_kB_0
&= -\frac{1}{2}\Delta(u_0^\dagger u_k) - \big\langle\nabla\log u_0^\dagger, \nabla(u_0^\dagger u_k)\big\rangle - \Big[u_0v_0\rho_1 + \frac{u_0\Delta u_0^\dagger}{2} - B_0u_0\Big] u_0^\dagger u_k\\
&\qquad + u_0\bigg[\frac{\Delta(u_0^\dagger S_{k-1})}{2} + v_0\rho_1S_{k-1} + \sum_{l=0}^{k-1}\sum_{i=0}^l \frac{\Delta^{k+1-l}(v_0\rho_{l-i})}{2^{k+1-l}(k+1-l)!} + \sum_{i=0}^{k-1}\frac{\Delta(v_i\rho_{k-i})}{2} + \sum_{i=0}^{k-1}v_i\rho_{k-i+1}\bigg].
\end{align*}

Using the above piece, we get
\begin{align*}
&\quad\,\big(v_kA_0 + v_0A_k\big) - \big(u_kB_0 + u_0B_k\big)\\
&= \Delta(u_ku_0^\dagger) + \big\langle\nabla\log \rho_0,\nabla(u_0^\dagger u_k)\big\rangle + \Big[\frac{v_0\Delta v_0^\dagger}{2} + S_0 + u_0v_0\rho_1 + \frac{u_0\Delta u_0^\dagger}{2} - B_0u_0\Big]u_0^\dagger u_k\\
&\qquad - u_0\bigg[\frac{\Delta(u_0^\dagger S_{k-1})}{2} + v_0\rho_1S_{k-1} + \sum_{l=0}^{k-1}\sum_{i=0}^l \frac{\Delta^{k+1-l}(v_0\rho_{l-i})}{2^{k+1-l}(k+1-l)!} + \sum_{i=0}^{k-1}\frac{\Delta(v_i\rho_{k-i})}{2} + \sum_{i=0}^{k-1}v_i\rho_{k-i+1}\bigg]\\
&\qquad + \bigg[v_0\sum_{l=0}^{k-1}\frac{\Delta^{k-l+1}(u_l\rho_0)}{2^{k-l+1}(k-l+1)!} - S_0S_{k-1}\bigg].
\end{align*}
Note that we have
\begin{align*}
\frac{v_0\Delta v_0^\dagger}{2} + S_0 + u_0v_0\rho_1 + \frac{u_0\Delta u_0^\dagger}{2} - B_0u_0
&= \frac{u_0v_0}{2}\big[v_0^\dagger \Delta u_0^\dagger - u_0^\dagger \Delta v_0^\dagger + 2\rho_1\big] + 2S_0 + v_0\Delta v_0^\dagger\\
&= 2S_0 + v_0\Delta v_0^\dagger
= v_0\Delta v_0^\dagger - 2v_0A_0\\
&= v_0\Delta v_0^\dagger - v_0\Delta (u_0\rho_0)\\
&=0,
\end{align*}
and also
\begin{align*}
\Delta(u_ku_0^\dagger) + \big\langle\nabla\log \rho_0,\nabla(u_0^\dagger u_k)\big\rangle
= \frac{1}{\rho_0}\nabla\cdot\big(\rho_0(u_0^\dagger u_k)\big).
\end{align*}
Therefore, we have
\begin{align*}
&\quad\,\big(v_kA_0 + v_0A_k\big) - \big(u_kB_0 + u_0B_k\big)\\
&= \frac{1}{\rho_0}\nabla\cdot\big(\rho_0\nabla(u_0^\dagger u_k)\big) + \bigg[v_0\sum_{l=0}^{k-1}\frac{\Delta^{k-l+1}(u_l\rho_0)}{2^{k-l+1}(k-l+1)!} - S_0S_{k-1}\bigg]\\
&\qquad - u_0\bigg[\frac{\Delta(u_0^\dagger S_{k-1})}{2} + v_0\rho_1S_{k-1} + \sum_{l=0}^{k-1}\sum_{i=0}^l \frac{\Delta^{k+1-l}(v_0\rho_{l-i})}{2^{k+1-l}(k+1-l)!} + \sum_{i=0}^{k-1}\frac{\Delta(v_i\rho_{k-i})}{2} + \sum_{i=0}^{k-1}v_i\rho_{k-i+1}\bigg].
\end{align*}
Combining with~\eqref{eqn: eqn1} yields
\begin{align*}
\nabla\cdot\big(\rho_0\nabla(u_0^\dagger u_k)\big)
&= \rho_0\bigg[\sum_{i=1}^{k-1} u_{k-i}B_i - \sum_{i=1}^{k-1} v_{k-i}A_i\bigg] - \rho_0\bigg[v_0\sum_{l=0}^{k-1}\frac{\Delta^{k-l+1}(u_l\rho_0)}{2^{k-l+1}(k-l+1)!} - S_0S_{k-1}\bigg]\\
&\qquad + v_0^\dagger\bigg[\frac{\Delta(u_0^\dagger S_{k-1})}{2} + v_0\rho_1S_{k-1} + \sum_{l=0}^{k-1}\sum_{i=0}^l \frac{\Delta^{k+1-l}(v_0\rho_{l-i})}{2^{k+1-l}(k+1-l)!} + \sum_{i=0}^{k-1}\frac{\Delta(v_i\rho_{k-i})}{2} + \sum_{i=0}^{k-1}v_i\rho_{k-i+1}\bigg]
\end{align*}

\subsection{Controlling Holder norm of right-hand side}\label{app: Holder norm}

\noindent\underline{Step 1: Holder norm of $A_i$, $B_i$ and $S_{k-1}$.}
By the definitions of $A_i$, we have
\begin{align*}
\|A_i\|_{\m C^{s,\alpha}}
&\leq \sum_{l=0}^{i}\frac{\|\Delta^{i-l+1}(u_l\rho_0)\|_{\m C^{s,\alpha}}}{2^{i-l+1}(i-l+1)!}
\leq \sum_{l=0}^i \frac{\|u_l\rho_0\|_{\m C^{s+2i-2l+2,\alpha}}}{2^{i-l+1}(i-l+1)!}
\lesssim_{d,s,\alpha,i} \|\rho_0\|_{\m C^{s+2i+2,\alpha}}\sum_{l=0}^i \frac{\|u_l\|_{\m C^{s+2i-2l+2,\alpha}}}{2^{i-l+1}(i-l+1)!}.
\end{align*}
Similarly, for $B_i$, we have
\begin{align*}
\|B_i\|_{\m C^{s,\alpha}}
&\leq \sum_{l=0}^{i}\sum_{j=0}^l\frac{\|\Delta^{i-l+1}(v_j\rho_{l-j})\|_{\m C^{s,\alpha}}}{2^{i-l+1}(i-l+1)!} + \sum_{j=0}^i \|v_j\rho_{i-j+1}\|_{\m C^{s,\alpha}}\\
&\lesssim_{d,s,\alpha,i}\sum_{l=0}^i \sum_{j=0}^l \frac{\|v_j\|_{\m C^{s+2i-2l+2,\alpha}}\|\rho_{l-j}\|_{\m C^{s+2i-2l+2,\alpha}}}{2^{i-l+1}(i-l+1)!} + \sum_{j=0}^i \|v_j\|_{\m C^{s,\alpha}}\|\rho_{i-j+1}\|_{\m C^{s,\alpha}}\\
&\leq \sum_{j=0}^i \|v_j\|_{\m C^{s+2i-2j+2,\alpha}}\sum_{l=j}^i \frac{\|\rho_{l-j}\|_{\m C^{s+2i-2l+2,\alpha}}}{2^{i-l+1}(i-l+1)!} + \sum_{j=0}^i\|v_j\|_{\m C^{s,\alpha}}\|\rho_{i-j+1}\|_{\m C^{s,\alpha}}\\
&\lesssim_{d,s,\alpha,i} \sum_{j=0}^i\|v_j\|_{\m C^{s+2i-2j+2,\alpha}}\sum_{l=0}^i\|\rho_l\|_{\m C^{s+2i-2l+2,\alpha}} + \sum_{j=0}^i\|v_j\|_{\m C^{s,\alpha}} \|\rho_{i-j+1}\|_{\m C^{s,\alpha}}.
\end{align*}
For $S_{k-1}$, we have
\begin{align*}
\|S_{k-1}\|_{\m C^{s,\alpha}}
&= \|v_0A_{k-1} - v_0^\dagger D_{k-1}\|_{\m C^{s,\alpha}}
= \bigg\|v_0 A_{k-1} + \sum_{i=1}^{k-1}v_{k-i}v_i^\dagger\bigg\|\\
&\lesssim_{d,s,\alpha}\|v_0\|_{\m C^{s,\alpha}}\|A_{k-1}\|_{\m C^{s,\alpha}} + \sum_{i=1}^{k-1}\|v_{k-i}\|_{\m C^{s,\alpha}}\|v_i^\dagger\|_{\m C^{s,\alpha}}.
\end{align*}

\vspace{0.5em}
\noindent\underline{Step 2: control the Holder norm of each part.}

\underline{Term 1.} Note that
\begin{align*}
\|\rho_0 u_{k-i}B_i\|_{\m C^{s,\alpha}}
\lesssim_{d,s,\alpha} \|\rho_0\|_{\m C^{s,\alpha}}\|u_{k-i}\|_{\m C^{s,\alpha}}\|B_i\|_{\m C^{s,\alpha}}.
\end{align*}

\underline{Term 2.} Note that
\begin{align*}
\|\rho_0 v_{k-i}A_i\|_{\m C^{s,\alpha}}
\lesssim_{d,s,\alpha} \|\rho_0\|_{\m C^{s,\alpha}}\|v_{k-i}\|_{\m C^{s,\alpha}}\|A_i\|_{\m C^{s,\alpha}}.
\end{align*}

\underline{Term 3.} We have
\begin{align*}
&\bigg\|\rho_0v_0\sum_{l=0}^{k-1}\frac{\Delta^{k-l+1}(u_l\rho_0)}{2^{k-l+1}(k-l+1)!}\bigg\|_{\m C^{s,\alpha}}
\lesssim_{d,s,\alpha} \|u_0^\dagger\|_{\m C^{s,\alpha}}\sum_{l=0}^{k-1}\frac{\|\Delta^{k-l+1}(u_l\rho_0)\|_{\m C^{s,\alpha}}}{2^{k-l+1}(k-l+1)!}\\
&\lesssim_{d,s,\alpha}\|u_0^\dagger\|_{\m C^{s,\alpha}}\sum_{l=0}^{k-1}\frac{\|u_l\rho_0\|_{\m C^{s+2k-2l+2,\alpha}}}{2^{k-l+1}(k-l+1)!}
\lesssim_{d,s,\alpha,k} \|u_0^\dagger\|_{\m C^{s,\alpha}}\|\rho_0\|_{\m C^{s+2k+2,\alpha}}\sum_{l=0}^{k-1}\frac{\|u_l\|_{\m C^{s+2k-2l+2,\alpha}}}{2^{k-l+1}(k-l+1)!}.
\end{align*}

\underline{Term 4.}
\begin{align*}
\|\rho_0S_0S_{k-1} + \rho_1S_{k-1}\|_{\m C^{s,\alpha}}
&\lesssim_{d,s,\alpha}\|\rho_0S_0 + \rho_1\|_{\m C^{s,\alpha}}\|S_{k-1}\|_{\m C^{s,\alpha}}
= \Big\|\rho_1 - \frac{u_0^\dagger\Delta v_0^\dagger}{2}\Big\|_{\m C^{s,\alpha}}\|S_{k-1}\|_{\m C^{s,\alpha}}\\
&\lesssim_{d,s,\alpha}\big[\|\rho_1\|_{\m C^{s,\alpha}} + \|u_0^\dagger\|_{\m C^{s,\alpha}}\|v_0^\dagger\|_{\m C^{s+2,\alpha}}\big]\|S_{k-1}\|_{\m C^{s,\alpha}}.
\end{align*}

\underline{Term 5.}
\begin{align*}
\Big\|\frac{\Delta(u_0^\dagger S_{k-1})}{2}\Big\|_{\m C^{s,\alpha}}
\leq \frac{1}{2}\|u_0^\dagger S_{k-1}\|_{\m C^{s+2,\alpha}}
\lesssim_{d,s,\alpha} \|u_0^\dagger\|_{\m C^{s+2,\alpha}}\|S_{k-1}\|_{\m C^{s+2,\alpha}}
\end{align*}

\underline{Term 6.}
\begin{align*}
&\bigg\|\sum_{l=0}^{k-1}\sum_{i=0}^l \frac{\Delta^{k+1-l}(v_0\rho_{l-i})}{2^{k+1-l}(k+1-l)!}\bigg\|_{\m C^{s,\alpha}}
\leq \sum_{l=0}^{k-1}\sum_{i=0}^l \frac{\|v_0\rho_{l-i}\|_{\m C^{s+2k-2l+2,\alpha}}}{2^{k+1-l}(k+1-l)!}\\
&\lesssim_{d,s,\alpha,k}\sum_{l=0}^{k-1}\sum_{i=0}^l \frac{\|v_0\|_{\m C^{s+2k-2l+2,\alpha}}\|\rho_{l-i}\|_{\m C^{s+2k-2l+2,\alpha}}}{2^{k+1-l}(k+1-l)!}
\lesssim_{d,s,\alpha,k} \|v_0\|_{\m C^{s+2k+2,\alpha}}\sum_{l=0}^{k-1}\sum_{i=0}^l \frac{\|\rho_{l-i}\|_{\m C^{s+2k-2l+2,\alpha}}}{2^{k+1-l}(k+1-l)!}\\
&\lesssim_{d,s,\alpha,k} \|v_0\|_{\m C^{s+2k+2,\alpha}}\sum_{l=0}^{k-1}\frac{(k-l)\|\rho_l\|_{\m C^{s+2k-2l+2,\alpha}}}{8}.
\end{align*}

\underline{Term 7.}
\begin{align*}
\bigg\|\sum_{i=0}^{k-1}\frac{\Delta(v_i\rho_{k-i})}{2}\bigg\|_{\m C^{s,\alpha}}
\lesssim_{d,s,\alpha} \frac{1}{2}\sum_{i=0}^{k-1}\|v_i\rho_{k-i}\|_{\m C^{s+2,\alpha}}
\lesssim_{d,s,\alpha} \frac{1}{2}\sum_{i=0}^{k-1}\|v_i\|_{\m C^{s+2,\alpha}}\|\rho_{k-i}\|_{\m C^{s+2,\alpha}}.
\end{align*}

\underline{Term 8.}
\begin{align*}
\bigg\|\sum_{i=0}^{k-1}v_i\rho_{k-i+1}\bigg\|_{\m C^{s,\alpha}}
\lesssim_{d,s,\alpha}\sum_{i=0}^{k-1}\|v_i\|_{\m C^{s,\alpha}}\|\rho_{k-i+1}\|_{\m C^{s,\alpha}}
\end{align*}

\noindent\underline{Step 3: conclusion.}
By checking every items shown above, we conclude that the upper bound of $\m C^{s,\alpha}$-norm only depends on
\begin{align*}
&d,s,k\in\mb Z_+, \alpha\in(0, 1)\\
& \|u_0\|_{\m C^{s+2k+2,\alpha}}, \|u_1\|_{\m C^{s+2k,\alpha}} \cdots, \|u_{k-1}\|_{\m C^{s+4,\alpha}}\\
& \|v_0\|_{\m C^{s+2k,\alpha}}, \|v_1\|_{\m C^{s+2k-2,\alpha}}, \cdots, \|v_{k-1}\|_{\m C^{s+2,\alpha}}\\
& \|v_1^\dagger\|_{\m C^{s+2,\alpha}}, \cdots, \|v_{k-1}^\dagger\|_{\m C^{s+2,\alpha}}\\
& \|\rho_0\|_{\m C^{s+2k+2,\alpha}}, \|\rho_1\|_{\m C^{s+2k,\alpha}}, \cdots, \|\rho_{k+1}\|_{\m C^{s,\alpha}}\\
& \|u_0^\dagger\|_{\m C^{s+2,\alpha}}, \|v_0^\dagger\|_{\m C^{s+2,\alpha}}, \|v_0\|_{\m C^{s+2k+2,\alpha}}.
\end{align*}

\end{document}